\documentclass[11pt]{amsart}

\usepackage{enumerate, amsmath, amsthm, amsfonts, amssymb, xy,  mathrsfs, graphicx}
\usepackage[usenames, dvipsnames]{color}
\usepackage[margin=1in]{geometry} 
\usepackage[boxsize=1em]{ytableau}
\usepackage{stmaryrd}
\usepackage[bookmarks, colorlinks=true, linkcolor=blue, citecolor=blue, urlcolor=blue]{hyperref}

\input xy
\xyoption{all}

\newcommand{\cA}{\mathcal{A}}

\newcommand{\bC}{\mathbf{C}}

\newcommand{\bD}{\mathbf{D}}

\newcommand{\rD}{\mathrm{D}}

\newcommand{\bG}{\mathbf{G}}

\newcommand{\bM}{\mathbf{M}}

\newcommand{\bO}{\mathbf{O}}

\newcommand{\bP}{\mathbf{P}}

\newcommand{\bQ}{\mathbf{Q}}

\newcommand{\cR}{\mathcal{R}}

\newcommand{\bS}{\mathbf{S}}
\newcommand{\cS}{\mathcal{S}}

\newcommand{\cU}{\mathcal{U}}

\newcommand{\rU}{\mathrm{U}}

\newcommand{\cV}{\mathcal{V}}

\newcommand{\bZ}{\mathbf{Z}}

\newcommand{\fa}{\mathfrak{a}}

\newcommand{\fp}{\mathfrak{p}}

\newcommand{\fq}{\mathfrak{q}}

\newcommand{\fs}{\mathfrak{s}}

\newcommand{\bw}{\mathbf{w}}

\renewcommand{\phi}{\varphi}
\renewcommand{\emptyset}{\varnothing}

\renewcommand{\tilde}[1]{\widetilde{#1}}
\newcommand{\ol}[1]{\overline{#1}}
\newcommand{\ul}[1]{\underline{#1}}

\newcommand{\arxiv}[1]{\href{http://arxiv.org/abs/#1}{{\tt arXiv:#1}}}

\makeatletter
\def\Ddots{\mathinner{\mkern1mu\raise\p@
\vbox{\kern7\p@\hbox{.}}\mkern2mu
\raise4\p@\hbox{.}\mkern2mu\raise7\p@\hbox{.}\mkern1mu}}
\makeatother

\DeclareMathOperator{\rad}{rad} 
\DeclareMathOperator{\srad}{srad}

\DeclareMathOperator{\trace}{Tr}
\DeclareMathOperator{\diag}{diag}
\DeclareMathOperator{\rank}{rank}

\DeclareMathOperator{\Sym}{Sym}

\DeclareMathOperator{\Aut}{Aut}

\DeclareMathOperator{\Tor}{Tor}

\DeclareMathOperator{\Spec}{Spec}

\DeclareMathOperator{\sgn}{sgn}

\DeclareMathOperator{\len}{len}
\DeclareMathOperator{\Mod}{Mod}

\newcommand{\GL}{\mathbf{GL}}
\newcommand{\SL}{\mathbf{SL}}
\newcommand{\Sp}{\mathbf{Sp}}

\newcommand{\Gr}{\mathbf{Gr}}

\newcommand{\fgl}{\mathfrak{gl}}

\newcommand{\noarticle}{%
  \removelastskip%
  \vskip.6\baselineskip%
  \par}

\newcounter{article}
\newcommand{\article}[1][-]{%
  \setcounter{equation}{0}%
  \removelastskip%
  \vskip.6\baselineskip%
  \refstepcounter{article}%
  \noindent%
  {\bf (\thearticle)} %
  \ifx-#1{}\else{\it #1.}\fi \phantomsection}

\def\ifempty#1{\if&#1&}

\def\Thmtpl#1#2#3#4{%
  \removelastskip%
  \vskip.6\baselineskip%
  \refstepcounter{article}%
  \noindent%
  {\bf (\thearticle)} %
  \bgroup{}#2{}#1%
  \ifempty{#4}\else{ (#4)}\fi%
  .\egroup%
  #3%
  \ifempty{#4}\ \fi
\phantomsection }
\def\endthm{\par\vskip.4\baselineskip}

\def\thmtpl#1#2#3#4{%
  \removelastskip%
  \par\vskip.4\baselineskip%
  \addtocounter{article}{-1}
  \refstepcounter{article}
  \noindent%
  \bgroup#2#1%
  \ifempty{#4}\else{ (#4)}\fi%
  .\egroup%
  #3%
  \ifempty{#4}\ \fi
\phantomsection
}

\def\newtheoremenvironment#1#2#3#4#5{%
  \newenvironment{#1}[1][]{\Thmtpl{#3}{#4}{#5}{##1}}{\endthm}%
  \newenvironment{#2}[1][]{\thmtpl{#3}{#4}{#5}{##1}}{\endthm}%
  }

\newtheoremenvironment{Theorem}{theorem}{Theorem}{\bf}{\it}
\newtheoremenvironment{Conjecture}{conjecture}{Conjecture}{\bf}{\it}
\newtheoremenvironment{Proposition}{proposition}{Proposition}{\bf}{\it}
\newtheoremenvironment{Lemma}{lemma}{Lemma}{\bf}{\it}
\newtheoremenvironment{Corollary}{corollary}{Corollary}{\bf}{\it}
\newtheoremenvironment{Example}{example}{Example}{\it}{}
\newtheoremenvironment{Definition}{definition}{Definition}{\it}{}
\newtheoremenvironment{Remark}{remark}{Remark}{\it}{}
\newtheoremenvironment{Question}{question}{Question}{\it}{}

\renewenvironment{proof}[1][Proof]{
  \removelastskip%
  \par\vskip.4\baselineskip\noindent%
  \topsep=0pt%
  \partopsep=0pt%
  \labelsep=0pt%
  \begin{trivlist}%
  \item[\emph{#1.} ]%
  \pushQED{\qed} 
}{
\popQED%
  \end{trivlist}%
  \vskip.4\baselineskip%
}

\newcommand{\xsection}[1]{%
  \section{#1}
  \setcounter{article}{0}
  \renewcommand{\thearticle}{\thesection.\arabic{article}}%
  }

\renewcommand{\subsection}[1]{%
  \refstepcounter{subsection}%
  \addcontentsline{toc}{subsection}{\hskip 4.2ex\thesubsection{}. #1}%
  \removelastskip%
  \vskip.6\baselineskip%
  \noindent%
  \thesubsection{}. {\bf #1}%
  \leavevmode%
  \setcounter{article}{0}
  \renewcommand{\thearticle}{\thesubsection.\arabic{article}}%
  \nopagebreak%
  }

\let\mf\mathfrak
\let\mc\mathcal
\let\mb\mathbf
\let\mr\mathrm

\let\wt\widetilde
\let\ol\overline
\let\ul\underline

\let\lbb\llbracket
\let\rbb\rrbracket

\newcommand{\pref}[1]{{\bf (}\ref{#1}{\bf )}}
\newcommand{\lw}{{\textstyle \bigwedge}}
\newcommand{\Or}{\mathrm{Or}}

\newcommand{\op}{\mathrm{op}}
\newcommand{\gfin}{\mathrm{gf}}
\newcommand{\pol}{\mathrm{pol}}
\newcommand{\fin}{\mathrm{f}}

\renewcommand{\Vec}{\mathrm{Vec}}
\renewcommand{\fs}{\mathrm{(fs)}}
\newcommand{\uotimes}{\, \ul{\otimes} \,}

\newcommand{\db}{(\mathrm{db})}
\newcommand{\ub}{(\mathrm{ub})}
\newcommand{\tors}{\mathrm{tors}}
\renewcommand{\bw}[1]{{\textstyle \bigwedge}^{#1}}

\DeclareMathOperator{\Ind}{Ind}
\DeclareMathOperator{\Hom}{Hom}
\DeclareMathOperator{\Rep}{Rep}
\DeclareMathOperator{\Fun}{Fun}

\newcommand{\Br}{\mathrm{Br}}
\newcommand{\HS}{\mathrm{HS}}
\newcommand{\VS}{\mathrm{VS}}

\title{Introduction to twisted commutative algebras}
\author{Steven V Sam}
\address{Department of Mathematics, University of California, Berkeley, CA}
\email{svs@math.berkeley.edu}
\author{Andrew Snowden}
\address{Department of Mathematics, MIT, Cambridge, MA}
\email{asnowden@math.mit.edu}
\date{September 21, 2012}
\thanks{S.~Sam was supported by an NDSEG fellowship and a Miller research fellowship. A.~Snowden was partially supported by NSF fellowship DMS-0902661.}


\begin{document}

\maketitle

\begin{abstract}
This article is an expository account of the theory of twisted commutative algebras, which simply put, can be thought of as a theory for handling commutative algebras with large groups of linear symmetries.  Examples include the coordinate rings of determinantal varieties, Segre--Veronese embeddings, and Grassmannians. The article is meant to serve as a gentle introduction to the papers of the two authors on the subject, and also to point out some literature in which these algebras appear. The first part reviews the representation theory of the symmetric groups and general linear groups. The second part introduces a related category and develops its basic properties. The third part develops some basic properties of twisted commutative algebras from the perspective of classical commutative algebra and summarizes some of the results of the authors. We have tried to keep the prerequisites to this article at a minimum. The article is aimed at graduate students interested in commutative algebra, algebraic combinatorics, or representation theory, and the interactions between these subjects.
\end{abstract}

\setcounter{tocdepth}{2}
\tableofcontents

\section*{Introduction}
\setcounter{article}{0}
\renewcommand{\thearticle}{\arabic{article}}

\article
\label{i1}
There are many important examples in commutative algebra where the objects of interest (commutative rings, modules, free resolutions, etc.) form a sequence indexed by the positive integers and where the general linear group $\GL(n)$ acts on the $n$th object in the sequence.  Such examples often arise by applying a natural construction to a vector space of dimension $n$.  Some examples include the Grassmannian of $k$-planes in $n$-space (for $k$ fixed), the $k$th Veronese embedding of the projective space $\bP(\bC^n)$ (for $k$ fixed), the space symmetric $n \times n$ matrices of rank $\le k$ (for $k$-fixed) and the tangent and secant varieties to these examples.  In such situations, one is usually interested in understanding the behavior for all values of $n$ in a uniform manner.  The theory of twisted commutative algebras (tca's) offers a framework for doing exactly this.

\article
\label{i2}
There are several equivalent ways to define tca's.  We give two definitions now, and a third one later in the introduction.
\begin{itemize}
\item Definition 1:  A rule which associates to each vector space $V$ a commutative ring $A(V)$, and to each linear map of vector spaces $V \to V'$ a ring homomorphism $A(V) \to A(V')$.
\item Definition 2:  A commutative ring $A$ equipped with an action of the infinite general linear group $\GL(\infty)=\bigcup_{n \ge 1} \GL(n)$ by algebra homomorphisms.
\end{itemize}
In each definition there is a technical condition (polynomiality) which is required; this is discussed in the body of the paper.  The connection of Definition~1 to the discussion of \pref{i1} is clear.  For instance, letting $A(V)$ be the projective coordinate ring of the Grassmannian $\Gr(k, V)$ (for some fixed $k$) defines such a rule, and is thus an example of a tca.  The $k=1$ case is particularly simple:  in this case, $A(V)$ is just $\Sym(V)$, the symmetric algebra on $V$.  We denote this particular tca by $\Sym(\bC\langle 1 \rangle)$.  Getting to Definition~2 from Definition~1 is easy:  simply evaluate on the vector space $\bC^{\infty}=\bigcup_{n \ge 1} \bC^n$.  Thus, in Definition~2, the tca $\Sym(\bC\langle 1 \rangle)$ is $\Sym(\bC^{\infty})=\bC[x_1, x_2, \ldots]$.

\article
\label{i3}
There is a notion of ``module'' over a tca $A$.  There are two definitions, matching the two definitions for tca's:
\begin{itemize}
\item Definition 1:  A rule which associates to each vector space $V$ an $A(V)$-module $M(V)$, and to each linear map of vector spaces $V \to V'$ a map of $A(V)$-modules $M(V) \to M(V')$.
\item Definition 2: An $A$-module equipped with a compatible action of $\GL(\infty)$.
\end{itemize}
Again, there is a technical condition which is required.  Let us give a basic example of a module.  Let $A$ be the tca (in Definition~1) such that $A(V)=\Sym(\Sym^2(V))$ is the projective coordinate ring of $\bP(\Sym^2(V))$.  Inside of $\bP(\Sym^2(V))$ is the image of the second Veronese embedding of $\bP(V)$.  This is defined by a $\GL(V)$-stable ideal $I(V) \subset A(V)$, and thus $I$ is an $A$-module.  In Definition~2, we can identify $A$ with the polynomial ring $\bC[x_{ij}]$, with $1 \le i,j \le \infty$ and $x_{ij}=x_{ji}$, and $I$ is the ideal generated by $x_{ij}x_{k\ell}-x_{i\ell}x_{kj}$ for all $i$, $j$, $k$ and $\ell$.

\article
\label{i4}
From the point of view of commutative algebra, it is interesting to study projective resolutions of modules over tca's.  Let us now give an explicit example of a resolution.  Let $A$ be the tca $\Sym(\bC\langle 1 \rangle)$ and let $M$ be the $A$-module $\bC$, that is, $M(V)$ is the $A(V)$-module $\bC$ for all $V$.  Let $x_1, \dots, x_n$ be a basis for $V$. When $n=1$ the resolution of $M(V)$ over $A(V)$ is simply
\begin{displaymath}
0 \to A(V) \xrightarrow{\cdot x_1} A(V).
\end{displaymath}
The cases $n=2$ and $n=3$ become more interesting:
\[
0 \to A(V) \xrightarrow{\tiny \begin{pmatrix} -x_2 \\ x_1 \end{pmatrix}} A(V)^2 \xrightarrow{\tiny \begin{pmatrix} x_1 & x_2 \end{pmatrix}} A(V)
\]
\[
0 \to A(V) \xrightarrow{\tiny \begin{pmatrix} x_1 \\ x_2 \\ x_3 \end{pmatrix}} A(V)^3 \xrightarrow{\tiny \begin{pmatrix} 0 & x_3 & -x_2 \\ -x_3 & 0 & x_1 \\ x_2 & -x_1 & 0 \end{pmatrix}} A(V)^3 \xrightarrow{\tiny \begin{pmatrix} x_1 & x_2 & x_3 \end{pmatrix}} A(V).
\]
Writing the resolution explicitly for larger $n$ can be quite cumbersome. A basis-free way to write down the differentials makes use of the exterior powers $\lw^d{V}$, which we now explain. The space $\lw^d{V}$ spanned by symbols of the form $v_1 \wedge \cdots \wedge v_d$ for $v_i \in V$ subject to bilinearity and skew-symmetric relations. There is a natural map
\begin{align*}
\lw^d V &\to V \otimes \lw^{d-1} V\\
v_1 \wedge \cdots \wedge v_d &\mapsto \sum_{i=1}^d (-1)^i v_i \otimes v_1 \wedge \cdots \hat{v_i} \cdots \wedge v_d
\end{align*}
where $\hat{v_i}$ means that we have omitted $v_i$. One can check that it is well-defined and that it respects the natural action of $\GL(V)$ on the left and right hand sides, i.e., it is {\bf equivariant}. There is a natural way to extend this to a map 
\begin{align*}
\Sym(V) \otimes \lw^d V &\to \Sym(V) \otimes \lw^{d-1} V\\
p(v) \otimes v_1 \wedge \cdots \wedge v_d &\mapsto \sum_{i=1}^d (-1)^i v_ip(v) \otimes v_1 \wedge \cdots \hat{v_i} \cdots \wedge v_d.
\end{align*}
With suitable choices of bases for our matrices above, we can now rewrite the complexes for $n=1,2,3$ as 
\begin{align*}
0 \to A(V) \otimes V \to A(V),\\
0 \to A(V) \otimes \lw^2 V \to A(V) \otimes V \to A(V),\\
0 \to A(V) \otimes \lw^3 V \to A(V) \otimes \lw^2 V \to A(V) \otimes V \to A(V),
\end{align*}
where the differentials are the maps that we have just explained. Now note that $\lw^d V = 0$ if $d>n$.  It follows that we have a resolution of $\bC$ valid for any value of $n=\dim{V}$ as follows:
\[
\cdots \to A(V) \otimes \lw^d V \to A(V) \otimes \lw^{d-1} V \to \cdots \to A(V) \otimes \lw^2 V \to A(V) \otimes V \to A(V).
\]
This is the {\bf Koszul complex}.  In fact, it is a resolution of $\bC$ in the sense of tca's.  That is, if we let $F_i(V)=A(V) \otimes \lw^i{V}$ then $F_i$ is a module over the tca $A$, the differentials $F_i \to F_{i-1}$ appearing in the above complex are maps of $A$-modules and the complex $F_{\bullet}$ has homology only in degree 0, where it is $\bC$.

\article
There are two features of the above resolution that are worth pointing out:
\begin{itemize}
\item When the resolution $F_{\bullet}$ of the $A$-module $\bC$ is specialized to $V$ (i.e., when we evaluate on $V$), we obtain the resolution $F_{\bullet}(V)$ of the $A(V)$-module $\bC$.
\item Although $\bC$ admits a finite length resolution over $A(V)$ for each $V$, the resolution of $\bC$ over $A$ is infinite.
\end{itemize}
This behavior is typical, and holds in most circumstances.  The first point, though somewhat trivial, is very powerful, as it allows us to study the resolutions of all the modules $M(V)$ at once.  For instance, by studying the resolution of the module $I$ from \pref{i3}, one is actually studying all resolutions of second Veronese embeddings simultaneously.  The second point is the source of much difficulty in studying resolutions over tca's:  one can almost never see the full picture by specializing to finite dimensional vector spaces.

\article
The resolution of the residue field discussed in \pref{i4} is fairly trivial.  However, resolving other finite length modules over $\Sym(\bC\langle 1 \rangle)$ is not at all easy and has turned out to be very important:  the pure resolutions constructed in \cite{efw} are exactly of this kind.  Pure resolutions are important because their graded Betti tables are the extremal rays in the cone of all graded Betti tables, so they can be thought as fundamental building blocks. That paper did not use the language of tca's, but it very much used the same point of view in an essential manner.  We elaborate on this in \S\ref{ss:efw}.

\article
There is a notion for an $A$-module $M$ to be finitely generated.  To explain this, we take the view of Definition~2.  Given an element $m$ of $M$, there is a smallest submodule $\langle m \rangle \subset M$ containing $m$.  Here, we use the word ``submodule'' in the sense of tca's:  it is required to be stable under the action of $\GL(\infty)$.  Explicitly, $\langle m \rangle$ is obtained by first taking the $\GL(\infty)$-representation generated by $m$ and then the $A$-submodule of $M$ (in the usual sense) generated by this space.  The module $M$ is finitely generated if there are finitely many elements $m_1, \ldots, m_n$ such that $M=\sum_{i=1}^n \langle m_i \rangle$.  For example, the module $I$ constructed in \pref{i3} is finitely generated:  in fact, it is generated by any non-zero element of degree 2.  Note, however, that in the usual sense, i.e., without the $\GL(\infty)$ action, $I$ is not finitely generated!  That fact that very large objects can be regarded as ``finitely generated'' is one of the useful points of view that tca's offer.

\article
One of the first indications that the theory of tca's is interesting is that many tca's of interest are noetherian.  (If it were not for this, there would probably not be much to say about tca's!)  This can lead to interesting finiteness results ``for free'':  for instance, if $A$ is noetherian and $M$ is a finitely generated $A$-module then $\Tor^A_p(M, \bC)$ is a finite length representation of $\GL(\infty)$; this implies that there are finitely many ``universal'' $p$-syzygies which generate the module of $p$-syzygies of $M(V)$, for any $V$.  One of the most important open problems in the general theory of tca's is whether or not all finitely generated tca's are noetherian.  For instance, we have no idea if the tca $\Sym(\Sym^3(\bC^{\infty}))$ is noetherian or not.

\article
As promised, we now give a third equivalent definition of tca's:
\begin{itemize}
\item Definition 3:  An associative unital graded ring $A=\bigoplus_{n \ge 0} A_n$ equipped with an action of $S_n$ on $A_n$ which satisfies the following ``twisted'' version of the commutativity axiom:  for $x \in A_n$ and $y \in A_m$ we have $yx=\tau(xy)$, where $\tau \in S_{n+m}$ interchanges $1, \ldots, n$ and $n+1, \ldots, n+m$.
\end{itemize}
It is not at all obvious how this definition relates to the other two.  The link is through Schur--Weyl duality, which relates the representation theory of general linear and symmetric groups.

\article
A basic example of a tca from the point of view of Definition~3 is the tensor algebra on a vector space.  Precisely, let $U$ be a finite dimensional vector space and put $A_n=U^{\otimes n}$.  There is a multiplication map $A_n \otimes A_m \to A_{n+m}$ defined by concatenating tensors.  This gives $A=\bigoplus_{n \ge 0} A_n$ the structure of an associative unital algebra.  This algebra is highly non-commutative, but obviously satisfies the twisted commutativity axiom.  This point of view on the tensor algebra can be extremely useful, as it allows one to regard it as a commutative algebra (this perspective is used in \cite{gs}).  In Definition~2, the tca $A$ is given as $\Sym(U \otimes \bC^{\infty})$.

\article
To deal with the many points of view of tca's and their modules, we adopt the very convenient language of tensor categories.  A tensor category (for us) is an abelian category equipped with a (symmetric) tensor functor.  Each tensor category provides an entire world where all the basic definitions of commutative algebra (though not necessarily all of the results) hold.  The simplest tensor category is the category of vector spaces; algebras in this category are algebras in the usual sense, and so the theory of commutative algebra afforded by this category is the usual theory.  The category of graded vector spaces is a tensor category, and its algebras are graded algebras.  In this paper, we are primarily interested in the following three tensor categories:
\begin{itemize}
\item The category of polynomial functors of vector spaces.
\item The category of polynomial representations of $\GL(\infty)$.
\item The category of sequences $(A_n)_{n \ge 0}$ where $A_n$ is a representation of the symmetric group $S_n$.
\end{itemize}
The definitions 1--3 of tca's are actually just the definitions of algebras in the above three categories.  In fact, the above three categories are equivalent, and this is why the three definitions of tca's are equivalent.  However, the equivalence of the categories shows more:  it means that all the basic definitions one can make for tca's, such as ``module'' or ``ideal,'' can be canonically transferred between the three points of view.

\article
In many instances, we prefer not to think about which point of view we are taking towards tca's.  We therefore introduce an abstract tensor category, which we name $\cV$.  We regard the three categories mentioned above as ``models'' or ``incarnations'' of $\cV$, and move between them as convenient.  (There are a few other models we introduce as well.)  The different points of view can be extremely useful, as certain constructions or results are easy to see in one model but not another.  For instance, we can define a functor $\cV \to \cV$ by taking a sequence $(A_n)$ to the sequence $(A_n^*)$, where $A_n^*$ is the usual linear dual of $A_n$.  What does this operation correspond to in terms of representations of $\GL(\infty)$?  (Hint: it is not ``dual!'')

\article
We motivated tca's by observing that many examples in commutative algebra occur by applying a natural construction to a vector space.  However, there are many similar examples that come by applying a natural construction to several vector spaces.  Examples include determinantal varieties, Segre varieties and the secant and tangent varieties to these varieties.  The algebras and modules which arise are algebras are called {\bf multivariate tca's}, and can be viewed as algebras in tensor powers of $\cV$.  We do not discuss them much in the body of the paper, but they share many properties with tca's, so we will just list some pointers to the literature here:
\begin{itemize}
\item The minimal free resolutions of determinantal ideals were calculated by Lascoux \cite{lascoux}, and the description is given in terms of an analogue of Definition 1. The  calculations for symmetric and skew-symmetric matrices were obtained in \cite{jpw}. See also \cite[Chapter 6]{weyman} for these results. We point to \cite{ssw} for a generalization, which is discussed briefly in \S\ref{ss:infrank}.
\item The analogue of Definition 3 has been successfully used in the papers \cite{raicu} and \cite{oedingraicu} to calculate the ideal of definition of the secant and tangential varieties to arbitrary Segre--Veronese embeddings.
\item In a different direction, one can make precise the notion of letting the number of factors become infinite, i.e., working with infinite tensor products. This was used in \cite{draismakuttler} to show that for every $k$, there is a constant $d(k)$ so that the $k$th secant variety to a Segre embedding is defined (set-theoretically) by equations of degree at most $d(k)$ (independent of the number of factors in the Segre embedding).
\end{itemize}

\vskip.6\baselineskip\noindent
{\bf Outline.}\nopagebreak

\article
An analogy worth keeping in mind is that vector spaces are to commutative algebra what the representation theory of the general linear and symmetric groups are to twisted commutative algebra.  It is therefore essential that one have a basic understanding of this representation theory before studying tca's.  We give an overview of this theory in the first part of the paper.

\article
In the second part of the paper, we introduce and study the category $\cV$.  In \S\ref{sec:models} we introduce the various models for the category $\cV$ and discuss the equivalences between them.  In \S\ref{sec:cvprop} we discuss the extremely rich structure of $\cV$; it is much more than just a tensor category!   The final section of this part, \S\ref{s:transp}, is somewhat technical:  it recalls the exact definition of a symmetric tensor category, which is needed for precisely stating the categorical properties of the transpose operation.

\article
We finally get to tca's in the third part of the paper.  In \S\ref{sec:tca-gen} we discuss the general theory.  We begin by giving basic definitions (tca's, modules, ideals, finite generation, etc.).  We then give less basic definitions (nilradicals, prime ideals, etc.) and prove some basic results.  While we manage to prove some theorems and give some interesting examples, there are many basic questions which we do not answer.  For instance:  does a noetherian tca have finitely many minimal prime ideals?

\article
In \S\ref{sec:tca-bd} we discuss the class of bounded tca's.  These are much easier to deal with than general tca's and tend to be much better behaved.  For instance, the nilradical in a noetherian bounded tca is nilpotent, while this is not true in the unbounded case!  Fortunately, many tca's of interest are bounded.

\article
In \S\ref{sec:existingapps} we review four existing applications of the theory of tca's, either implicit or explicit:  (1) The construction of pure resolutions in \cite{efw}, mentioned above.  (2) The theory of FI-modules \cite{fimodules}.  In fact, FI-modules are simply modules over the tca $\Sym(\bC\langle 1 \rangle)$!  (3) The work of the second author on syzygies of Segre embeddings.  Here tca's are used to establish basic properties of more exotic algebraic structures, called $\Delta$-modules.  (4) Applications of tca's to certain problems in invariant theory.

\article
Finally, in \S\ref{sec:announce}, we announce some of our new results on tca's that have already appeared or are yet to appear.  We briefly mention a few of these results here:
\begin{itemize}
\item In \cite{symc1} we give a very thorough description of the category of $\Sym(\bC\langle 1 \rangle)$-modules.
\item In \cite{koszul}, we show that in many cases projective resolutions over tca's, while unbounded, have strong finiteness properties.
\item In \cite{reptheory}, we show that certain representation categories of infinite rank groups can be described as module categories over tca's.  For instance, $\Rep(\bO(\infty))$ is equivalent to the category of modules over the tca $\Sym(\Sym^2(\bC^{\infty}))$.  This point of view can be very useful as it allows one to transfer results and constructions from commutative algebra to representation theory.
\end{itemize}

\article
The first part of the article is entirely expository and the material is widely known.  The second part is mostly expository still, though the material is somewhat more obscure.  The third part of the article is semi-expository.  Some of the results and definitions have appeared in a few papers in the literature, while others are new.

\vskip.6\baselineskip\noindent
{\bf Purpose of this article.}
We had two main sources of motivation for writing this article.  First, as research articles are now appearing that use tca's, we thought it would be useful to have an account of the basic theory in the literature to serve as a reference.  And second, it seemed difficult to us to find a single source that covers all of the background material needed for the theory of tca's, so we have tried to collect most of it here. In the process of learning this material ourselves, we have found the references \cite{fulton, fultonharris, james, kraftprocesi, macdonald, stanley, weyman} useful. Our aim with this paper is not to give a complete self-contained account of the requisite background theory. Rather, we aim to give a working guide with pointers to the relevant literature as necessary.

\vskip.6\baselineskip\noindent
{\bf Reading plan.}
On a first reading, we suggest skipping several of the sections, and the suggested reading will depend on the background of the user:

\begin{itemize}
\item For those unfamiliar with representation theory, we have tried to give quick access to the relevant background in Part~\ref{part:reptheory}.
\item The definition of the category $\cV$ is essential to the rest of the theory, and for that, the reader should see \S\S\ref{sec:sequenceV}, \ref{sec:GLmodelV}, \ref{sec:schurV}. The proofs can be skipped without loss of continuity. The basic operations and structures on $\cV$ are described in \S\ref{sec:cvprop}, though only \S\ref{ss:cvbasic} is essential.
\item Twisted commutative algebras are the main object of study.  Their definition is given in \S\ref{ss:tca-defn}, some basic examples are given in \S \ref{ss:tca-ex} and their general properties are developed in the remainder of \S\ref{sec:tca-gen}.  In \S\ref{sec:tca-bd}, the important class of bounded tca's are studied.
\end{itemize}

The rest of the reading plan depends on the taste of the reader. We include some existing uses of twisted commutative algebras in \S\ref{sec:existingapps} and an announcement of our own new results in \S\ref{sec:announce}. 

\vskip.6\baselineskip\noindent
{\bf Conventions.}
Throughout, we work over the field of complex numbers $\bC$. However, everything works exactly the same over any field of characteristic 0, as all constructions are defined over the field of rational numbers $\bQ$.

\part{Background on symmetric and general linear groups} \label{part:reptheory}

\xsection{Partitions and Young diagrams}

\article[Partitions]
A {\bf partition} $\lambda$ is a weakly decreasing sequence $(\lambda_1, \lambda_2, \ldots)$ of non-negative integers.  We regard partitions as infinite sequences which are eventually 0, although we will often not write the trailing zeros.  We will often need to use partitions where one or several numbers are repeated many times.  We use a superscript to denote such repetition.  For example, $(5,3,1^3)$ denotes the partitions $(5,3,1,1,1)$.  The partition $\lambda=(0, 0, \ldots)$ is perfectly valid, and called the {\bf zero partition}. Sometimes it is denoted by the symbol $\emptyset$.

\article
Let $\lambda=(\lambda_1, \lambda_2, \ldots)$ be a partition.  We write $\vert \lambda \vert$ for the sum of the $\lambda_i$, and call this the {\bf size} of $\lambda$. If $\lambda$ has size $n$ we also say that $\lambda$ is a partition of $n$, and write $\lambda \vdash n$.  We write $\ell(\lambda)$ for the number of indices $i$ for which $\lambda_i$ is non-zero, and call this the {\bf length} of $\lambda$.  For example, the partition $(5,3,1)$ has size 9 and length 3.  The zero partition has size 0 and length 0.

\article[Young diagrams]
Partitions are represented visually using Young diagrams.  Such a diagram consists of a number of rows of boxes which are aligned on their left sides and whose lengths are weakly decreasing.  The second condition means no row extends to the right of any row above it.  The partition $\lambda=(\lambda_1, \lambda_2, \ldots)$ corresponds to the Young diagram whose $i$th row has length $\lambda_i$.  For example, the Young diagram
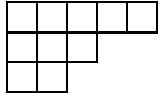
\begin{displaymath}
\ydiagram{5,3,2}
\end{displaymath}
corresponds to the partition $(5,3,2)$.  The total number of boxes in the Young diagram is equal to the size of the corresponding partition, while the number of rows is equal to the length of the corresponding partition.

\article
\label{young:diff}
Now that we have this graphical language, we say that $\lambda$ is contained in $\mu$, written $\lambda \subseteq \mu$, if the Young diagram of $\lambda$ fits into that of $\mu$. This is equivalent to the condition $\lambda_i \le \mu_i$ for all $i$. By definition, the diagram of $\mu / \lambda$ is the complement of the Young diagram of $\mu$ inside of the Young diagram of $\lambda$. These are also known as {\bf skew Young diagrams}.

\article[Transpose]
By flipping a Young diagram along its diagonal one obtains a new Young diagram, called the {\bf transposed} diagram.  For example, the transpose of the diagram pictured above is the diagram
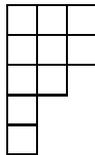
\begin{displaymath}
\ydiagram{3,3,2,1,1}
\end{displaymath}
As we have identified Young diagrams and partitions, we obtain a notion of transpose for partitions.  The above example shows that the transpose of $(5,3,2)$ is $(3,3,2,1,1)$.  We write $\lambda^{\dag}$ for the transpose of $\lambda$.  It can be described symbolically as follows:  
\begin{align}
(\lambda^{\dag})_i = \#\{j \mid \lambda_j \ge i\}.
\end{align}

\article[Frobenius coordinates]
The {\bf rank} of a partition $\lambda$ is the number of boxes along the main diagonal of its Young diagram.  Suppose $\lambda$ has rank $r$.  Let $a_i$ (resp.\ $b_i$) for $1 \le i \le r$ be the number of boxes to the left (resp.\ below) the $i$th box on the main diagonal.  The $a_i$ and $b_i$ are the {\bf Frobenius coordinates} of $\lambda$, and we write $\lambda=(a_1, \ldots, a_r \mid b_1, \ldots, b_r)$ to express $\lambda$ in terms of its Frobenius coordinates.  For example, if $\lambda=(5,3,2)$ then $\lambda$ has rank 2 and $\lambda=(4, 1 \mid 2, 1)$ is the expression of $\lambda$ in terms of its Frobenius coordinates.  Note that $\vert \lambda \vert=r+\sum a_i+\sum b_i$ and $\lambda^{\dag}=(b_1, \ldots, b_r \mid a_1, \ldots, a_r)$.

\xsection{Symmetric groups}

\article[Conjugacy classes] \label{ss:conj}
Let $n \ge 0$ be an integer and let $S_n$ be the symmetric group on $n$ letters.  There is a natural bijection between conjugacy classes in $S_n$ and partitions of $n$.  Under this bijection, a partition $\lambda$ corresponds to the conjugacy class $c_{\lambda}$ of a product $\tau_1 \cdots \tau_k$, where $\tau_i$ is a cycle of length $\lambda_i$ and $k=\ell(\lambda)$.  For example, $c_{(n)}$ is the conjugacy class of an $n$-cycle while $c_{(1^n)}$ is the conjugacy class of the identity element.

\article[Irreducible representations] \label{ss:sym-irrep}
Just like the conjugacy classes, the irreducible representations of the symmetric group $S_n$ are naturally indexed by partitions of $n$.  We write $\bM_{\lambda}$ for the irreducible corresponding to $\lambda$.  This representation is often defined as the right ideal in the group algebra $\bC[S_n]$ generated by a certain idempotent element corresponding to $\lambda$, called the {\bf Young symmetrizer}.  As the details of this construction are not relevant for us, we point the reader to \cite[\S 7]{fulton} (a more in-depth treatment of symmetric group representations over arbitrary rings can be found in the book \cite{james}). These constructions imply that the complex representations of $S_n$ are all realizable over the field of rational numbers $\bQ$. There is an elegant combinatorial rule, known as the {\bf Murnaghan--Nakayama rule} for the characters of the symmetric group, see \cite[Example I.7.5]{macdonald} or \cite[\S\S 7.17--7.18]{stanley}.

\begin{Example}
Let us now give some simple examples.  
\begin{enumerate}
\item If $\lambda=(n)$, so that the corresponding Young diagram has a single row, then $\bM_{\lambda}$ is the trivial representation.  
\item If $\lambda=(1^n)$, so that the corresponding Young diagram has $n$ rows, then $\bM_{\lambda}$ is the sign representation (i.e., the one dimensional representation given by the sign character).  
\item If $\lambda=(n-1, 1)$ then $\bM_{\lambda}$ is the {\bf standard representation}, i.e., the subspace of $\bC^n$ where the coordinates sum to zero, with $S_n$ acting by permuting coordinates.  (The standard representation can also be described as the quotient of $\bC^n$ be the line spanned by $(1, 1, \ldots, 1)$.)
\item If $\lambda=(n-k,1^k)$, then $\bM_{\lambda}$ is the $k$th exterior power of the standard representation. \qedhere
\end{enumerate}
\end{Example}

\article
We have just seen that the irreducible representations of $S_n$ correspond to partitions of $n$, or equivalently, Young diagrams with $n$ boxes.  This leads to the important theme of describing properties or operations on representations in terms of the combinatorics of Young diagrams, or vice versa.

\article[Transpose] \label{ss:transposesign}
For example, there is an involution on the set of Young diagrams with $n$ boxes given by transposition.  What does this operation correspond to in terms of irreducible representations of $S_n$?  The most obvious guess one might make is that it corresponds to formation of the dual representation.  However, this is clearly not the case:  the dual of the trivial representation is again trivial, while the dual of the partition $(n)$ is the partition $(1^n)$.  In fact, \emph{every} finite dimensional representation of $S_n$ is isomorphic to its dual (this is equivalent to all characters being real-valued) so duality does not give an interesting involution on the set of irreducibles.  Rather, transposition of Young diagrams corresponds to twisting by the sign character:  
\begin{align}
\bM_{\lambda^{\dag}} \cong \bM_{\lambda} \otimes \sgn.
\end{align}

\article[The hook-length formula]
Perhaps the most fundamental property of a representation is its dimension.  The dimension of $\bM_{\lambda}$ can be obtained from the Young diagram of $\lambda$ by the so-called {\bf hook-length formula}.  To describe this formula, consider a box $b$ in a Young diagram.  The {\bf hook} of $b$, denoted ${\rm hook}(b)$, is the collection of all boxes to the right of, and in the same row as, $b$, together with those below, and in the same column as $b$; the box $b$ itself is counted as well.  The {\bf hook length} is the number of boxes in the hook.  The hook-length formula \cite[\S 20]{james} then says that
\begin{align}
\dim(\bM_{\lambda})=\frac{\vert \lambda \vert !}{\prod_{b \in \lambda} \textrm{hook}(b)}.
\end{align}
One can also give a determinantal formula
\begin{align}
\dim(\bM_\lambda) = n! \det\left(\frac{1}{(\lambda_i - i + j)!} \right)
\end{align}
with the convention that $1/r! = 0$ if $r<0$ and $0! = 1$ \cite[Corollary 19.5]{james}.

\begin{Example}
To compute the dimension of $\bM_{(5,3,2)}$ we first compute the hook length of each box in the Young diagram.  The result is the following:
\begin{displaymath}
\begin{ytableau}
7 & 6 & 4 & 2 & 1\\
4 & 3 & 1 \\
2 & 1
\end{ytableau}
\end{displaymath}
The hook-length formula now gives
\begin{displaymath}
\dim(\bM_{(5,3,2)})=\frac{10!}{7 \cdot 6 \cdot 4 \cdot 4 \cdot 3 \cdot 2 \cdot 2}=450.
\end{displaymath}
The determinantal formula gives
\[
\dim(\bM_{(5,3,2)}) = 10! \det \begin{pmatrix} 1/5! & 1/6! & 1/7! \\ 1/2! & 1/3! & 1/4! \\ 1/0! & 1/1! & 1/2! \end{pmatrix} = 450. \qedhere
\]
\end{Example}

\article[The Pieri rule] \label{ss:pieri}
There is an inclusion $S_n \subset S_{n+1}$; if we think of $S_{n+1}$ as automorphisms of the set $\{1, \ldots, n+1\}$ then $S_n$ can be described as the stabilizer of $n+1$. While there are other possible embeddings, they are all conjugate to one another, so we lose nothing by considering this one. Given a representation of $S_n$, we can induce it via this inclusion to obtain a representation of $S_{n+1}$.  The Pieri rule, in its simplest form, gives a combinatorial description of the decomposition of the induction of an irreducible representation.  Precisely, it says that if $\lambda$ is a partition of $n$ then
\begin{align}
\Ind_{S_n}^{S_{n+1}}(\bM_{\lambda})=\bigoplus_{\mu \supset \lambda,\ |\mu|-|\lambda|=1} \bM_{\mu},
\end{align}
(This section is a special case of \pref{ss:lw}, so we wait until then to give references.) An important feature of this result is that the induction is multiplicity-free.

\begin{Example}
The induction of the irreducible of $S_5$ corresponding to the Young diagram
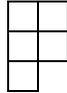
\begin{displaymath}
\ydiagram{2,2,1}
\end{displaymath}
is the direct sum of the irreducibles of $S_6$ correspond to the following Young diagrams:
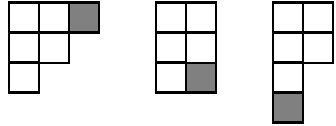
\begin{displaymath}
\ydiagram[*(white)]{2,2,1}*[*(gray)]{3,2,1} \qquad
\ydiagram[*(white)]{2,2,1}*[*(gray)]{2,2,2} \qquad
\ydiagram[*(white)]{2,2,1}*[*(gray)]{2,2,1,1}
\end{displaymath}
The shaded boxes indicate those added by the rule.
\end{Example}

\article
As mentioned, the above rule is just the simplest form of Pieri's rule.  We now give its most general formulation.  Let $\lambda$ be a partition of $n$ and let $m \ge 0$.  We regard $S_n \times S_m$ as a subgroup of $S_{n+m}$ in the obvious manner, and we regard $\bM_{\lambda}$ as a representation of $S_n \times S_m$ with $S_m$ acting trivially.  If $\mu$ contains $\lambda$, we say that $\mu / \lambda$ is a {\bf horizontal strip} of size $m$ if $\mu$ is obtained from $\lambda$ by adding $m$ boxes, no two of which appear in the same column. In this case, we write $\mu / \lambda \in \HS_m$. We also set $\HS = \bigcup_m \HS_m$. Pieri's rule then states
\begin{align}
\Ind_{S_n \times S_m}^{S_{n+m}}(\bM_{\lambda})=\bigoplus_{\mu,\, \mu / \lambda \in \HS_m} \bM_{\mu},
\end{align}
The version of Pieri's rule given above is just the $m=1$ case of this rule.  As in the $m=1$ case, the induction is multiplicity-free.  (Note:  one might think that a more direct generalization of the $m=1$ case of Pieri's rule would be a rule computing the induction from $S_n$ to $S_{n+m}$.  However, this can be obtained easily by iteratively applying the $m=1$ case.)

\begin{Example}
Consider the case where $n=5$, $m=2$ and $\lambda$ is given by the diagram
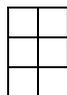
\begin{displaymath}
\ydiagram{2,2,1}
\end{displaymath}
Pieri's rule says that the induced representation decomposes into the irreducibles corresponding to
\begin{displaymath}
\ydiagram[*(white)]{2,2,1}*[*(gray)]{4} \qquad
\ydiagram[*(white)]{2,2,1}*[*(gray)]{3,2,2} \qquad
\ydiagram[*(white)]{2,2,1}*[*(gray)]{3,2,1,1} \qquad
\ydiagram[*(white)]{2,2,1}*[*(gray)]{2,2,2,1}
\end{displaymath}
\end{Example}

\article
Using Frobenius reciprocity \cite[\S 7.2]{serre}, we also get combinatorial rules for restricting a representation of $S_{n+1}$ to $S_n$, namely:
\begin{align}
\bM_\lambda|_{S_n} = \bigoplus_{\mu \subset \lambda,\ |\lambda|-|\mu|=1} \bM_\mu.
\end{align}
So we just consider all possible ways to remove a single box from $\lambda$. The rule for restriction from $S_{n+m}$ to $S_n \times S_m$ is given in \pref{ss:lw}.

\article[The Littlewood--Richardson rule] \label{ss:lw}
As we just saw, the Pieri rule computes the decomposition of the representation
\begin{displaymath}
\Ind_{S_n \times S_m}^{S_{n+m}}(\bM_{\lambda} \otimes 1)
\end{displaymath}
into irreducibles, where $\lambda$ is a partition of $n$ and 1 denotes the trivial representation of $S_m$.  Written in this form, it is natural to try to replace the trivial representation 1 with an arbitrary irreducible of $S_m$.  This is exactly what the Littlewood--Richardson rule accomplishes.  Precisely, let $\lambda$ be a partition of $n$ and let $\mu$ be a partition of $m$.  We then have a decomposition
\begin{align}
\Ind_{S_n \times S_m}^{S_{n+m}}(\bM_{\lambda} \otimes \bM_{\mu})=\bigoplus_{\nu} \bM_{\nu}^{\oplus c_{\lambda,\mu}^{\nu}},
\end{align}
where the sum is over all partitions $\nu$ of $n+m$ and $c_{\lambda,\mu}^{\nu}$ is a non-negative integer.  The numbers $c_{\lambda,\mu}^{\nu}$ are called the {\bf Littlewood--Richardson coefficients}, and the eponymous rule gives a combinatorial description of them. 

\article
There are many descriptions for this rule, and we will formulate it via lattice words. First consider the skew-diagram $\nu / \lambda$. We fill the boxes with positive integers so that $i$ appears exactly $\mu_i$ times. Then $c^\nu_{\lambda, \mu}$ counts the number of such fillings which satisfy the properties (Littlewood--Richardson tableaux):
\begin{itemize}
\item semistandard: the entries are weakly increasing from left to right in each row, and the entries are strictly increasing from top to bottom in each column
\item lattice word: Read the entries right to left in each row, starting with the top row to get a sequence of positive integers (reading word). Then each initial segment of this sequence has the property that for each $i$, $i$ occurs at least as many times as $i+1$.
\end{itemize}
See \cite[\S 5, \S 7.3]{fulton} or \cite[\S I.9]{macdonald}. See also \cite[Appendix 7.A.1.3]{stanley} for some other formulations of the rule. Here are some simple consequences of the Littlewood--Richardson rule:
\begin{itemize}
\item If $c^\nu_{\lambda, \mu} \ne 0$, then $\lambda \subseteq \nu$ and $\mu \subseteq \nu$.
\item For all partitions $\lambda, \mu$, $c^{\lambda + \mu}_{\lambda, \mu} = 1$, and $c^{\lambda \cup \mu}_{\lambda, \mu} = 1$ where $\lambda \cup \mu$ denotes the partition obtained by sorting the sequence $(\lambda, \mu)$. To prove these, fill the Young diagram of $\mu$ with the number $i$ in each box in the $i$th row. Append the $i$th row to the $i$th row of $\lambda$ to see $c^{\lambda+\mu}_{\lambda, \mu} \ge 1$. Append the $i$th column to the $i$th column of $\lambda$ to see $c^{\lambda \cup \mu}_{\lambda, \mu} \ge 1$. The reverse inequalities follow by the extremality of these shapes.
\item For all integers $N > 0$, we have $c^{N\nu}_{N\lambda, N\mu} \ge c^{\nu}_{\lambda, \mu}$, which can be seen by ``stretching'' the Littlewood--Richardson tableau. As a consequence, if $c^{\nu}_{\lambda, \mu} > 0$, then $c^{N\nu}_{N\lambda, N\mu} > 0$ for any $N>0$. The converse of this statement is also true, i.e., if $c^{N\nu}_{N\lambda, N\mu} > 0$ for some $N>0$, then $c^{\nu}_{\lambda, \mu} > 0$. This is a highly non-trivial fact known as the saturation theorem, see \cite{kt, dw, km} for different proofs of it. Furthermore, the function $N \mapsto C^{N\nu}_{N\lambda,N\mu}$ is a polynomial in $N \ge 0$ for any fixed choice of $\lambda, \mu, \nu$ \cite[Corollary 3]{LRpoly}.
\end{itemize}

And here are some properties which are not obvious from the Littlewood--Richardson rule, but follow easily from the representation-theoretic interpretation:
\begin{itemize}
\item Symmetry: $c^\nu_{\lambda, \mu} = c^\nu_{\mu, \lambda}$. One way to give a symmetric combinatorial rule for $c^\nu_{\lambda, \mu}$ is to use the plactic monoid and jeu de taquin \cite[\S 2, \S 5.1]{fulton}.
\item Transpose symmetry: $c^{\nu^\dagger}_{\lambda^\dagger, \mu^\dagger} = c^\nu_{\lambda, \mu}$.
\end{itemize}

\begin{Example} 
We calculate $c^{(5,3,2,1)}_{(3,1), (4,2,1)} = 3$. The Littlewood--Richardson tableaux are
\[
\begin{ytableau}
\ & & & 1 & 1 \\
& 1 & 1 \\ 
2 & 2\\
3
\end{ytableau} \qquad
\begin{ytableau}
\ & & & 1 & 1 \\
& 1 & 2 \\ 
1 & 2\\
3
\end{ytableau} \qquad
\begin{ytableau}
\ & & & 1 & 1 \\
& 1 & 2 \\ 
1 & 3\\
2
\end{ytableau} \qquad
\]
The reading words are 1111223, 1121213, and 1121312, respectively.
It is easier to calculate this number after swapping the roles of $(3,1)$ and $(4,2,1)$:
\[
\begin{ytableau}
\ &  &  & & 1 \\
 & & 2 \\
 & 1 \\
1 
\end{ytableau} \qquad 
\begin{ytableau}
\ &  &  & & 1 \\
 & & 1 \\
 & 2 \\
1 
\end{ytableau} \qquad
\begin{ytableau}
\ &  &  & & 1 \\
 & & 1 \\
 & 1 \\
2
\end{ytableau} 
\]
\end{Example}

\article
Using Frobenius reciprocity \cite[\S 7.2]{serre}, we get the following restriction rule 
\begin{align}
\bM_\nu |^{S_{n+m}}_{S_n \times S_m} = \bigoplus_{\lambda, \mu} (\bM_\lambda \otimes \bM_\mu)^{\oplus c^\nu_{\lambda, \mu}}.
\end{align}

\article[Decomposition of tensor products] \label{ss:sym-ten}
Let $\lambda$ and $\mu$ be partitions of $n$.  We have a decomposition
\begin{align}
\bM_{\lambda} \otimes \bM_{\mu} = \bigoplus_{\nu} \bM_{\nu}^{\oplus g_{\lambda,\mu,\nu}},
\end{align}
where the sum is over the partitions $\nu$ of $n$ and the $g_{\lambda,\mu,\nu}$ are non-negative integers, called the {\bf Kronecker coefficients}.  Determining the multiplicities $g_{\lambda,\mu,\nu}$ above may seem like a more natural problem than determining the Littlewood--Richardson coefficients.  However, for our applications it is the Littlewood--Richardson coefficients that are more relevant. Furthermore, a combinatorial formula for the $g_{\lambda,\mu,\nu}$ analogous to the Littlewood--Richardson rule is not known in general. A positive, combinatorial rule for $g_{\lambda, \mu, \nu}$ when two of the partitions $\lambda, \mu, \nu$ have at most 2 parts is given in \cite{gct4} (and see the references therein for other special cases). We remark that, because all irreducibles of $S_n$ are self-dual, the quantity $g_{\lambda,\mu,\nu}$ is symmetric in $\lambda$, $\mu$ and $\nu$.

\xsection{General linear groups}

\article
Let $n \ge 0$ be an integer and let $G=\GL(n)$ be the general linear group, i.e., the group of linear automorphisms of the vector space $\bC^n$.  We let $B$ denote the subgroup of $G$ consisting of upper triangular matrices in $G$; this is the {\bf standard Borel}.  We let $T$ denote the subgroup of $B$ consisting of diagonal matrices; this is the {\bf standard maximal torus}.  We let $U$ denote the subgroup of strictly upper triangular matrices, i.e., the elements of $B$ whose diagonal entries are 1; this is the {\bf unipotent radical} of $B$.

We use $\fgl(n)$ to denote the Lie algebra of $G$. Its universal enveloping algebra is denoted $\rU(\fgl(n))$.

\article[Rational and polynomial representations]
Let $V$ be a finite dimensional representation of $G$ and let $\rho\colon G \to \GL(V)$ denote the action map.  We say that $V$ is {\bf algebraic} or {\bf rational} (the two terms are synonymous) if the matrix entries of $\rho(g)$ are expressible as rational functions of those of $g$.  Similarly, we say that $V$ is {\bf polynomial} if the matrix entries of $\rho(g)$ are expressible as polynomials in the matrix entries of $G$.  Every representation we consider will be rational, so we typically say ``representation'' in place of ``rational representation.''  The class of rational representations is closed under formation of direct sums, tensor products and duals.  The class of polynomial representations is closed under formation of direct sums and tensor products, but not duals.

\article
The simplest examples of these concepts are provided by powers of the determinant:  for any integer $k$, we have a homomorphism $\det^k\colon G \to \bC^{\times}$, which we can regard as a one-dimensional representation of $G$.  Since the determinant is a polynomial in the entries of $G$, this representation is always rational, and it is polynomial if $k$ is non-negative.  A second basic example is the standard representation:  the action of $G=\GL(n)$ on $\bC^n$ is a polynomial representation.  The dual of this representation is a rational representation, but is no longer polynomial.

\article[The weight lattice]
Let $X$ denote the set of algebraic homomorphisms $T \to \bC^{\times}$.  An element of $X$ is called a {\bf weight} and $X$ is called the {\bf weight lattice}; it forms an abelian group.  If we denote by $[a_1, \ldots, a_n]$ the diagonal matrix with entries $a_1, \ldots, a_n$, then it is easy to see that any weight is of the form
\begin{align}
[a_1, \ldots, a_n] \mapsto a_1^{k_1} \cdots a_n^{k_n}
\end{align}
for integers $k_i$.  We thus have a natural isomorphism of $X$ with $\bZ^n$, and in what follows we often identify the two.  A weight is {\bf dominant} if, in the above notation, the sequence $k_i$ is weakly decreasing.  A weight is {\bf non-negative} if the $k_i$ are all non-negative.

\article[Highest weight theory] \label{ss:highestweight}
The main results of the representation theory of $\GL(n)$ are summarized as follows (in what follows all representations are rational):
\begin{enumerate}[(a)]
\item Every representation of $G$ is a direct sum of irreducible representations.
\item A representation $V$ of $G$ is irreducible if and only if $\dim V^U = 1$.  If $V$ is irreducible, then the action of $T$ on $V^U$ is through a dominant weight.  We call this the {\bf highest weight} of $V$.
\item Two irreducible representations of $G$ are isomorphic if and only if their highest weights are equal.
\item For any dominant weight $\lambda$ there exists an irreducible representation $V_{\lambda}$ with highest weight $\lambda$.
\item The irreducible representation $V_{\lambda}$ is polynomial if and only if $\lambda$ is non-negative.  A general rational representation is polynomial if and only if all of its irreducible constituents are.
\end{enumerate}
By the above, a polynomial irreducible representation has highest weight of the form $\lambda=(\lambda_1, \ldots, \lambda_n)$ where each $\lambda_i$ is non-negative and the $\lambda_i$ are weakly decreasing.  We can therefore think of $\lambda$ as a partition of length at most $n$. See \cite[\S 5.8]{kraftprocesi} or \cite[\S 8.2]{fulton} for details.

\article[The hook-content formula] \label{art:hookcontent}
The dimension of $V_\lambda$ is given by the {\bf hook-content formula}. Given a box $b$ in the $i$th row and $j$th column of $\lambda$, its {\bf content} is ${\rm cont}(b) = j-i$. Then the hook-content formula says
\begin{align}
\dim V_\lambda = \prod_{b \in \lambda} \frac{n + {\rm cont}(b)}{{\rm hook}(b)}.
\end{align}
This is connected to a few facts. First, $\dim V_\lambda$ counts ``semistandard Young tableaux''. See for example, \cite[Example I.A.8.1]{macdonald} or \cite[Proposition 2.1.4]{weyman} (in this reference, $L_\lambda$ is isomorphic to $V_{\lambda^\dagger}$). The hook-content formula is often stated as an enumeration for semistandard Young tableaux (\cite[Example I.3.4]{macdonald} or \cite[Corollary 7.21.4]{stanley}). In particular, as $n$ varies, but $\lambda$ is fixed, this dimension is a polynomial in $n$ of degree $|\lambda|$.

\begin{Example}
Let $n = 6$ and $\lambda = (4,2,1)$. The contents and hooks for $\lambda$ are, respectively,
 \begin{displaymath}
\begin{ytableau}
0 & 1 & 2 & 3 \\
\text{-}1 & 0 \\
\text{-}2
\end{ytableau} \qquad
\begin{ytableau}
6 & 4 & 2 & 1 \\
3 & 1 \\
1
\end{ytableau}
\end{displaymath}
So the hook-content formula gives (we order boxes left to right, top to bottom)
\[
\dim V_{(4,2,1)} = \frac{6 \cdot 7 \cdot 8 \cdot 9 \cdot 5 \cdot 6 \cdot 4}{6 \cdot 4 \cdot 2 \cdot 1 \cdot 3 \cdot 1 \cdot 1} = 2520. \qedhere
\]
\end{Example}

\article
Let us now give some examples to illustrate highest weight theory.
\begin{enumerate}[(a)]
\item {\it One-dimensional representations.}  Consider the one dimensional representation $V$ given by $\det^k$.  Of course, $U$ acts trivially on $V$ and so $V=V^U$.  The action of an element $[a_1, \ldots, a_n]$ in $T$ on $V$ is given by $a_1^k \cdots a_n^k$.  This is the weight $(k, \ldots, k)$.  Thus the representation $\det^k$ has highest weight $(k, \ldots, k)$.
\item {\it The standard representation.}  Let $x_1, \ldots, x_n$ be the standard basis for $V=\bC^n$.  Then $V^U$ is the line spanned by $x_1$.  If $t=[a_1, \ldots, a_n]$ then $tx_1=a_1x_1$.  Thus the action of $T$ on $V^U$ is through the weight $(1, 0, 0, \ldots, 0)$, and so this is the highest weight of the standard representation.  A similar computation shows that the highest weight of $V^*$, the dual of the standard representation, is given by $(0, \ldots, 0, -1)$.
\item {\it Symmetric powers.}  Now consider the case $V=\Sym^k(\bC^n)$.  We can think of $V$ as the space of homogeneous degree $k$ polynomials in the variables $x_1, \ldots, x_n$.  An easy computation shows that $V^U$ is spanned by $x_1^k$.  It follows that $\Sym^k(\bC^n)$ is irreducible and has highest weight $(k, 0, \ldots, 0)$.
\item {\it Exterior powers.}  Finally, consider the case $V=\bw{k}(\bC^n)$ with $k \le n$.  We can think of $V$ as the space spanned by $k$-fold wedges in the elements $x_i$.  One computes that $V^U$ is spanned by $x_1 \wedge x_2 \wedge \cdots \wedge x_k$.  It follows that $\bw{k}(\bC^n)$ is irreducible and has highest weight $(1, \ldots, 1, 0, \ldots, 0)$, where there are $k$ 1's and $n-k$ 0's.
\end{enumerate}

\article[Tensor products] \label{ss:gln-ten}
Having classified irreducible representations, we would now like to understand how the tensor product of two irreducible representations decomposes.  To begin with, one easily sees that $V_{\lambda} \otimes \det^k$ is the irreducible with highest weight $(\lambda_1+k,\ldots,\lambda_n+k)$.  Thus to say how $V_{\lambda} \otimes V_{\mu}$ decomposes in general it suffices to treat the case where $\lambda$ and $\mu$ are non-negative, since we can first twist by an appropriate power of the determinant to move into this case, then decompose and then untwist.  Assuming $\lambda$ and $\mu$ are non-negative, we have the following incredible result:
\begin{align}
V_{\lambda} \otimes V_{\mu} = \bigoplus_{\nu} V_{\nu}^{\oplus c^{\nu}_{\lambda,\mu}}
\end{align}
where the sum is over non-negative dominant weights $\nu$ with $\vert \nu \vert=\vert \lambda \vert+\vert \mu \vert$ and $c^{\nu}_{\lambda,\mu}$ is the Littlewood--Richardson coefficients introduced in \pref{ss:lw}!  We will show how this result can be deduced from Schur--Weyl duality in \pref{ss:sw-ten}.

\article \label{art:gln-pieri}
An important corollary of \pref{ss:gln-ten} is that if $V=V_{\lambda}$ is an irreducible representation of $G$ then $V \otimes \Sym^k(\bC^n)$ is multiplicity-free for any $k$.  Indeed, $\Sym^k(\bC^n)$ has highest weight $\mu=(k, 0, \ldots, 0)$ and so the Littlewood--Richardson coefficients $c_{\lambda,\mu}^{\nu}$ are computed by Pieri's rule, which we know is multiplicity free, i.e., the coefficients are all 0 or 1 (see \pref{ss:pieri}). Similar statements hold for exterior powers. In particular, we get the formulas
\begin{align}
V_\lambda \otimes V_d &= \bigoplus_{\mu,\ \mu / \lambda \in \HS_d} V_\mu,\\
V_\lambda \otimes V_{1^d} &= \bigoplus_{\mu,\ \mu / \lambda \in \VS_d} V_\mu.
\end{align}
Explicit formulas for the inclusions $V_\mu \subset V_\lambda \otimes V_d$ and $V_\mu \subset V_\lambda \otimes V_{1^d}$ were given in \cite[\S 6]{olver}, and computer implementations in {\tt Macaulay 2} of these maps have been written \cite{pierimaps}.

\article[Branching rules] \label{ss:GLbranching}
We can embed $\GL(n) \times \GL(m)$ into $\GL(n+m)$ as the block diagonal matrices. We can describe the restriction of an irreducible polynomial representation of $\GL(n+m)$ again using Littlewood--Richardson coefficients. To avoid confusion, we use $V^{(N)}_\lambda$ if we want to emphasize that $V_\lambda$ is a representation $\GL(N)$. Then the branching formula is
\begin{align}
V^{(n+m)}_\nu|^{\GL(n+m)}_{\GL(n) \times \GL(m)} = \bigoplus_{\lambda, \mu} (V^{(n)}_{\lambda} \boxtimes V^{(m)}_\mu)^{\oplus c^\nu_{\lambda, \mu}},
\end{align}
where again the $c^\nu_{\lambda, \mu}$ are Littlewood--Richardson coefficients. As above, the case of a general rational representation can be reduced to the polynomial case. We will deduce this result from Schur--Weyl duality in \pref{art:sw-branching}.

\article \label{ss:kroneckerbranch}
We also have a map $\pi \colon \GL(n) \times \GL(m) \to \GL(nm)$ by considering the natural action of $\GL(n) \times \GL(m)$ on $\bC^n \otimes \bC^m$. Given a representation $V^{(nm)}_\nu$ of $\GL(nm)$, we can describe its pullback to $\GL(n) \times \GL(m)$ using Kronecker coefficients (see \pref{ss:sym-ten})
\begin{align}
\pi^*V^{(nm)}_\nu = \bigoplus_{\lambda, \mu} (V^{(n)}_\lambda \boxtimes V^{(m)}_\mu)^{\oplus g_{\lambda, \mu, \nu}}.
\end{align}
Again, this will be deduced from Schur--Weyl duality in \pref{art:sw-kronecker}.

\article[Cauchy identities] \label{ss:cauchy}
While we have mentioned that the coefficients $g_{\lambda, \mu, \nu}$ are difficult to describe, there are two easy cases which are very useful. First, if $\nu = (k)$ is the one-row partition, then $g_{\lambda, \mu, (k)} = \delta_{\lambda, \mu}$ since representations of $S_k$ are self-dual. In particular, this implies
\begin{align}
\Sym^k(\bC^n \otimes \bC^m) = \bigoplus_\lambda V^{(n)}_\lambda \boxtimes V^{(m)}_\lambda
\end{align}
where the sum is over all partitions $\lambda$ with $\ell(\lambda) \le \min(n,m)$. Second, if $\nu = (1^k)$ is the one-column partition, then $g_{\lambda, \mu, (1^k)} = \delta_{\lambda, \mu^\dagger}$ (see \pref{ss:transposesign}). This implies
\begin{align}
\bigwedge^k(\bC^n \otimes \bC^m) = \bigoplus_\lambda V^{(n)}_\lambda \boxtimes V^{(m)}_{\lambda^\dagger}
\end{align}
where the sum is over all partitions $\lambda$ with $\ell(\lambda) \le n$ and $\lambda_1 \le m$.

\article[Infinite dimensional representations]
We define an infinite dimensional representation of $G$ to be rational, resp.\ polynomial, if it is a direct sum of finite dimensional rational, resp.\ polynomial, representations.  The class of infinite dimensional rational or polynomial representations is abelian and stable under tensor products, but not duality (even in the rational case).  A typical example of the kind of infinite dimensional representations we will encounter is $\Sym(\bC^n)$, the symmetric algebra on $\bC^n$.  It is the direct sum of the various $\Sym^k(\bC^n)$'s.

\xsection{Schur--Weyl duality}
\label{s:sw}

\article \label{ss:sw}
Let $n$ and $k$ be non-negative integers.  The groups $\GL(n)$ and $S_k$ each act on the space $(\bC^n)^{\otimes k}$ --- the $\GL(n)$ action comes from its action on $\bC^n$, while the group $S_k$ acts by permuting tensor factors.  These two actions commute, and so the product group $S_k \times \GL(n)$ acts.  {\bf Schur--Weyl duality} describes how this space decomposes into irreducible representations under the product group.  Precisely, it states
\begin{align}
(\bC^n)^{\otimes k}=\bigoplus_{\lambda} \bM_{\lambda} \otimes V_{\lambda}
\end{align}
where the sum is over partitions $\lambda$ of $k$ of length at most $n$ \cite[\S 5.9]{kraftprocesi}.  This is an extremely important result, as it provides an explicit link between the representation theories of general linear groups and symmetric groups.

\begin{Example}
When $k=2$ the theorem reduces to the decomposition
\begin{displaymath}
\bC^n \otimes \bC^n = (\Sym^2(\bC^n) \otimes 1) \oplus (\bw{2}(\bC^n) \oplus \sgn).
\end{displaymath}
In words, every element of $\bC^n \otimes \bC^n$ can be written as a sum of a symmetric tensor and an anti-symmetric tensor and furthermore, the symmetric tensors form the space $\Sym^2(\bC^n)$ while the anti-symmetric tensors form the space $\bw{2}(\bC^n)$.
\end{Example}

\article[Weyl's construction] \label{ss:weyl}
An important application of Schur--Weyl duality is to the \emph{construction} of the irreducible representations of $\GL(n)$.  Indeed, it follows easily from the theorem that
\begin{displaymath}
\Hom_{S_k}(\bM_{\lambda}, (\bC^n)^{\otimes k}) = \begin{cases}
V_{\lambda} & \textrm{if $\ell(\lambda) \le n$} \\
0 & \textrm{if $\ell(\lambda)>n$.}
\end{cases}
\end{displaymath}
In fact, this was the method Weyl used to construct $V_{\lambda}$ \cite[Theorem 4.4.F]{weyl}. 

\article[Tensor product decompositions for $\GL(n)$] \label{ss:sw-ten}
Let us now show how the Schur--Weyl decomposition can be used to deduce the tensor product rule for representations of $\GL(n)$ stated in \pref{ss:gln-ten}.  Thus let $\lambda$ and $\mu$ be two partitions of length at most $n$ and of size $i$ and $j$ respectively.  We have
\begin{displaymath}
\begin{split}
V_{\lambda} \otimes V_{\mu}
&=\Hom_{S_i}(\bM_{\lambda}, (\bC^n)^{\otimes i}) \otimes \Hom_{S_j}(\bM_{\mu}, (\bC^n)^{\otimes j}) \\
&=\Hom_{S_i \times S_j}(\bM_{\lambda} \otimes \bM_{\mu}, (\bC^n)^{\otimes (i+j)}).
\end{split}
\end{displaymath}
Now, Frobenius reciprocity \cite[\S 7.2]{serre} says that giving an $S_i \times S_j$ equivariant map from $\bM_{\lambda} \otimes \bM_{\mu}$ to a representation of $S_{i+j}$ is the same as giving an $S_{i+j}$ equivariant map from the induction.  We thus have
\begin{displaymath}
\begin{split}
V_{\lambda} \otimes V_{\mu}
&=\Hom_{S_{i+j}}(\Ind_{S_i \times S_j}^{S_{i+j}}(\bM_{\lambda} \otimes \bM_{\mu}), (\bC^n)^{\otimes (i+j)}) \\
&= \bigoplus_{\nu} \Hom_{S_{i+j}}(\bM_{\nu}, (\bC^n)^{\otimes (i+j)})^{\oplus c_{\lambda,\mu}^{\nu}} \\
&= V_{\nu}^{\oplus c_{\lambda,\mu}^{\nu}}.
\end{split}
\end{displaymath}
In the second line we used the Littlewood--Richardson rule \pref{ss:lw} to decompose the induced representation.  This completes the derivation.

\article[Branching rules for $\GL(n+m)$] \label{art:sw-branching}
Now we use Schur--Weyl duality to deduce the branching rule stated in \pref{ss:GLbranching}. Set $k = |\nu|$. We have
\begin{align*}
V^{(n+m)}_\nu &= \Hom_{S_k}(\bM_\nu, (\bC^{n+m})^{\otimes k}).
\end{align*}
Now $(\bC^{n+m})^{\otimes k}$ can be written as $\bigoplus_I \bigotimes_{j=1}^k A_{j,I}$ where the sum is over all subsets of $\{1,\dots,k\}$ and $A_{j,I} = \bC^n$ if $j \in I$ and $A_{j,I} = \bC^m$ if $j \notin I$. For every $N$, the symmetric group $S_k$ preserves the sum $\bigoplus_{|I|=N} \bigotimes_j A_{j,I}$, and this representation is the induced representation
\[
\Ind^{S_k}_{S_N \times S_{k-N}}((\bC^n)^{\otimes N} \otimes (\bC^m)^{\otimes (k-n)}).
\]
Hence by Frobenius reciprocity \cite[\S 7.2]{serre}, we can write
\begin{align*}
V^{(n+m)}_\nu &= \bigoplus_{N=0}^k \Hom_{S_N \times S_{k-N}}(\bM_\nu|^{S_k}_{S_N \times S_{k-N}}, (\bC^n)^{\otimes N} \otimes (\bC^m)^{\otimes (k-N)})\\
&= \bigoplus_{N=0}^k \bigoplus_{\lambda, \mu} \Hom_{S_N \times S_{k-N}}(\bM_\lambda \otimes \bM_\mu, (\bC^n)^{\otimes N} \otimes (\bC^m)^{\otimes (k-N)})^{\oplus c^\nu_{\lambda, \mu}}\\
&= \bigoplus_{\lambda, \mu} (V^{(n)}_\lambda \boxtimes V^{(m)}_\mu)^{\oplus c^\nu_{\lambda, \mu}}.
\end{align*}

\article \label{art:sw-kronecker}
Now we use Schur--Weyl duality to deduce the identities in \pref{ss:kroneckerbranch}. First, we have
\begin{align*}
(\bC^n \otimes \bC^m)^{\otimes k} &= \bigoplus_{|\nu|=k} V^{(nm)}_\nu \boxtimes \bM_\nu
\end{align*}
as representations of $\GL(nm) \times S_k$. Alternatively, as $\GL(n) \times \GL(m) \times S_k$ representations, we have
\begin{align*}
(\bC^n)^{\otimes k} \otimes (\bC^m)^{\otimes k} &= (\bigoplus_{|\lambda|=k} V^{(n)}_\lambda \boxtimes \bM_\lambda) \otimes (\bigoplus_{|\mu|=k} V^{(m)}_\mu \boxtimes \bM_\mu)\\
&= \bigoplus_{|\lambda| = |\mu| = |\nu| = k} (V^{(n)}_\lambda \boxtimes V^{(m)}_\mu \boxtimes \bM_\nu)^{\oplus g_{\lambda, \mu, \nu}}.
\end{align*}
The result follows by taking the $\bM_\nu$-isotypic component of both expressions.

\part{The category $\cV$}

\section{Models for $\cV$}
\label{sec:models}

\subsection{The sequence and fs models} \label{sec:sequenceV}

\article
We define $\Rep(S_{\ast})$ to be the following category:
\begin{itemize}
\item Objects are sequences $V=(V_n)_{n \ge 0}$ where $V_n$ is a representation of $S_n$.  We call $V_n$ the ``degree $n$ piece'' of $V$.
\item A morphism $f\colon V \to W$ is a sequence $f=(f_n)_{n \ge 0}$ where $f_n\colon V_n \to W_n$ is a map of $S_n$-representation.
\end{itemize}
We call $\Rep(S_{\ast})$ the ``sequence model.''  One easily verifies that $\Rep(S_{\ast})$ is an abelian category.  Direct sums, kernels and cokernels are computed point-wise, e.g., if $f\colon V \to W$ is a morphism then $(\ker{f})_n=\ker(f_n)$.

\article[Simple objects]
For a partition $\lambda$ we have an irreducible representation $\bM_{\lambda}$ of $S_n$, where $n=\vert \lambda \vert$
(see \pref{ss:sym-irrep}).  We can regard these as objects of $\Rep(S_{\ast})$ by placing 0 in degrees $\ne n$.
It is clear that $\bM_{\lambda}$ defines a simple object of $\Rep(S_{\ast})$, and that any simple object is isomorphic to one of this form.  We therefore see that the isomorphism classes of $\Rep(S_{\ast})$ are in natural bijective correspondence with partitions --- there is no restriction on the length or size of the partition.  Every object of $\Rep(S_{\ast})$ is a (possibly infinite) direct sum of simple objects.

\article[Tensor product of graded vector spaces]
Let $V$ and $W$ be graded vector spaces.  There are two ways one can define a tensor product of these spaces:
\begin{displaymath}
(V \boxtimes W)_n=V_n \otimes W_n, \qquad (V \otimes W)_n=\bigoplus_{i+j=n} V_i \otimes W_j.
\end{displaymath}
For almost all purposes, the second tensor product is the correct one.  For instance, it has the favorable property that its underlying vector space is the usual tensor product of $V$ and $W$.  This is not true for the first tensor product:  for instance, if $V$ and $W$ are supported in complementary degrees then $V \boxtimes W$ vanishes!

\article[The tensor product in $\Rep(S_{\ast})$] \label{ss:repStensor}
Similarly, there are two ways one could define the tensor product of objects $V$ and $W$ of $\Rep(S_{\ast})$:
\begin{align}
(V \boxtimes W)_n&=V_n \otimes W_n, \\
(V \otimes W)_n&=\bigoplus_{i+j=n} \Ind_{S_i \times S_j}^{S_n} (V_i \otimes W_i).
\end{align}
As before, the first tensor product --- which we refer to as the {\bf point-wise tensor product} --- is usually not the one we want to use (although it will come up on occasion).  We refer to the second product simply as ``the'' tensor product.  To decompose the tensor product of two irreducibles one applies the Littlewood--Richardson rule (see \pref{ss:lw}).  The point-wise tensor product of two irreducibles in complementary degrees is zero.  To decompose the point-wise tensor products of two irreducibles in the same degree one uses the Kronecker coefficients discussed in \pref{ss:sym-ten}.

\article[The category $\Vec^{\fs}$]
Let $\fs$ denote the category whose objects are finite sets and whose morphisms are bijections of finite sets.  Let $\Vec^{\fs}$ denote the category of functors $\fs \to \Vec$.  To elaborate, we have the following description of objects and morphisms in $\Vec^{\fs}$:
\begin{itemize}
\item An object $V$ of $\Vec^{\fs}$ assigns to each finite set $L$ a vector space $V_L$ and to each bijection of finite sets $L \to L'$ an isomorphism of vector spaces $V_L \to V_{L'}$ in a manner compatible with composition.  In particular, $V_L$ is a representation of the group $\Aut(L)$ (which is isomorphic to $S_n$ with $n=\# L$).
\item A morphism $f\colon V \to V'$ in $\Vec^{\fs}$ assigns to each finite set $L$ a linear map $f_L\colon V_L \to V'_L$ such that if $L \to L'$ is a bijection of finite sets then the diagram
\begin{displaymath}
\xymatrix{
V_L \ar[r]^{f_L} \ar[d] & V'_L \ar[d] \\
V_{L'} \ar[r]^{f_{L'}} \ar[r] & V'_{L'} }
\end{displaymath}
commutes.  In particular, $f_L\colon V_L \to V'_L$ is a map of $\Aut(L)$-representations.
\end{itemize}
We call $\Vec^{\fs}$ the ``fs-model.'' The objects in this category was previously introduced by Joyal under the name of ``tensorial species'' \cite{joyal}.

\article[The equivalence between $\Rep(S_{\ast})$ and $\Vec^{\fs}$] \label{ss:sym-fs-equiv}
Let $[n]$ denote the finite set $\{1,\ldots,n\}$.  If $V$ is an object of $\Vec^{\fs}$ then $V_{[n]}$ carries a representation of $S_n$.  We thus have a functor
\begin{align}
\Vec^{\fs} \to \Rep(S_{\ast}), \qquad V \mapsto (V_{[n]})_{n \ge 0}.
\end{align}
This functor is easily seen to be an equivalence, since every object of $\fs$ is isomorphic to some $[n]$.

\article
Given that $\Rep(S_{\ast})$ and $\Vec^{\fs}$ are so obviously equivalent, one may wonder why we bother introducing $\Vec^{\fs}$ at all.  In fact, each has its place.  The category $\Rep(S_{\ast})$ is more elementary and concrete, so it can be easier to deal with at times.  However, the category $\Vec^{\fs}$ is often more natural, and many constructions can be simpler when phrased in its language.  A good example of this is the tensor product, discussed in the next section.

\article[Tensor products in $\Vec^{\fs}$]
Let $V$ and $W$ be objects of $\Vec^{\fs}$.  We define their tensor product by
\begin{align}
(V \otimes W)_L=\bigoplus_{L=A \amalg B} V_A \otimes V_B,
\end{align}
where the sum is over all partition of $L$ into two subsets.  An easy exercise shows that the equivalence $\Vec^{\fs} \to \Rep(S_{\ast})$ given in the previous section is a tensor functor, i.e., the image of $V \otimes W$ is naturally isomorphic to the tensor product of the images of $V$ and $W$.  For many purposes, the tensor product in $\Vec^{\fs}$ is easier to deal with than the one in $\Rep(S_{\ast})$, as it does not involve the complicated operation of induction.  For instance, it is plainly evident that the tensor product in $\Vec^{\fs}$ is associative.  It is not hard to show that this is the case for the one in $\Rep(S_{\ast})$, but it is certainly not as immediate.

\article
The point-wise tensor product is also easy to express in $\Vec^{\fs}$:
\begin{align}
(V \boxtimes W)_L=V_L \otimes W_L.
\end{align}

\begin{Remark}
The category $\fs$ is a monoid under disjoint union.  Intuitively, one can therefore think of $\fs$ as being like a group and one can think of $\Vec^{\fs}$ as the space of functions on it.  In this analogy, the point-wise tensor product corresponds to the point-wise product of functions, while the tensor product corresponds to convolution of functions.
\end{Remark}

\subsection{The $\GL$-model} \label{sec:GLmodelV}

\article
Let $\bC^{\infty}$ denote the vector space with basis $e_1, e_2, \ldots$.  We let $\GL(\infty)$ denote the group of automorphisms $g$ of $\bC^{\infty}$ such that $ge_i=e_i$ for $i \gg 0$.  We regard $\bC^n$ as the subspace of $\bC^{\infty}$ spanned by $e_1, \ldots, e_n$, and we regard $\GL(n)$ as a subgroup of $\GL(\infty)$ in the corresponding manner.  We can thus describe $\bC^{\infty}$ as the union of the $\bC^n$'s and $\GL(\infty)$ as the union of $\GL(n)$'s.

\begin{Proposition}
\label{ss:sw-inf}
For a partition $\lambda$ of $k$, put
\begin{align}
V_{\lambda}=\Hom_{S_k}(\bM_{\lambda}, (\bC^{\infty})^{\otimes k}).
\end{align}
We have the following:
\begin{enumerate}[\rm (a)]
\item The space $V_{\lambda}$ is a non-zero irreducible representation of $\GL(\infty)$.
\item If $V_{\lambda}$ and $V_{\mu}$ are isomorphic then $\lambda=\mu$.
\item The natural map
\begin{align}
\bigoplus_{\lambda \vdash k} \bM_{\lambda} \otimes V_{\lambda} \to (\bC^{\infty})^{\otimes k}
\end{align}
is an isomorphism of $S_k \times \GL(\infty)$ representations.
\end{enumerate}
\end{Proposition}

\begin{proof}
For a non-negative integer $n$, put
\begin{displaymath}
V_{\lambda,n}=\Hom_{S_k}(\bM_{\lambda}, (\bC^n)^{\otimes k}).
\end{displaymath}
We have inclusions $V_{\lambda,n} \subset V_{\lambda,n+1} \subset V_{\lambda}$.  As $V_{\lambda,n}$ is non-zero for $n \ge \ell(\lambda)$ (see \pref{ss:weyl}), we see that $V_{\lambda}$ is non-zero.  It is clear that the natural map
\begin{displaymath}
\varinjlim V_{\lambda,n} \to V_{\lambda}
\end{displaymath}
is an isomorphism, since any map $\bM_{\lambda} \to (\bC^{\infty})^{\otimes k}$ has image in $(\bC^n)^{\otimes k}$ for $n$ large enough.  Suppose now that $U$ is a non-zero $\GL(\infty)$-stable subspace of $V_{\lambda}$.  Then $U \cap V_{\lambda,n}$ is a $\GL(n)$-stable subspace of $V_{\lambda,n}$.  For $n$ sufficiently large, this intersection is necessarily non-empty and therefore all of $V_{\lambda,n}$, as $V_{\lambda,n}$ is irreducible for $\GL(n)$ (see \pref{ss:weyl}).  Thus $U$ contains $V_{\lambda,n}$ for $n$ large, and is therefore all of $V_{\lambda}$.  This establishes statement (a).

We now prove (b).  Let $U_n$ be the unipotent radical of the standard Borel in $\GL(n)$ and let $T_n$ be the standard maximal torus in $\GL(n)$.  Let $U$ (resp.\ $T$) be the union of the $U_n$ (resp.\ $T_n$).  Note that any partition $\lambda$ can be regarded as a weight of $T$, via
\begin{displaymath}
[a_1,a_2,\ldots] \mapsto a_1^{\lambda_1} a_2^{\lambda_2} \ldots
\end{displaymath}
Let $n \ge \ell(\lambda)$.  We know that $V_{\lambda,n}^{U_n}$ is one dimensional, and $T_n$ acts on it through the weight $\lambda$.  It is clear that the $\lambda$ weight spaces of $(\bC^n)^{\otimes k}$ and $(\bC^{n+1})^{\otimes k}$ coincide, as no tensor involving the basis vector $e_{n+1}$ can have weight $\lambda$.  Thus $V_{\lambda,n+1}^{U_{n+1}}$ is contained in $(\bC^n)^{\otimes k}$, and therefore must equal $V_{\lambda,n}^{U_n}$.  In other words, the highest weight vector in $V_{\lambda,n}^{U_n}$, with $n=\ell(\lambda)$, is invariant under $U_m$ and is acted on by $T_m$ via the weight $\lambda$, for all $m>n$.  It follows that this vector is invariant under $U$ and that $T$ acts on it through $\lambda$.  We have thus shown that $V_{\lambda}^U$ is one dimensional and $T$ acts on it through $\lambda$.  This shows that we can recover $\lambda$ from the isomorphism class of $V_{\lambda}$, which proves (b).

Statement (c) is immediate, even without having proved (a) and (b).
\end{proof}

\article[The category $\Rep^{\pol}(\GL)$]
We say that a representation of $\GL(\infty)$ is {\bf polynomial} if it appears as a subquotient of a (possibly infinite) direct sum of representations of the form $(\bC^{\infty})^{\otimes k}$.  By the previous section, every object of $\Rep^{\pol}(\GL)$ is a (possibly infinite) direct sum of objects of the form $V_{\lambda}$.  In particular, $\Rep^{\pol}(\GL)$ is a semi-simple abelian category and its simple objects are naturally indexed by partitions.  We call $\Rep^{\pol}(\GL)$ the ``$\GL$-model.''

\article[The tensor product]
We endow $\Rep^{\pol}(\GL)$ with a tensor product by using the usual tensor product of representations.  It is clear from the definition that the tensor product of two polynomial representations is again polynomial:  indeed if $V$ is a constituent of $(\bC^{\infty})^{\otimes n}$ and $W$ is a constituent of $(\bC^{\infty})^{\otimes m}$ then $V \otimes W$ is a constituent of $(\bC^{\infty})^{\otimes (n+m)}$.  The Schur--Weyl result of Proposition~\pref{ss:sw-inf} together with the derivation of \pref{ss:sw-ten} shows that the decomposition of $V_{\lambda} \otimes V_{\mu}$ into irreducibles is accomplished using the Littlewood--Richardson rule.

\article \label{art:defn:flat}
We let $T$ be the diagonal torus in $\GL(\infty)$.  A {\bf weight} of $T$ is a homomorphism $T \to \bC^\times$ (the group of nonzero complex numbers under multiplication) which depends on only finitely many matrix entries, i.e., it is of the form $\diag(a_1, a_2, \ldots) \mapsto a_1^{n_1} \cdots a_r^{n_r}$ for integers $n_1, \ldots, n_r$.  We identify weights with integer sequences which are eventually zero.  In particular, a partition defines a weight of $T$.  A weight is {\bf flat} if it consists of all 1's and 0's.  Two weights are {\bf disjoint} if their supports are disjoint.  Let $V$ be a polynomial representation of $\GL(\infty)$.  Every non-zero element $x$ of $V$ admits a unique decomposition $x=\sum_{i=1}^n x_i$ where each $x_i$ is a non-zero weight vector, and the weights of the $x_i$ are distinct.  We say that an element $x$ of $V$ is {\bf flat} if it is a weight vector and its weight is flat.  Suppose now $W$ is a second polynomial representation.  We say that two vectors $x \in V$ and $y \in W$ are {\bf disjoint} if, in the decompositions $x=\sum x_i$ and $y=\sum y_i$ into weight vectors, the weight of each $x_i$ is disjoint from the weight of each $y_j$.

\begin{Proposition}
\label{prop:disjoint}
Let $V_i$ for $1 \le i \le n$ be polynomial representations of $\GL(\infty)$, let $x_i \in V_i$ be mutually disjoint elements and let $V'_i$ be the subrepresentation of $V_i$ generated by $x_i$.  Then the subrepresentation of $V_1 \otimes \cdots \otimes V_n$ generated by $x_1 \otimes \cdots \otimes x_n$ is $V_1' \otimes \cdots \otimes V_n'$.
\end{Proposition}

\begin{proof}
By induction, we can reduce to the case $n=2$. Suppose that $V_1' = X \oplus Y$ as $\GL(\infty)$-representations, and let $\pi_X$ and $\pi_Y$ be the corresponding projections from $V_1'$. Since these projection maps are equivariant, we see that $\pi_X(x_1)$ generates $X$ and that $\pi_Y(x_1)$ generates $Y$. Also, if $\pi_1(x_1) \otimes x_2$ generates $X \otimes V'_2$ and $\pi_2(x_1) \otimes x_2$ generates $Y \otimes V'_2$, then we know that $x_1 \otimes x_2$ generates $V'_1 \otimes V'_2$. Similar remarks apply to decompositions of $V'_2$. So we can reduce to the case that $V'_1$ and $V'_2$ are both irreducible. We will make one more simplification: we will prove the result for $\GL(N)$ for all $N$ sufficiently large.

Decompose $x_1 = \sum_i v_i$ and $x_2 = \sum_j w_j$ into weight vectors. If we multiply $x_1$ and $x_2$ by enough generic elements of the maximal torus $T$, we get generic linear combinations of the $v_i$ and $w_j$, and by taking suitable linear combinations of them, we can generate the vectors $v_i \otimes w_j$ from $x_1 \otimes x_2$. Pick $v_1$ and $w_1$ two such weight vectors that appear. By applying permutations, we can assume that the support of $v_1$ is $\{1, 2, \dots, r\}$ and that the support of $w_1$ is $\{s, s+1, \dots, N\}$ where $s > r$. Now we can transform $v_1$ (and fix $w_1$) into a highest weight vector by using upper triangular matrices that fix $\{e_{r+1}, e_{r+2}, \dots, e_N\}$. Similarly, we can transform $w_1$ (and fix $v_1$) into a lowest weight vector by using lower triangular matrices that fix $\{e_1, e_2, \dots, e_{s-1}\}$. Now we appeal to the fact that the tensor product of a highest weight vector and a lowest weight vector generates $V'_1 \otimes V'_2$ \cite{tatsuuma}.
\end{proof}

\subsection{The Schur model} \label{sec:schurV}

\article
We begin by recalling some elementary terminology.  A map of finite dimensional vector spaces $f \colon V \to W$ is {\bf polynomial} if there exists bases $\{v_i\}_{1 \le i \le n}$ of $V$ and $\{w_j\}_{1 \le j \le m}$ of $W$ and polynomials $\{f_j\}_{1 \le j \le m}$ such that
\begin{equation}
f(x_1v_1+\cdots+x_n v_n) = \sum f_j(x_1, \ldots, x_n) w_j
\end{equation}
for all complex numbers $x_1, \ldots, x_n$.  A polynomial map $f$ is {\bf homogeneous of degree $n$} if each $f_i$ is.

\article
We now extend the above definitions to functors of vector spaces.  A functor $F \colon \Vec^f \to \Vec^f$ is {\bf polynomial} if for any pair of finite dimensional vector spaces $V$ and $W$, the natural map
\begin{equation}
F \colon \Hom(V, W) \to \Hom(F(V), F(W))
\end{equation}
is a polynomial map of vector spaces.  Similarly, the polynomial functor $F$ is {\bf homogeneous of degree $n$} if the above map is, for all $V$ and $W$.

\article
We denote by $\mc{S}^0$ the category of all polynomial functors $\Vec^{\fin} \to \Vec^{\fin}$ and $\mc{S}^0_n$ the subcategory of all polynomial functors which are homogeneous of degree $n$.  Both are clearly additive categories.  It is not difficult to show (see \cite[Appendix~I.A]{macdonald}) that for any $F \in \mc{S}^0$ we have a direct sum decomposition $F=\bigoplus_{n \ge 0} F_n$ with $F_n \in \mc{S}^0_n$.  Furthermore, it is clear that there are no non-zero maps between homogeneous polynomial functors of different degrees.  Therefore, in order to understand the structure of $\mc{S}^0$, it suffices to understand the structure of $\mc{S}^0_n$ for each $n$.

\article 
Let $M$ be a finite dimensional representation of $S_n$.  Define a functor $F_M \colon \Vec^{\fin} \to \Vec^{\fin}$ by
\begin{equation}
\label{eq:schur}
F_M(V) = (V^{\otimes n} \otimes M)^{S_n},
\end{equation}
where $S_n$ acts on $V^{\otimes n}$ by permuting the factors.  One easily verifies that $F_M$ is a homogeneous degree $n$ polynomial functor, and thus belongs to $\mc{S}^0_n$.  Letting $\Rep^{\fin}(S_n)$ denote the category of finite dimensional representations of $S_n$, the above construction yields a functor
\begin{equation}
\label{sw-func}
\Rep^{\fin}(S_n) \to \mc{S}^0_n, \qquad M \mapsto F_M.
\end{equation}
We then have the following important result:

\begin{theorem}[{\cite[Appendix~I.A]{macdonald}}] \label{thm:sw-func}
The functor \eqref{sw-func} is an equivalence of categories.
\end{theorem}

Obviously, we can extend $M \mapsto F_M$ to an additive functor $\Rep^{\fin}(S_{\ast}) \to \cS^0$ which is an equivalence of categories.

\article \label{art:schurfunctordefn}
We let $\bS_{\lambda}$ be the image of the irreducible $\bM_{\lambda}$ under the functor \eqref{sw-func}.  It is called the {\bf Schur functor} associated to the partition $\lambda$.  As a corollary of Theorem~\pref{thm:sw-func}, we see that $\mc{S}^0_n$ is semi-simple and the $\bS_{\lambda}$ with $\vert \lambda \vert =n$ are the simple objects.  From \pref{ss:weyl} we see that $\bS_{\lambda}(\bC^n)$ is the irreducible representation $V_{\lambda}$ of $\GL(n)$ if $\ell(\lambda) \le n$ and 0 otherwise.  The same manipulation as in \pref{ss:sw-ten} shows that tensor products of Schur functors decompose using the Littlewood--Richardson rule; in fact, these same manipulations show that $M \mapsto F_M$ is a tensor functor.

\article
The functor $F_M$ defined in \eqref{eq:schur} makes sense on infinite dimensional vector spaces as well, and defines a functor $F_M \colon \Vec \to \Vec$.  We define $\cS$ to the full subcategory of $\Fun(\Vec, \Vec)$ on objects which are isomorphic to $F_M$ with $M \in \Rep(S_{\ast})$, i.e., we allow infinite direct sums of the $\bS_{\lambda}$.  From the above results, once can show that $M \mapsto F_M$ defines an equivalence of tensor categories $\Rep(S_{\ast}) \to \cS$.  We call $\cS$ the {\bf Schur model}.  One can characterize $\cS$ as the category of functors which are, in a suitable sense, polynomial and which commute with direct limits.  An example of a functor which does not belong to $\cS$ is the double dual.

\subsection{Equivalences}

\article
We have the following fundamental theorem:

\begin{theorem}
\label{equivthm}
The following four symmetric tensor categories are equivalent:
\begin{displaymath}
\Rep(S_{\ast}), \qquad
\Vec^{\fs}, \qquad
\Rep^{\pol}(\GL), \qquad
\mc{S}.
\end{displaymath}
\end{theorem}

We have essentially proved this theorem already, as we know that each category is semi-simple and the rule for decomposing tensor products is the same in each.  However, we wish to give explicit equivalences.

\article[From fs to sequence]
We have already done this in \pref{ss:sym-fs-equiv}, but we state it again here for completeness.  There is a functor
\begin{align}
\Vec^{\fs} \to \Rep(S_{\ast}), \qquad V \mapsto (V_{[n]})_{n \ge 0}.
\end{align}
It is easily seen to be an equivalence and respect the tensor structure.

\article[From Schur to GL]
There is a natural functor
\begin{align}
\mc{S} \to \Rep^{\pol}(\GL), \qquad F \mapsto F(\bC^{\infty}).
\end{align}
This takes $\bS_{\lambda}$ to the representation $V_{\lambda}$ discussed in \pref{ss:sw-inf}.  It also commutes with infinite direct sums.  It is therefore an equivalence.  It is obviously compatible with the tensor structure.

\article[From sequence to Schur]
Let $V=(V_n)$ be an object of $\Rep(S_{\ast})$.  For a vector space $T$, put
\begin{align}
S_V(T)=\bigoplus_{n \ge 0} (T^{\otimes n} \otimes V_n)_{S_n}.
\end{align}
Then $S_V$ is an object of $\mc{S}$, and so we have a functor
\begin{align}
\Rep(S_{\ast}) \to \mc{S}, \qquad V \mapsto S_V.
\end{align}
By definition, this functor takes $\bM_{\lambda}$ to $\bS_{\lambda}$.  It also commutes with infinite direct sums.  It is therefore an equivalence by the structure of the two categories.  We have
\begin{displaymath}
\begin{split}
S_{V \otimes W}(T)
&=\bigoplus_{n \ge 0} \bigoplus_{i+j=n} (T^{\otimes (i+j)} \otimes \Ind_{S_i \times S_j}^{S_{i+j}}(V_i \otimes W_j))_{S_{i+j}} \\
&=\bigoplus_{i,j \ge 0} (T^{\otimes i} \otimes V_i)_{S_i} \otimes (T^{\otimes j} \otimes W_j)_{S_j} \\
&=S_V(T) \otimes S_W(T).
\end{split}
\end{displaymath}
In the third line we used Frobenius reciprocity \cite[\S 7.2]{serre}.  This gives a direct proof that $V \mapsto S_V$ is a tensor functor.

\article[From GL to sequence] \label{art:gl2seq}
Let $V$ be a polynomial representation of $\GL(\infty)$.  Let $V_{[n]}$ be the weight space for the weight $1^n=(1,1,\ldots,1,0,0,0,\ldots)$.  Then $S_n$ acts on $V_{[n]}$. We obtain a functor
\begin{align}
\Rep^{\pol}(\GL) \to \Rep(S_{\ast}), \qquad V \mapsto (V_{[n]})_{n \ge 0}.
\end{align}
A computation shows that this takes $V_{\lambda}$ to $\bM_{\lambda}$.  As it commutes with direct sums, it is an equivalence.  It can also be seen to be a tensor functor.

\xsection{The category $\cV$ and its properties} \label{sec:cvprop}

\subsection{The basics}
\label{ss:cvbasic}

\article[The category $\cV$]
We define $\cV$ to be any of the four categories of Theorem~\pref{equivthm}.  We prefer to think of $\cV$ and its objects abstractly, while we think of the four categories of Theorem~\pref{equivthm} as explicit ``models'' for $\cV$.  We will often switch between the various models when working in $\cV$, as some are more suited to certain tasks than others.  In the remainder of this section, we discuss other structures on $\cV$ and some other models for it.

\article[Finiteness conditions]
The category $\cV$ has arbitrary direct sums, and so some of its objects are quite large.  There are two finiteness conditions of interest to us:
\begin{itemize}
\item An object of $\cV$ has {\bf finite length} if it is a finite direct sum of simple objects.  We write $\cV_{\fin}$ for the full subcategory on the objects of finite length.
\item An object of $\cV$ is {\bf graded-finite} if every simple appears with finite multiplicity.  We write $\cV_{\gfin}$ for the full subcategory on the graded-finite objects.
\end{itemize}
These finiteness conditions admit nice descriptions in the sequence model:  an object $V$ of $\Rep(S_{\ast})$ is graded finite if and only if each $V_n$ is finite dimensional, while it has finite length if and only if it is graded finite and furthermore only finitely many $V_n$ are non-zero.

\article[Grading]
\label{ss:grade}
Let $\cV_n$ be the full subcategory of $\cV$ consisting of objects which are direct sums of simple objects of the form $\bS_{\lambda}$ with $\vert \lambda \vert=n$.  We call $\cV_n$ the ``degree $n$ piece'' of $\cV$.  The category $\cV$ decomposes as a direct sum of the $\cV_n$, meaning any object $V$ of $\cV$ admits a decomposition of the form $\bigoplus_n V_n$ where each $V_n$ belongs to $\cV_n$.  Thus every object of $\cV$ is canonically graded.  An object of $V$ is graded-finite if and only if its graded pieces are of finite length. The tensor product respects this grading, i.e., $\otimes \colon \cV_i \times \cV_j \to \cV_{i+j}$.

We now give a different description of the grading in the sequence model.  We can regard $\Rep(S_n)$ as a subcategory of $\Rep(S_{\ast})$ via extension by zero.  That is, if $V$ is a representation of $S_n$ then we put $V_n=V$ and $V_k=0$ for $k \ne n$ to obtain an object $V$ of $\Rep(S_{\ast})$.  Clearly, $\Rep(S_n)$ corresponds to $\cV_n$.

We now give yet another description of the grading, this time in the GL model.  For $z \in \bC^{\infty}$ let $[z]_n$ denote the diagonal element of $\GL(\infty)$ whose first $n$ entries are $z$ and whose remaining entries are 1.  The group $\GL(\infty)$ as we have defined it has trivial center, but the elements $[z]_n$ can be thought of as being approximately central.  One way this manifests is as follows:  if $v$ is any element of $(\bC^{\infty})^{\otimes k}$ then $[z]_nv=z^k v$ for all $n$ sufficiently large.  More generally, a polynomial representation $V$ belongs to $\cV_k$ if and only if for all $v \in V$ we have $[z]_nv=z^k v$ for $n \gg 0$.  In this way, the grading on a polynomial representation of $\GL(\infty)$ can be seen as being induced by the action of the ``center.''

\article[The objects $\bC\langle n \rangle$]
The functor $\Rep(S_{\ast}) \to \Vec$ taking $V$ to its degree $n$ piece $V_n$ is representable by an object $\bC\langle n \rangle$.  That is, there is a natural isomorphism
\begin{align}
\Hom_{\Rep(S_{\ast})}(\bC\langle n \rangle, V)=V_n.
\end{align}
The object $\bC\langle n \rangle$ is easy to describe:  it is the regular representation $\bC[S_n]$ of $S_n$ in degree $n$ and 0 in all other degrees.  In the GL model, $\bC \langle n \rangle$ is given by the object $(\bC^{\infty})^{\otimes n}$ while in the Schur model it is given by the functor $V \mapsto V^{\otimes n}$.  We have 
\begin{align}
\bC\langle n \rangle \otimes \bC \langle m \rangle=\bC\langle n+m \rangle.
\end{align}

For an object $V$ of $\cV$, we let $V \langle n \rangle$ denote the tensor product $V \otimes \bC\langle n \rangle$.  In the sequence model, $(V\langle n \rangle)_k$ is given by 
\begin{align}
(V\langle n \rangle)_k = \Ind_{S_{k-n} \times S_n}^{S_k}(V_{k-n} \otimes \bC[S_n]) = \Ind_{S_{k-n}}^{S_k}(V_{k-n}),
\end{align}
and so can be computed by repeated application of the Pieri rule (see \pref{ss:pieri}).  We regard vector spaces as degree 0 objects of $\cV$.  Thus for a vector space $U$ we have an object $U\langle n \rangle$ of $\cV$.  One should think of $\langle n \rangle$ as a sort of shift of grading functor:  if $V$ has degree $m$ then $V\langle n \rangle$ has degree $n+m$.  However, since $\bC\langle n \rangle$ has no tensor inverse, there is no way to undo this shift in grading.

\article[Transpose] \label{ss:tp1}
As we have discussed, there is a notion of transpose for partitions and this corresponds to twisting by the sign character on irreducible representations of $S_n$.  This operation carries over and defines an involution on the category $\cV$, which we denote by $V \mapsto V^{\dag}$ and call {\bf transpose}.  In the sequence model, the transpose functor is defined by $(V^{\dag})_n=V_n \otimes \sgn$.  Transpose is a tensor functor, but not a \emph{symmetric} tensor functor.  This point is very important, but somewhat subtle; an in-depth discussion is given in \S\ref{s:transp}.

\article[Duality]
Let $V$ be an object of $\Rep(S_{\ast})$.  We define the {\bf dual} of $V$, denoted $V^{\vee}$, by $(V^{\vee})_n=(V_n)^*$, where here $U^*$ denotes the dual vector space of a vector space $U$.  Duality defines a functor $\cV \to \cV^{\op}$.  There is a canonical map $V \to (V^{\vee})^{\vee}$, which is an isomorphism if (and only if) $V$ is graded-finite.  Thus duality provides an equivalence
\begin{align}
(\cdot)^\vee \colon \cV_{\gfin} \xrightarrow{\cong} \cV_{\gfin}^{\op}.
\end{align}
Duality interacts well with the tensor product on graded-finite objects:  if $V$ and $W$ are graded finite then the natural map 
\begin{align}
V^{\vee} \otimes W^{\vee} \xrightarrow{\cong} (V \otimes W)^{\vee}
\end{align}
is an isomorphism.  Note that duality does not change the isomorphism class of simple objects; more generally, if $V$ is graded-finite then $V$ and $V^{\vee}$ are isomorphic, though not in any canonical way.

Duality can be described in the $\GL$-model as follows:
\begin{align}
V^{\vee}=\Hom_{\GL(\infty)}(V, \Sym(\bC^{\infty} \otimes \bC^{\infty})).
\end{align}
Here we regard $\GL(\infty) \times \GL(\infty)$ as acting on $\bC^{\infty} \otimes \bC^{\infty}$ and we take maps which are equivariant with respect to the first $\GL$; the second copy of $\GL$ then acts on the $\Hom$ space.  To see that this coincides with the definition given in the sequence model, note that
\begin{displaymath}
\Sym(\bC^{\infty} \otimes \bC^{\infty})=\bigoplus \bS_{\lambda}(\bC^{\infty}) \otimes \bS_{\lambda}(\bC^{\infty}),
\end{displaymath}
where the sum is over all partitions (this identity is discussed further in \pref{ss:cauchy}).  Thus 
\begin{align}
\bS_{\lambda}(\bC^{\infty})^{\vee} = \bS_{\lambda}(\bC^{\infty}).
\end{align}
This shows that $V \mapsto V^{\vee}$ is a contravariant functor which does not change the isomorphism class of simple objects; it must therefore coincide with the functor constructed in the sequence model.

\begin{remark}
If $V$ is a polynomial representation of $\GL(\infty)$ then $V^*$ is (in general) \emph{not} a polynomial representation of $\GL(\infty)$.  In particular, the duality discussed above is not the same as taking the linear dual.
\end{remark}

\subsection{Co-addition and co-multiplication}

\article[Tensor powers of $\cV$]
One can make sense of the tensor product of two abelian categories, at least under certain assumptions; see \cite[\S 5]{Deligne} for a general discussion. Given rings $R, S$, this tensor product behaves well with respect to module categories: $\Mod_R \otimes \Mod_S \cong \Mod_{R \otimes S}$. The tensor power $\cV^{\otimes r}$ exists, and admit models similar to those of $\cV$:
\begin{itemize}
\item The sequence model consists of families of vector spaces $(V_{n_1, \ldots, n_r})$ indexed by elements of $\bZ_{\ge 0}^r$ such that $V_{n_1, \ldots, n_r}$ is equipped with an action of $S_{n_1} \times \cdots \times S_{n_r}$.
\item The fs-model consists of functors $\fs^r \to \Vec$, i.e., functors which take $r$ finite sets and yield a vector space.
\item The GL-model consists of polynomial representations of the group $\GL(\infty)^r$.  Such representations can be described as those which appear as a constituent of a representation of the form $(\bC^{\infty})^{\otimes n_1} \otimes \cdots \otimes (\bC^{\infty})^{\otimes n_r}$, or a direct sum of such representations.
\item The Schur model consists of polynomial functors $\Vec^r \to \Vec$.  Here a functor $F$ is polynomial if $F(V_1, \ldots, V_r)$ is a polynomial representation of $\GL(V_1) \times \cdots \times \GL(V_r)$ (in the evident sense) whenever the $V_i$'s are finite dimensional, and $F$ satisfies a certain continuity condition.
\end{itemize}
The category $\cV^{\otimes r}$ is semi-simple and its simple objects are all external tensor products of simple objects of $\cV$.  In the GL-model this just means that every polynomial representation of $\GL(\infty)^r$ decomposes as a direct sum of irreducible representations, and that each simple is a tensor product of irreducible representations of $\GL(\infty)$.  In the Schur model, this amounts to the slightly less obvious statement that a polynomial functor $F\colon \Vec^r \to \Vec$ admits a decomposition of the form
\begin{align}
F(V_1, \ldots, V_r)=\bigoplus_{i \in I} \bS_{\lambda_{i,1}}(V_1) \otimes \cdots \otimes \bS_{\lambda_{i,r}}(V_r)
\end{align}
for some index set $I$ and partitions $\lambda_{i,j}$.

\article
In what follows, we write $F \uotimes G$ for the object in $\cV^{\otimes 2}$ given by the external tensor product of $F$ and $G$ in $\cV$.  Thus, in the Schur model, $F$ and $G$ are functors $\Vec \to \Vec$ while $F \uotimes G$ is the functor $\Vec^2 \to \Vec$ given by $(V, W) \mapsto F(V) \otimes G(W)$.

\article[Co-addition] \label{ss:coadd}
Let $F$ be a polynomial functor $\Vec \to \Vec$.  Then the functor $\Vec^2 \to \Vec$ given by $(V, W) \mapsto F(V \oplus W)$ is a polynomial functor, and thus defines an object of $\cV^{\otimes 2}$, which we denote by $a^*F$.  We thus have a functor
\begin{align}
a^*\colon \cV \to \cV^{\otimes 2}
\end{align}
which we call {\bf co-addition}.

\article
On simple objects, co-addition is computed using the Littlewood--Richardson rule:
\begin{align}
a^*(\bS_{\lambda})=\bigoplus_{\mu,\nu} (\bS_{\mu} \uotimes \bS_{\nu})^{\oplus c_{\mu,\nu}^{\lambda}},
\end{align}
where the sum is over partitions $\mu$, $\nu$ with $\vert \mu \vert+\vert \nu \vert=\vert \lambda \vert$.  This is the identity
\begin{displaymath}
\bS_{\lambda}(V \oplus W)
= \bigoplus_{\mu,\nu} (\bS_{\mu}(V) \otimes \bS_{\nu}(W))^{\oplus c_{\mu,\nu}^{\lambda}}
\end{displaymath}
which is discussed in \pref{ss:GLbranching}.

\article
One can also see co-addition easily in the fs-model.  Disjoint union provides a functor $\fs^2 \to \fs$.  Thus, given a functor $F\colon \fs \to \Vec$ we obtain a functor $\fs^2 \to \Vec$ by composing with disjoint union; the result is $(L, L') \mapsto F_{L \amalg L'}$.  This defines a map $\cV=\Vec^{\fs} \to \Vec^{\fs^2}=\cV^{\otimes 2}$ which coincides with the co-addition map discussed above.

\article
Let us give an important example.  The only way $\bS_{\lambda} \otimes \bS_{\mu}$ can contain a copy of $\Sym^n$ is if $\bS_{\lambda}=\Sym^i$ and $\bS_{\mu}=\Sym^j$ and $i+j=n$.  (Reason: the Littlewood--Richardson rule shows that all constituents of a tensor product have at least the number of rows of the factors, so the only way to get a constituent with one row is it both factors have only one rows.)  In other words, the coefficient $c^{(n)}_{\mu,\nu}$ is non-zero only when $\mu=(i)$ and $\nu=(j)$ with $i+j=n$; it is then equal to 1.  This gives a ``binomial theorem'' (the ``binomial theorem'' for exterior powers is obtained in a similar way)
\begin{align}
\Sym^n(V \oplus W)&=\bigoplus_{i+j=n} \Sym^i(V) \otimes \Sym^j(W),\\
\bigwedge^n(V \oplus W)&=\bigoplus_{i+j=n} \bigwedge^i(V) \otimes \bigwedge^j(W).
\end{align}
These can also be proven directly without any prior knowledge about Schur functors.

\article[Co-multiplication] \label{ss:comult}
The same discussion as in \pref{ss:coadd} applies if we use tensor product instead of direct sum.  We thus obtain a {\bf co-multiplication} map
\begin{align}
m^*\colon \cV \to \cV^{\otimes 2}.
\end{align}
In the Schur model $m^*$ is given by $(m^*F)(V,W)=F(V \otimes W)$.  
On simple objects, co-multiplication is computed using the Kronecker coefficients $g_{\lambda,\mu,\nu}$ (see \pref{ss:sym-ten}).  Precisely,
\begin{align}
m^*(\bS_{\lambda})=\bigoplus_{\mu,\nu} (\bS_{\mu} \uotimes \bS_{\nu})^{\oplus g_{\lambda,\mu,\nu}},
\end{align}
where the sum is over partitions $\mu$, $\nu$ of the same size as $\lambda$. This was discussed in \pref{ss:kroneckerbranch}.

\article
\label{schur:cauchy}
Also, as discussed in \pref{ss:cauchy}, we get the following two Cauchy identities
\begin{align*}
\Sym^n(V \otimes W)&=\bigoplus_{\lambda \vdash n} \bS_{\lambda}(V) \otimes \bS_{\lambda}(W),\\
\bw{n}(V \otimes W)&=\bigoplus_{\lambda \vdash n} \bS_{\lambda}(V) \otimes \bS_{\lambda^{\dag}}(W).
\end{align*}

\subsection{Composition (plethysms)} \label{ss:comp}

\article
The composition of two polynomial functors $\Vec \to \Vec$ is again a polynomial functor.  We thus obtain a functor
\begin{align}
\cV \times \cV \to \cV, \qquad (F, G) \mapsto F \circ G,
\end{align}
which we call {\bf composition}.

\article
Determining the decomposition of the composition of two simple objects is known as the {\bf plethysm problem} and is notoriously difficult.  We mention some known results.  The decomposition of $\bS_{\lambda}^{\otimes 2}$ is known by the Littlewood--Richardson rule.  Determining $\Sym^2(\bS_{\lambda})$ and $\bw{2}(\bS_{\lambda})$ amounts to understanding the action of $S_2$ on the multiplicity spaces of the Littlewood--Richardson rule.  This has been done in \cite{carreleclerc}, and so these plethysms are known. The case when $\lambda$ is $(n)$ or $(1^n)$ has a simpler formula \cite[Example I.8.9(a)]{macdonald}. Formulas for $\bS_\lambda \circ \Sym^n$ when $|\lambda|=3$ are given in \cite[Example I.8.9(b)]{macdonald}. As far as we are aware, the only plethysms $\bS_{\lambda} \circ \bS_{\mu}$ that are known when $\lambda$ is large are when both $\bS_{\lambda}$ and $\bS_{\mu}$ are either symmetric or wedge powers and $\vert \mu \vert=2$.  For example, we have
\begin{align*}
\Sym^n \circ \Sym^2=\bigoplus_{\lambda \vdash n} \bS_{2\lambda}.
\end{align*}
For this formula and the other 3 variations, see \cite[Example I.8.6]{macdonald}. In particular, these plethysms are multiplicity-free.  This is not true for general plethysms.

\article
Composition can also be seen in the sequence model.  Let us restrict our attention to two simple objects $\bM_{\lambda}$ and $\bM_{\mu}$ of degrees $n$ and $m$.  Then the composition is given by
\begin{align}
\bM_{\lambda} \circ \bM_{\mu}=\Ind_{S_n \rtimes S_m^n}^{S_{nm}} (\bM_{\lambda} \otimes \bM_{\mu}^{\otimes n}).
\end{align}
Let us explain what this equation means.  First, think of a set with $nm$ elements organized into $n$ columns, each with $m$ elements.  Then $S_m^n$ acts by permuting the elements within each column, while $S_n$ acts by permuting the columns themselves.  The action of $S_n$ normalizes that of $S_m^n$, and this realizes the semi-direct product $S_n \rtimes S_m^n$ as a subgroup of $S_{nm}$.  The space $\bM_{\mu}^{\otimes n}$ is a representation of $S_m^n$ and extends to a representation of $S_n \rtimes S_m^n$ with $S_n$ acting by permuting the factors.  The space $\bM_{\lambda}$ is a representation of $S_n \rtimes S_m^n$ via the projection from this group to $S_n$.

\begin{Example}
We see that
\begin{displaymath}
\bM_{(n)} \circ \bM_{(2)}=\Ind_{S_n \rtimes S_2^n}^{S_{2n}}(1).
\end{displaymath}
Let $\mc{M}_{2n}$ denote the set of perfect undirected matchings on $2n$ vertices.  (Such a matching is a graph in which each vertex belongs to exactly one edge.)  The group $S_{2n}$ acts transitively on the set of matchings, and the stabilizer of any element is a subgroup of the form $S_n \rtimes S_2^n$.  Thus the above induction is equivalent to the permutation representation on $\mc{M}_{2n}$.  Combining this discussion with the formula for $\Sym^n \circ \Sym^2$ given earlier (which is just $\bM_{(n)} \circ \bM_{(2)}$ in the Schur model), we see that
\begin{displaymath}
\mc{M}_{2n}=\bigoplus_{\lambda \vdash n} \bM_{2\lambda}.
\end{displaymath}
See also \cite[Example 7.A2.9]{stanley}.
\end{Example}

\article
Composition can be stated in the fs-model as follows:
\begin{displaymath}
(F \circ G)_L=\bigoplus \left[ F_I \otimes \bigotimes_{i \in I} G_{U_i} \right].
\end{displaymath}
Here the sum is over all partitions $\cU=\{U_i\}_{i \in I}$ of the set $L$.  To be more careful, we should sum over all equivalence relations $\sim$ on $L$, let $U_i$ be the equivalence classes and let $I=L/\sim$.

\subsection{The Schur derivative} \label{ss:deriv}

\article
Let $F$ be a polynomial functor $\Vec \to \Vec$.  We define the {\bf Schur derivative} of $F$, denote $\bD(F)$, to be the functor $\Vec \to \Vec$ given by
\begin{align}
(\bD F)(V)=F(V \oplus \bC)^{(1)},
\end{align}
where the superscript denotes the subspace on which $\bC^{\times}$ acts through its standard character.  In other words, one expands $F(V \oplus \bC)$ (which can be done using the Littlewood--Richardson rule, see \pref{ss:coadd}) and then takes the $\bC^\times$-isotypic component of $\bS_{(1)}(\bC)$.  From this description, it is evident that $\bD \bS_{\lambda}$ is computed by the Pieri rule; in fact,
\begin{align}
\bD \bS_{\lambda}=\bigoplus_{\mu,\, \lambda/\mu \in \HS_1} \bS_{\mu}
\end{align}
For example, we have $\bD(\Sym^n)=\Sym^{n-1}$, and so $\bD(\Sym)=\Sym$.  These identities are manifestations of the analogies between $\Sym^n$ and the divided power $\gamma_n(x)=x^n/n!$, and between the symmetric algebra $\Sym$ and the exponential function.

\article
The Schur derivative can also be seen easily in the sequence model.  If $V$ is an object of $\Rep(S_{\ast})$ then $\bD V$ is the object with $(\bD V)_n=V_{n+1}$.  That is, one simply restricts the representation $V_{n+1}$ from $S_{n+1}$ to $S_n$.  From this point of view, it is clear that $\bD$ is adjoint (both left and right) to the shift functor $\langle 1 \rangle$.  Thus $\bD$ is the closest thing there is to an inverse to $\langle 1 \rangle$.

\article
The Schur derivative satisfies many of the usual properties of the derivative. 
\begin{itemize}
\item Additivity:
$\bD(F \oplus G)=\bD(F) \oplus \bD(G)$
\item Leibniz rule:
$\bD(F \otimes G)=(\bD(F) \otimes G) \oplus (F \otimes \bD(G))$
\item Chain rule:
$\bD(F \circ G)=(\bD(F) \circ G) \otimes \bD(G)$
\end{itemize}

\article[Higher derivatives]
There are also ``higher Schur derivatives.'' For a partition $\lambda$ and a polynomial functor $F$, we define $\bD_{\lambda}(F)$ to be the coefficient of $\bS_{\lambda}$ in $a^*(F)$.  Thus $\bD=\bD_{(1)}$.  The co-addition formula in \pref{ss:coadd} shows that
\begin{align}
\bD_\nu \bS_\lambda = \bigoplus_\mu \bS_\mu^{\oplus c^\lambda_{\mu, \nu}}.
\end{align}
So we can alternatively define $\bD_\lambda$ as the adjoint of the functor $F \mapsto F \otimes \bS_\lambda$. More precisely,
\begin{align}
\Hom_{\cV}(\bD_\lambda F, G) = \Hom_{\cV}(F, G \otimes \bS_\lambda).
\end{align}
This implies immediately that the Schur derivatives commute: 
\begin{align}
\bD_\lambda  \circ \bD_\mu = \bD_\mu \circ \bD_\lambda.
\end{align}
In the theory of symmetric functions, $\bD_\lambda$ corresponds to the ``skewing operator.'' The higher Schur derivatives $\bD_{\lambda}$ behave like higher order derivatives, from the point of view of formal properties. In general, we have
\begin{align}
\bD_\nu(F \otimes G) &= \bigoplus_{\lambda, \mu} (\bD_\lambda(F) \otimes \bD_\mu(G))^{\oplus c^\nu_{\lambda, \mu}}.
\end{align}
This follows from the formula for co-addition from \pref{ss:coadd} (see also \cite[Example I.5.25(d)]{macdonald}).
For example, 
\begin{displaymath}
\bD_{(2)}(F \otimes G)=(\bD_{(2)}(F) \otimes G) \oplus (\bD(F) \otimes \bD(G)) \oplus (F \otimes \bD_{(2)}(G)).
\end{displaymath}
In addition, recall that by Schur--Weyl duality, we have $\bS_1^{\otimes k} = \bigoplus_\lambda \bM_\lambda \boxtimes \bS_\lambda$ as representations of $S_k \times \GL(\infty)$. So the higher Schur derivatives give us a decomposition of the iterated derivative
\begin{align}
\bD^{\circ k} = \sum_\lambda \dim(\bM_\lambda) \bD_\lambda.
\end{align}
In the sequence model, $(\bD_{\lambda} V)_n$ is given by $\Hom_{S_k}(\bM_{\lambda}, V_{n+k})$, where $k=\vert \lambda \vert$.

\subsection{Further remarks}

\article[Universal property] \label{ss:univ}
Let $\mc{A}$ be a $\bC$-linear symmetric tensor category (definition reviewed in \S\ref{ss:tensorcat}).  Let $A$ be an object of $\mc{A}$.  Since the tensor structure is symmetric, $S_n$ acts on $A^{\otimes n}$.  We define
\begin{align}
\bS_{\lambda}(A)=(A^{\otimes n} \otimes \bM_{\lambda})_{S_n}.
\end{align}
One can verify that this defines a functor
\begin{align}
\cV \times \mc{A} \to \cA, \qquad (\bS_{\lambda}, A) \mapsto \bS_{\lambda}(A).
\end{align}
In the case where $\mc{A}$ is $\cV$, the above map is the composition map discussed in \S \ref{ss:comp}.

Using this action, $\cV$ can be characterized as the universal $\bC$-linear tensor category.  That is, to give a symmetric tensor functor from $\cV$ to an arbitrary symmetric tensor category $\mc{A}$ is the same as to give an object of $\mc{A}$.  In other words, the functor
\begin{align}
\Fun^{\otimes}(\cV, \mc{A}) \to \mc{A}, \qquad F \mapsto F(\bC\langle 1 \rangle)
\end{align}
is an equivalence.  The discussion of the previous paragraph shows that this functor is essentially surjective:  given an object $A$ of $\mc{A}$ we obtain a functor $F\colon \cV \to \mc{A}$ by $F(\bS_{\lambda})=\bS_{\lambda}(A)$.  This is clearly a tensor functor and satisfies $F(\bC\langle 1 \rangle)=A$.  We leave the remainder of the proof that the above functor is an equivalence to the interested reader.

\article[The category $\Sym(\Vec)$]
A word of warning:  this section is not meant to be entirely rigorous.  We have not developed the necessary foundations to rigorously prove the assertions we make here.  We will not use any of these statements in what follows; however, we believe they offer a useful picture to keep in mind.

We now give another description of $\cV$:  namely, it can be identified with the symmetric algebra on $\Vec$.  Let us make sense of this statement in two ways.  First, $\Sym(\Vec)$ should mean the universal $\bC$-linear symmetric tensor category equipped with an additive functor from $\Vec$.  If $\mc{A}$ is an arbitrary $\bC$-linear symmetric tensor category, then giving a symmetric tensor functor $\Sym(\Vec) \to \mc{A}$ is the same as giving an additive functor $\Vec \to \mc{A}$, which is the same as giving an object of $\mc{A}$.  Thus $\Sym(\Vec)$ satisfies the same universal property as $\cV$, and so the two are equivalent.

We now give a different explanation of the equivalence.  In \pref{ss:univ}, we saw that $\cV$ acts on any $\bC$-linear symmetric tensor category.  We believe that one should regard this action as analogous to a divided power structure.  Thus every $\bC$-linear symmetric tensor category comes with a canonical divided power structure.  In particular, the symmetric algebra and divided power algebra on a $\bC$-linear abelian category are the same.  We can therefore think of $\Sym(\Vec)$ as $\bigoplus_{n \ge 0} (\Vec^{\otimes n})^{S_n}$ (this is how the divided power algebra is constructed).  Now, $\Vec^{\otimes n}=\Vec$, and $\Vec^{S_n}=\Rep(S_n)$ (the left side means $S_n$-equivariant objects in $\Vec$; since the action is trivial this reduces to representations of $S_n$).  We thus see that $\Sym(\Vec)$ is the direct sum of the categories $\Rep(S_n)$, i.e., $\Rep(S_{\ast})$.

\article[{Analogy with $\bC[t]$}]
The above structure on $\cV$ puts it in close analogy with $\bC[t]$.  Indeed, $\bC[t]$ is a ring (analogous to a tensor category) which has natural co-addition and co-multiplication maps (coming from the ring structure on the additive group $\Spec(\bC[t])=\bG_a$).  Furthermore, the ring $\bC[t]$ has a notion of composition and derivative.  From this point of view, the homogeneous pieces of an object of $\cV$ correspond to the coefficients of a polynomial in $\bC[t]$.  This can be a useful way to think about the operations (composition, tensor product, etc.) in $\cV$.

The constructions of $\bC[t]$ and $\cV$ can be made to look similar as well:  $\bC[t]$ is obtained by taking the ``unit vector space'' $\bC$ (i.e., the identity for the tensor product) and applying the ``exponential'' $\Sym$.  Similarly, $\cV$ is obtained by taking the ``unit abelian category'' $\Vec$ and applying $\Sym$.  In fact, from this point of view, there is one more object that fits into the analogy:  the number $e$ is obtained by taking the unit complex number 1 and applying the exponential function.  We therefore have a sequence of objects
\begin{displaymath}
e, \bC[t], \cV,
\end{displaymath}
each obtained by the same construction but at a different categorical level.  Does it continue in a meaningful way?

\section{Symmetric structures on tensor categories}
\label{s:transp}

To this point, though we have discussed tensor products on various abelian categories, we have avoided the technical formalism of general abelian tensor categories.  We now discuss this some, as the tensor properties of the transpose functor are somewhat subtle:  it is a tensor functor but is not a symmetric tensor functor.  In fact, the tensor product on $\cV$ admits two natural symmetric structures, and the transpose functor interchanges them.

\subsection{Symmetric tensor categories}
\label{ss:tensorcat}

\article
A {\bf tensor category} is an abelian category $\mc{A}$ equipped with a bi-additive functor
\begin{align}
\otimes \colon  \mc{A} \times \mc{A} \to \mc{A}.
\end{align}
This functor is required to be associative and have a unit object.  However, those requirements need not hold in the strict sense, but only up to isomorphism.  For instance, the associativity ``condition'' means that for every triple of objects $(A, B, C)$ in $\mc{A}$ we are given an isomorphism
\begin{align}
\alpha_{A,B,C}\colon (A \otimes B) \otimes C \to A \otimes (B \otimes C).
\end{align}
These isomorphisms must define a natural transformation in a suitable sense, and must satisfy an additional ``coherence'' condition called the pentagon axiom \cite[\S VII.1]{maclane}.  Similarly, the existence of a unit means there is an object 1 of $\mc{A}$ and isomorphisms $\beta_A \colon A \otimes 1 \to A$ and $\gamma_A \colon 1 \otimes A \to A$.  There are additional compatibilities between the $\alpha$, $\beta$ and $\gamma$.  All this data is not merely required to exist, but is part of the data of a tensor category.  That is, a tensor category is a tuple $(\mc{A}, \otimes, 1, \alpha, \beta, \gamma)$.

\article
Suppose $\mc{A}$ and $\mc{B}$ are tensor categories.  A tensor functor between them is an additive functor $F \colon \mc{A} \to \mc{B}$ which commutes with the tensor product.  Again, this is not a condition but extra data:  we require a functorial isomorphism 
\begin{align} \label{eqn:tensorfunctorisom}
F(A \otimes B) \xrightarrow{\cong} F(A) \otimes F(B)
\end{align}
which is compatible with all the data on each side in certain ways.  This isomorphism is part of the data of the tensor functor.

\article
The definition of tensor category has no requirement that the product $\otimes$ be commutative in any sense:  it is perfectly possible that there are objects $A$ and $B$ for which $A \otimes B$ and $B \otimes A$ are non-isomorphic.  Since many tensor categories have a commutative tensor product, we would like a definition that captures this.  Clearly, the first thing to ask for is the existence of a functorial isomorphism $\tau_{A,B}\colon A \otimes B \to B \otimes A$, compatible with the additional structure in certain ways.  This gives the notion of a {\bf braided tensor category}.  However, this definition does not enforce one of the basic rules we are accustomed to when dealing with tensors:  namely, that the switching-of-factors map $\tau$ be an involution.  There are interesting braided tensor categories where this is not the case, however, we will not discuss such things and so impose this condition.

\article
(Somewhat) formally, a {\bf symmetric tensor category} is a pair $(\mc{A}, \tau)$ where $\mc{A}$ is a tensor category and $\tau$ associates to every pair of objects $(A, B)$ an isomorphism 
\begin{align}
\tau_{A,B}\colon A \otimes B \xrightarrow{\cong} B \otimes A
\end{align}
in such a way that $\tau_{B,A} \circ \tau_{A,B}$ is the identity on $A \otimes B$, and $\tau$ interacts with the other structures on $\mc{A}$ (i.e., the $\alpha$, $\beta$, etc.) in the appropriate manner.  When the base tensor category $\cA$ is fixed, we refer to $\tau$ as a {\bf symmetric structure} on it.

\article
Let $(\mc{A}, \tau)$ and $(\mc{B}, \sigma)$ be symmetric tensor categories.  A {\bf symmetric tensor functor} is a tensor functor $F\colon \mc{A} \to \mc{B}$ such that the diagram
\begin{align}
\xymatrix{
F(A \otimes B) \ar[r] \ar[d]_{F(\tau_{A,B})} & F(A) \otimes F(B) \ar[d]^{\sigma_{F(A),F(B)}} \\
F(B \otimes A) \ar[r] & F(B) \otimes F(A) }
\end{align}
commutes for all $A$ and $B$.  Here the horizontal maps are the isomorphisms \eqref{eqn:tensorfunctorisom}.  Note that a symmetric tensor functor is just a tensor functor satisfying a condition:  it does not have any additional data associated to it.

\article
The above definitions are very complicated.  Fortunately, none of the inner workings of tensor categories will be relevant to us:  it is only the commutativity isomorphism $\tau$ that we will ever need to think about.

\article
Here is one concrete way in which the symmetric condition on $\tau$ comes in to play:  in a symmetric tensor category, the symmetric group $S_n$ acts on $A^{\otimes n}$ for any object $A$.  In a braided tensor category, this is not the case: only the braid group acts.  This is why we needed $\mc{A}$ to be a symmetric tensor category in \pref{ss:univ} to get an action of $\cV$.  If $F\colon \mc{A} \to \mc{B}$ is a symmetric tensor functor then the isomorphism $F(A^{\otimes n}) \to F(A)^{\otimes n}$ is $S_n$-equivariant.  It follows that $F$ is compatible with the actions of $\cV$ on each side; in particular, we have a natural isomorphism 
\begin{align}
F(\bS_{\lambda}(A)) \xrightarrow{\cong} \bS_{\lambda}(F(A)).
\end{align}

\subsection{Algebras and modules}
\label{ss:alg}

\article
While on the subject of tensor categories, we make a slight digression that will be relevant later on.  Let $\mc{A}$ be a tensor category.  An {\bf algebra} in $\mc{A}$ is an object $A$ equipped with a multiplication map $A \otimes A \to A$.  One can make sense of what it means for $A$ to be associative and unital.  These are conditions on  $A$ and not extra data.  All algebras we are interested in will satisfy both conditions.  However, one cannot make sense of what it means for $A$ to be commutative, in general.

\article
Suppose now that $\mc{A}$ is a \emph{symmetric} tensor category.  Then one \emph{can} make sense of the notion of commutativity for an algebra in $A$.  Namely, the multiplication map $A \otimes A \to A$ should be invariant under the action of $S_2$ on the source; remember, this action only exists because of the symmetric hypothesis.

\article
Given an associative unital algebra $A$ in a tensor category (not necessarily symmetric), a left $A$-module is an object $M$ of $\mc{A}$ equipped with a map $A \otimes M \to M$ satisfying the usual axioms.  There is an obvious notion of a map of left $A$-modules.  We write $\Mod_A$ for the category of left $A$-modules; it is an abelian category.  When $\mc{A}$ is symmetric and $A$ is commutative, left and right modules coincide (so we will drop the word ``left'') and the tensor product on $\mc{A}$ induces one on $\Mod_A$.

\subsection{An example}

\article
Let $\mc{A}$ be the category of $\bZ_{\ge 0}$-graded vector spaces.  We regard vector spaces as being objects in $\mc{A}$ of degree 0, and write $[n]$ for the grade shift functor, so that $\bC[n]$ is supported in degree $n$.  For two graded vector spaces $V$ and $W$ we define $V \otimes W$ to be the usual tensor product of $V$ and $W$ graded in the usual manner:
\begin{align}
(V \otimes W)_n=\bigoplus_{i+j=n} V_i \otimes W_j.
\end{align}
Thus $\bC[n] \otimes \bC[m]=\bC[n+m]$, as usual.  There are natural and obvious choices for all the extra structure required to make $\mc{A}$ a tensor category under $\otimes$.  From here on we regard $\mc{A}$ as a tensor category.

\article
The tensor category $\mc{A}$ admits an obvious symmetric structure $\tau$, defined by
\begin{align}
\tau_{V,W}\colon V \otimes W \to W \otimes V, \qquad \tau_{V,W}(v \otimes w)=w \otimes v.
\end{align}
This clearly satisfies the axiom to be symmetric, i.e., $\tau_{W,V} \tau_{V,W}$ is the identity map on $V \otimes W$.  In the formalism discussed in \S \ref{ss:alg}, commutative algebras in $(\mc{A}, \tau)$ are commutative graded rings in the usual sense.  For instance, $\Sym^k(\bC^n[1])=\Sym^k(\bC^n)[k]$, so $\Sym(\bC^n[1])$ is the standard commutative polynomial ring $\bC[x_1,\ldots, x_n]$ where each $x_i$ has degree 1.

\article
The tensor category $\mc{A}$ admits a second symmetric structure $\sigma$, defined by
\begin{align}
\sigma_{V,W}\colon V \otimes W \to W \otimes V, \qquad \sigma_{V,W}(v \otimes w)=(-1)^{\deg(v)\deg(w)} w \otimes v.
\end{align}
(Obviously this formula is valid only for homogeneous elements $v$ and $w$, and the map $\sigma_{V,W}$ is extended linearly.)  Again, it is clear that $\sigma$ is a symmetric.  With this symmetric structure, algebras in $(\mc{A}, \sigma)$ correspond to \emph{graded-commutative} rings.  For instance, $\Sym^k(\bC^n[1])=\bw{k}(\bC^n)[k]$, so $\Sym(\bC^n[1])$ is the exterior algebra on $n$ generators of degree 1.  The symmetric structure $\sigma$ is typically used (often tacitly) when equipping the category of chain complexes with a symmetric tensor product.

\begin{Proposition}
The two symmetric tensor categories $(\mc{A}, \tau)$ and $(\mc{A}, \sigma)$ are not equivalent.
\end{Proposition}

\begin{proof}
Suppose we had an equivalence $F\colon (\mc{A},\tau) \to (\mc{A},\sigma)$.  Since $F$ is an equivalence of the underlying tensor categories, we have $F(\bC)=\bC$ and
\begin{displaymath}
F(\bC[n])=F((\bC[1])^{\otimes n})=F(\bC[1])^{\otimes n}
\end{displaymath}
From the above it is clear that we must have $F(\bC[1])=\bC[1]$, otherwise $F$ could not be essentially surjective.  The above identity then shows that $F(\bC[n])=\bC[n]$ and so $F(V)=V$ for every vector space $V$.  Since $F$ is a symmetric tensor functor it is compatible with Schur functors.  However, we know that Schur functors act differently in the two structures:  for instance, $\Sym^2(\bC[1])$ is non-zero in $(\mc{A}, \tau)$ but vanishes in $(\mc{A}, \sigma)$.  This is a contradiction, so no such equivalence $F$ can exist.
\end{proof}

\subsection{Symmetric structures on $\cV$ and the transpose}

\article
We have defined the category $\cV$ to be any of various equivalent tensor categories.  In fact, each of those tensor categories admits an obvious symmetric structure $\tau$; for instance, in the $\GL$-model $\tau$ is just the usual isomorphism $V \otimes W \to W \otimes V$ of representations.  The various equivalences of tensor categories respect these structures, and so $\tau$ defines a symmetric structure on $\cV$, making it into a symmetric tensor category.

\article
As in the example discussed in the previous section, we can define an alternate symmetric structure $\sigma$ on $\cV$ introducing some signs.  Precisely, for $V$ and $W$ in $\Rep^{\pol}(\GL)$, put
\begin{align}
\sigma_{V,W}\colon V \otimes W \to W \otimes V, \qquad \sigma_{V,W}(v \otimes w)=(-1)^{\deg(v)\deg(w)} w \otimes v.
\end{align}
Here degree is defined using the canonical grading on objects of $\cV$; see \pref{ss:grade}.  One can easily see $\sigma_{V,W}$ in the other models as well.

\begin{Proposition}
\label{prop:transp}
Transpose is an equivalence of symmetric tensor categories $(\cV, \tau) \to (\cV, \sigma)$.
\end{Proposition}

This is our main result on the transpose functor.  In particular, it shows that $(\cV, \tau)$ and $(\cV, \sigma)$ are equivalent as symmetric tensor categories, in contrast to what we saw in the previous section.  The proposition can be proved directly by manipulations in the sequence model.  However, we prefer to give a clearer proof in the fs-model, and require some preliminary discussion.

\article
Let $A$ be an algebra in $\cV$.  We say that $A$ is commutative if it is so with respect to $\tau$, and graded-commutative if it is commutative with respect to $\sigma$.  The above result shows that the theories of commutative and graded-commutative algebras in $\mc{A}$ are completely equivalent.  This is somewhat amazing, as it is not true in the case of graded vector spaces.

\article
Let $L$ be a finite set.  Define the {\bf orientation space} of $L$, denoted $\Or_L$, to be the determinant (top exterior power) of $\bC[L]$, the vector space with basis $L$.  Thus $\Or_L$ is a one-dimensional vector space on which $\Aut(L)$ acts through the sign character; the space $\Or_L$ does not have a canonical basis.  It is clear that $\Or_L$ is functorial in $L$, and so defines an object $\Or$ of $\Vec^{\fs}$.  There is a natural isomorphism
\begin{equation}
\label{eq1}
i_{L,L'} \colon  \Or_L \otimes \Or_{L'} \xrightarrow{\cong} \Or_{L \amalg L'}
\end{equation}
given by concatenating wedge products, putting $L$ before $L'$.  The diagram
\begin{align}
\xymatrix{
\Or_L \otimes \Or_{L'} \ar[r]^-{i_{L,L'}} \ar[d] & \Or_{L \amalg L'} \ar[d] \\
\Or_{L'} \otimes \Or_L \ar[r]^-{i_{L',L}} & \Or_{L' \amalg L} }
\end{align}
is commutative up to sign:  the two paths differ by $(-1)^{(\#L)(\#L')}$.  This can be re-expressed as follows.  The map \eqref{eq1} yields a multiplication map $\Or \otimes \Or \to \Or$ and the commutativity-up-to-sign of the above diagram is the same as saying that this multiplication is graded-commutative.

\article
We now have the following formula for transpose in the fs-model:
\begin{align}
V^{\dag}=V \boxtimes \Or.
\end{align}
With this description, it is easier to verify the tensor properties of transpose.  

\article
We now prove Proposition~\pref{prop:transp}.  We have
\begin{displaymath}
\begin{split}
(V^{\dag} \otimes W^{\dag})_L
&=\bigoplus_{L=A \amalg B} V_A \otimes \Or_A \otimes W_B \otimes \Or_B \\
&=\bigoplus_{L=A \amalg B} V_A \otimes W_B \otimes \Or_L \\
&=(V \otimes W)_L \boxtimes \Or_L=(V \otimes W)^{\dag}_L
\end{split}
\end{displaymath}
To go from the first to the second line, we used the isomorphism $i_{A,B}$.  Call the composite isomorphism $f_{V,W}^L$.  These isomorphisms are functorial in $L$ and so define an isomorphism
\begin{displaymath}
f_{V,W}\colon V^{\dag} \otimes W^{\dag} \to (V \otimes W)^{\dag}
\end{displaymath}
of objects in $\Vec^{\fs}$.  This gives transpose the structure of a tensor functor.  From the properties of the isomorphisms $i$ discussed above, it is clear that the diagram
\begin{displaymath}
\xymatrix{
(V \otimes W)^{\dag} \ar[r]^{f_{V,W}} \ar[d]_{(\tau_{V,W})^{\dag}} &
V^{\dag} \otimes W^{\dag} \ar[d]^{\sigma_{V^{\dag}, W^{\dag}}} \\
(W \otimes V)^{\dag} \ar[r]^{f_{W,V}} & W^{\dag} \otimes V^{\dag} }
\end{displaymath}
commutes.  Thus transpose is a symmetric tensor functor, as stated.

\article
\label{transp:comp}
Let $V$ be an object of $\cV$ of degree $n$.  Since transpose is a tensor functor, there is a natural isomorphism between $(V^{\dag})^{\otimes n}$ and $(V^{\otimes n})^{\dag}$.  The fact that transpose turns $\tau$ into $\sigma$ means that this isomorphism is $S_n$-equivariant if we twist the usual $S_n$-action on $(V^{\otimes n})^{\dag}$ by $\sgn^n$.  It follows that $\bS_{\lambda}(V^{\dag})$ is isomorphic to $\bS_{\lambda}(V)^{\dag}$ if $n$ is even and $\bS_{\lambda^{\dag}}(V)^{\dag}$ is $n$ is odd.  More generally, suppose $V$ and $W$ are objects of $\cV$.  If $V$ is concentrated in even degrees then $(W \circ V)^{\dag}=W \circ (V^{\dag})$, while if $V$ is concentrated in odd degrees then $(W \circ V)^{\dag}=(W^{\dag}) \circ (V^{\dag})$.

\part{Twisted commutative algebras}

\xsection{Basic definitions and results}
\label{sec:tca-gen}

\subsection{The definition} \label{ss:tca-defn}

\article
A {\bf twisted commutative algebra} (tca) is an associative unital commutative algebra in the symmetric tensor category $\cV$.  (We use the standard symmetric structure $\tau$ on $\cV$ in all that follows.)  This definition has appeared before, see for example, \cite{baratt} and \cite{gs}. We now unpack the definition in the various models.

\article[Sequence model]
A tca in this model is an associative unital graded $\bC$-algebra $A$ supported in non-negative degrees equipped with an action of $S_n$ on the degree $n$ piece $A_n$, such that the multiplication map 
\begin{align}
A_n \otimes A_m \to A_{n+m}
\end{align}
is $S_n \times S_m$ equivariant and the following twisted version of the commutativity axiom holds:  for $x \in A_n$ and $y \in A_m$ we have $yx=\tau(xy)$, where $\tau \in S_{n+m}$ switches the first $n$ and last $m$ elements of $\{1,\ldots,n+m\}$.  An $A$-module $M$ is a graded $A$-module (in the usual sense) equipped with an action of $S_n$ on $M_n$ for which the multiplication map 
\begin{align}
A_n \otimes M_m \to M_{n+m}
\end{align}
is $S_n \times S_m$ equivariant.

From the point of view of the sequence model, tca's are non-commutative rings equipped with some extra structure that partially compensates for the non-commutativity.

\article[GL-model]
A tca in this model is a commutative associative unital graded $\bC$-algebra $A$ on which $\GL(\infty)$ acts (by algebra automorphisms), such that $A$ forms a polynomial representation.  An $A$-module $M$ is just an $A$-module in the usual sense equipped with a compatible action of $\GL(\infty)$ (meaning the multiplication map $A \otimes M \to M$ is $\GL(\infty)$-equivariant), such that $M$ forms a polynomial representation.

From the point of view of the $\GL$-model, tca's are just large commutative rings with a large group action.

\article[Schur model]
A tca in this model is a polynomial functor $A$ from $\Vec$ to the category of associative unital commutative $\bC$-algebras.  Here ``polynomial'' means that the resulting functor $\Vec \to \Vec$ given by forgetting the algebra structure is polynomial.  An $A$-module is a polynomial functor $M\colon \Vec \to \Vec$ such that $M(V)$ is equipped with the structure of an $A(V)$-module for each vector space $V$, and these structures are functorial, in the sense that if $f\colon V \to V'$ is a map of vector spaces then $M(V) \to M(V')$ is a map of $A(V)$-modules, where $M(V')$ is regarded as an $A(V)$-module through the homomorphism $A(V) \to A(V')$.

Thus if $A$ is a tca in the Schur model then $A(\bC^n)$ is a commutative associative unital ring with a $\GL(n)$-action, for each $n$.  In cases of interest, the rings $A(\bC^n)$ are finitely generated $\bC$-algebras, so this point of view connects tca's to familiar objects from commutative algebra.

\article[fs-model]
A tca in this model is a functor $A\colon \fs \to \Vec$ equipped with a multiplication map
\begin{align}
A_L \otimes A_{L'} \to A_{L \amalg L'}
\end{align}
for every pair of finite sets $(L, L')$ which is functorial, associative, unital and commutative.  The meaning of functorial and associative is clear.

The unit condition means that there is an element $1 \in A_{\emptyset}$ such that for any $x \in A_L$ the image of $x \otimes 1$ under the multiplication map is $x$.  Actually, this is not perfectly correct:  the image of $x \otimes 1$ is an element of $A_{L \amalg \emptyset}$, while $x$ belongs to $A_L$; these are different spaces, so we cannot compare the two elements.  However, there is a canonical bijection of sets $L \amalg \emptyset \to L$, and thus a canonical isomorphism $A_{L \amalg \emptyset} \to A_L$.  We really mean that the image of $x \otimes 1$ under the multiplication map corresponds to $x$ under this isomorphism.

The commutativity condition is that the diagram
\begin{align}
\xymatrix{
A_L \otimes A_{L'} \ar[r] \ar[d] & A_{L \amalg L'} \ar[d] \\
A_{L'} \otimes A_L \ar[r] & A_{L' \amalg L} }
\end{align}
commute.  Here the horizontal maps are the multiplication maps, the left vertical map is the usual switching-of-factors map and the right vertical map is the map induced by $A$ from the canonical bijection of sets $L \amalg L' \to L' \amalg L$.

An $A$-module $M$ is a functor $\fs \to \Vec$ equipped with a multiplication map
\begin{align}
A_L \otimes M_{L'} \to M_{L \amalg L'}
\end{align}
which satisfies the usual conditions.

\article
An ideal of a tca $A$ is just an $A$-submodule of $A$.  The map
\begin{displaymath}
\{ \textrm{ideals of A} \} \to \{ \textrm{$\GL(\infty)$-stable ideals of $A(\bC^{\infty})$} \}, \qquad
\mf{a} \mapsto \mf{a}(\bC^{\infty})
\end{displaymath}
is an order-preserving bijection, and this is often how we view ideals of $A$.

\article
Let $A$ be a tca taken and let $M$ be an $A$-module, taken in the GL-model.  For an element $x \in M$, we write $(x)$ for the submodule that $x$ generates in the usual sense, and $\langle x \rangle$ for the $\GL$-submodule that $x$ generates.  Thus $\langle x \rangle=(gx)_{g \in \GL(\infty)}$.  Note that for $x \in A$ and $y \in M$ then identity $\langle x \rangle \langle y \rangle=\langle xy \rangle$ is not valid in general; the former is equal to $( (gx)(hy) )_{g,h \in \GL(\infty)}$, while the latter is equal to $( g(xy))_{g \in \GL(\infty)}$.  However, it is true when $x$ and $y$ are disjoint (see Lemma~\pref{ideal:mult} below).

\begin{Lemma}
\label{ideal:wt}
Let $M$ be an $A$-module and let $x$ be an element of $M$.  Let $x=\sum_{i=1}^n x_i$ be the decomposition of $x$ into weight vectors.  Then $\langle x \rangle=\sum_{i=1}^n \langle x_i \rangle$.
\end{Lemma}

\begin{proof}
It is clear that $x$, and thus $\langle x \rangle$, is contained in $\sum_{i=1}^n \langle x_i \rangle$.  On the other hand, since each $x_i$ belongs to the $\GL$-submodule of $M$ generated by $x$, we have $x_i \in \langle x \rangle$ for each $i$, which gives the other containment.
\end{proof}

\begin{Lemma}
\label{ideal:mult}
Let $M$ be an $A$-module and let $x \in A$ and $y \in M$ be disjoint.  Then $\langle x \rangle \langle y \rangle=\langle xy \rangle$.
\end{Lemma}

\begin{proof}
It is enough to show that $(gx)(hy)$ belongs to the $\GL$-submodule of $M$ generated by $xy$.  By Proposition~\pref{prop:disjoint}, $gx \otimes hy$ belongs to the $\GL$-submodules of $A \otimes M$ generated by $x \otimes y$, and so the result follows.
\end{proof}

\subsection{Examples}
\label{ss:tca-ex}

\article
We now give some examples of tca's.  The easiest way to produce an example is to form the symmetric algebra on an object of $\cV$.  We call these ``polynomial tca's,'' and discuss certain cases in detail.  We then give some examples of non-polynomial tca's.

\article
The simplest nontrivial polynomial tca is $A=\Sym(\bC\langle 1 \rangle)$.  This is the univariate polynomial ring with one degree 1 generator.  In the sequence model, this is the graded ring $A=\bC[t]$, with all symmetric group actions trivial.  In the $\GL$-model, this is the ring $A=\Sym(\bC^{\infty})$, with the usual action of $\GL(\infty)$.  One reason this tca is so ``easy'' is that its graded pieces are simple objects of $\cV$; it is unique among polynomial tca's in this respect.  This tca also comes up in many applications (see \S\ref{ss:efw} and \S\ref{ss:fimod}).

The ideals of this tca are easy to classify:  each is generated by some $A_n$, and is thus a power of the maximal ideal.  In other words, the ideals of $A$ as a tca are the same as the ideals of $\bC[t]$ as a graded ring.  Note however, that $A$-modules are not the same thing as graded modules over $\bC[t]$.  There is a functor from $A$-module to graded $\bC[t]$-modules, but it typically loses information; for example, the $A$-modules $A \otimes \Sym^2(\bC^{\infty})$ and $A \otimes \bigwedge^2(\bC^{\infty})$ are non-isomorphic, but both give rise to a free module of rank 1 over $\bC[t]$.

The minimal projective resolution of the residue field of $A$ is the Koszul complex:
\begin{align}
\cdots \to A \otimes \bigwedge^i(\bC\langle 1 \rangle) \to \cdots \to A \otimes \bigwedge^2(\bC\langle 1 \rangle) \to A \otimes \bigwedge^1(\bC\langle 1 \rangle) \to A \to \bC \to 0.
\end{align}
Note that this resolution is infinite, and so the global dimension of $A$ is infinite.  A detailed analysis of the homological structure of modules over the algebra $A$ can be found in \cite{symc1}. We include a summary of some of these results in \S\ref{ss:symc1}.

\article \label{ss:multipol}
The next simplest tca's are the multivariate polynomial rings generated in degree 1.  Let $U$ be a finite dimensional vector space.  Recall that $U\langle 1 \rangle$ is the object $U \otimes \bC \langle 1 \rangle$ of $\cV$; up to isomorphism, this is just a direct sum of $\dim{U}$ copies of $\bC\langle 1 \rangle$.  Let $A$ be the tca $\Sym(U\langle 1 \rangle)$.  We begin by examining $A$ in the different models.
\begin{enumerate}[(a)]
\item {\it Sequence model.}  In this model, $A$ is the sequence $(A_n)$ where $A_n=U^{\otimes n}$.  Multiplication is the map $U^{\otimes n} \otimes U^{\otimes m} \to U^{\otimes (n+m)}$ given by concatenating tensors.  In other words, $A$ is just the tensor algebra on $U$.  This is highly non-commutative.  However, to regard $A$ as a tca, we remember the action of $S_n$ on each $U^{\otimes n}$, and we have twisted commutativity: for $x \in A_n$ and $y \in A_m$ we have $xy=\tau(yx)$, for the element $\tau$ of $S_{n+m}$ which switches the first $n$ and last $m$ elements of $\{1,\ldots,n+m\}$.
\item {\it fs-model.}  In this model, $A$ is the functor $\fs \to \Vec$ given by $A_L=U^{\otimes L}$.  The space $U^{\otimes L}$ is isomorphic to a tensor product of $\#L$ copies of $U$, but can described canonically as the universal space equipped with a multi-linear map from $U \times L$; in other words, we can think of pure tensors in $U^{\otimes L}$ as being indexed by elements of $L$, and this gives functoriality in $L$.  Multiplication is again given by concatenation.
\item {\it GL-model.}  In this model, $A$ is simply the ring $\Sym(U \otimes \bC^{\infty})$.  If we pick a basis $x_1, \ldots, x_n$ of $U$, then $A$ can be regarded as the polynomial ring $\bC[x_{ij}]$ with $1 \le i \le n$ and $j \ge 1$.  The group $\GL(\infty)$ acts by linear substitutions on the variables with respect to the second subscript.
\item {\it Schur model.}  In this model, $A$ is the functor which attaches to a vector space $V$ the ring $A(V)=\Sym(U \otimes V)$.  This is a polynomial ring in $\dim(U)\dim(V)$ variables.
\end{enumerate}
Using the Cauchy formula \pref{schur:cauchy}, we can decompose $A$ as an object of $\cV$:
\begin{equation}
\label{eq:multipol-1}
A = \bigoplus_{\ell(\lambda) \le \dim(U)} \bS_{\lambda}(U) \otimes \bS_{\lambda}.
\end{equation}
We thus see that $\bS_{\lambda}$ is a constituent of $A$ if and only if $\ell(\lambda) \le \dim(U)$, in which case its multiplicity is the dimension of $\bS_{\lambda}(U)$.

\article \label{ss:poly}
Moving beyond tca's generated in degree 1, we next consider the polynomial tca $A=\Sym(\bC\langle n \rangle)$, with $n>0$.  In the GL-model, this is the ring $\Sym((\bC^{\infty})^{\otimes n})$.  In the fs-model, $A$ admits a nice combinatorial description, which is most easily seen using \pref{art:gl2seq}.  Namely, $A_L$ is naturally the vector space having for a basis the set of perfect directed $n$-uniform hypermatchings on $L$.  (Recall that a ``directed $n$-uniform hyperedge'' on a vertex set $L$ is an ordered collection of $n$ distinct elements of $L$.  The ``perfect'' part means every element of $L$ belongs to a unique hyperedge.)  In fact, this description can be extended to general polynomial rings:  if $A=\Sym(\bC\langle n_1 \rangle \oplus \cdots \oplus \bC\langle n_r \rangle)$ then $A_L$ can be described as the vector space having for a basis the set of directed perfect hypermatchings on $L$ whose edges have been colored one of $r$ colors, and such that the edges of color $i$ are $n_i$ uniform.

\article
As we saw in \pref{ss:multipol}, the multivariate polynomial tca's generated in degree 1 can be decomposed exactly in $\cV$.  The corresponding decomposition for polynomial tca's generated in degree $>2$ is unknown, since the plethysms $\Sym^n \circ \bS_{\lambda}$ are unknown for $\vert \lambda \vert>2$.  However, in degree 2 we have do have some results, namely:
\begin{displaymath}
\Sym(\Sym^2) = \bigoplus_{\lambda} \bS_{2\lambda}, \qquad \Sym(\lw^2)=\bigoplus_{\lambda \in Q_{-1}} \bS_{\lambda}.
\end{displaymath}
The first sum is taken over all partitions, and $2\lambda$ denotes the partition $(2\lambda_1, 2\lambda_2, \ldots)$.  In the second sum, $Q_{-1}$ is the set of partitions $\lambda$ whose Frobenius coordinates $(a_1, \ldots, a_r \mid b_1, \ldots, b_r)$ satisfy $a_i=b_i-1$.  In particular, these decompositions are multiplicity-free.

\article[Determinantal varieties] \label{ss:detlvar}
The polynomial tca's $\Sym(U \langle 1 \rangle)$ contain many interesting non-polynomial quotients. Note that there is a natural $\GL(U)$ action on this tca. Perhaps the simplest class of ideals are the $\GL(U)$-equivariant ones, i.e., the determinantal ideals.

Pick $0 \le k \le \dim U$. By the Cauchy identity \pref{schur:cauchy}, we have 
\[
\bigwedge^k U \otimes \bigwedge^k \bC^\infty \subset \Sym^k(U \langle 1 \rangle).
\]
If we interpret $\Sym(U \langle 1 \rangle)$ as the coordinate ring of the space of $\dim U \times \infty$ matrices, then the above space is spanned by the $k \times k$ minors. The ideal generated by this is the determinantal ideal, and they have been intensely studied. These are usually considered in the case that $\bC^\infty$ is replaced by a finite-dimensional vector space. See  \cite{brunsvetter} for a general reference.

\subsection{Finiteness conditions}

\article
Let $A$ be a tca.  We say that $A$ is {\bf finitely generated} if there exists a surjection $\Sym(V) \to A$, where $V$ is a finite length object of $\cV$.  We say that an $A$-module $M$ is {\bf finitely generated} if there exists a surjection $A \otimes V \to M$ with $V$ a finite length object of $\cV$.  Of course, one can make sense of ``finite presentation'' as well.

\article
Let us examine these definitions in the sequence model.  The tca $A$ is finitely generated if there exist finitely many elements $x_i$ in various $A_{n_i}$ such that the smallest subspace of $A$ containing the $x_i$ and stable under multiplication and the action of the symmetric groups is $A$ itself.  Let us briefly explain how to see this.  Suppose $A$ is finitely generated, and let $\Sym(V) \to A$ be a surjection, with $V$ finite length.  We can write $V$ as a quotient of a finite direct sum of objects of the form $\bC\langle n_i \rangle$, and so $A$ is a quotient of $\Sym(\bigoplus_i \bC\langle n_i \rangle)$.  Giving an algebra map from the symmetric power is the same as giving a linear map from $\bigoplus_i \bC \langle n_i \rangle$, and  giving a map $\bC\langle n_i \rangle \to A$ in $\cV$ is the same as giving an element of $A_{n_i}$.  This is where the elements $x_i$ come from.  We leave the rest of the reasoning to the reader.  Finite generation of $A$-modules can be described similarly.

\article
Let us now look at how these definitions work in the GL-model.  The tca $A$ is finitely generated if there exist finitely many elements of $A$ such that the $\GL(\infty)$-subrepresentation they span generates $A$ as an algebra in the usual sense.  Similarly, an $A$-module $M$ is finitely generated if it contains finitely many elements such that the $\GL(\infty)$-subrepresentation they span generates $M$ as an $A$-module in the usual sense.

\article
Let $A$ be a finitely generated tca in the Schur model.  Then $A(V)$ is a finitely generated ring for any finite dimensional vector space $V$.  Indeed, we can write $A$ as a quotient of $\Sym(F)$ for some finite length polynomial functor $F$, and so $A(V)$ is a quotient of $\Sym(F(V))$.  The vector space $F(V)$ is finite dimensional, so the statement follows.  Similarly, if $M$ is a finitely generated $A$-module then $M(V)$ is a finitely generated $A(V)$-module for all finite dimensional $V$.  The converse to these statements is false.  For instance, let $F$ be the direct sum of all wedge powers.  Then $A=\Sym(F)$ is a tca, but not finitely generated.  However, $F(V)$ is finite dimensional for any finite dimensional space $V$, and so $A(V)$ is a finitely generated ring for all finite dimensional $V$.

\begin{Remark}
One can make sense of the notions of ``finitely generated'' and ``finitely presented'' in any abelian category:  for instance, one says that $M$ is finitely presented if the functor $\Hom(M,-)$ commutes with direct limits.  One can show that for $\Mod_A$ these general notions coincide with the ones given above.
\end{Remark}

\article
Let $A$ be a tca.  We say that an $A$-module $M$ is {\bf noetherian} if every ascending chain of $A$-submodules stabilizes.  As usual, this is equivalent to every submodule of $M$ being finitely generated.  Note that subs, quotients and extensions of noetherian modules are noetherian.  We say that $A$ is {\bf noetherian} (as a tca) if every finitely generated $A$-module is noetherian.  We say that $A$ is {\bf weakly noetherian} if $A$ is noetherian as an $A$-module, i.e., if ideals of $A$ satisfy the ascending chain condition (ACC).

\begin{Proposition}
\label{finite:imps}
Let $A$ be a tca, and consider the following conditions:
\begin{enumerate}[\rm (a)]
\item $A$ is noetherian.
\item $A$ is weakly noetherian.
\item $A_0$ is noetherian and $A$ is finitely generated over $A_0$.
\item $A(\bC^n)$ is noetherian for all $n$.
\end{enumerate}
Then {\rm (a)} $\implies$ {\rm (b)} $\implies$ {\rm (c)} $\implies$ {\rm (d)}.
\end{Proposition}

\begin{proof}
It is clear that (a) implies (b).  Assume now that $A$ satisfies (b).  As $A_0$ is a quotient of $A$, it satisfies ACC on ideals and is therefore noetherian in the usual sense.  Since $A$ is weakly noetherian, the ideal $A_+$ of $A$ is finitely generated.  Generators for this ideal are generators for $A$ as an algebra over $A_0$.  This proves (c).  Finally, suppose $A$ satisfies (c).  Then $A(\bC^n)$ is finitely generated over $A_0$, and is therefore noetherian.  This proves (d).
\end{proof}

\begin{remark}
We expect (b) implies (a), though we have not proved this.  A major open question is whether (c) implies (a).  Clearly (d) does not imply (c):  consider the symmetric algebra on the exterior algebra.  For ``bounded'' tca's, we show that (a)--(d) are equivalent (Proposition~\pref{cor:noeth}).
\end{remark}

\begin{Example}
\label{sym2:wnoeth}
The tca $A=\Sym(\lw^2)$ is weakly noetherian: in \cite[\S 3]{pfaffiansarithmetic}, it is shown that the poset of equivariant ideals of $\Sym(\lw^2)$ is isomorphic to the poset of partitions appearing in $\Sym(\lw^2)$ under the relationship of containment. Each partition appears at most once (precisely, we get those of the form $(2\lambda)^\dagger$), so the weakly noetherian statement is clear. For $\Sym(\Sym^2)$, an analogous statement is true, see \cite{abeasis} for the equivalence of the relevant posets.
\end{Example}

\subsection{Nakayama's lemma and applications}

\article
Let $A$ be a tca.  We let $A_+$ denote the sum of the positive degree pieces of $A$.  It is an ideal, and $A/A_+=A_0$ is the degree 0 piece of $A$.  We have the following form of Nakayama's lemma.  The usual proof applies.

\begin{proposition} \label{prop:nakayama}
Let $M$ be an $A$-module, let $V$ be an object of $\cV$ and let $V \to M$ be a map in $\cV$ such that $V$ surjects onto $M/A_+M$.  Then the map $A \otimes V \to M$ is surjective.
\end{proposition}

\begin{Proposition}
An $A$-module $M$ is finitely generated if and only if $M/A_+M$ is a finitely generated $A_0$-module.
\end{Proposition}

\begin{proof}
If $M$ is finitely generated over $A$ then clearly $M/A_+M$ is finitely generated over $A_0$.  Conversely, suppose $M/A_+M$ is finitely generated over $A_0$.  Choose a surjection $A_0 \otimes V \to M/A_+M$ with $V$ a finite length object in $\cV$.  Lift $V$ to $M$.  Then the $A_0$-submodule of $M$ generated by $V$ surjects onto $M/A_+M$, and so Proposition~\pref{prop:nakayama} shows that $V$ generates $M$ as an $A$-module, i.e., the map $A \otimes V \to M$ is surjective.  This shows that $M$ is finitely generated.
\end{proof}

\begin{corollary}
Suppose $A_0=\bC$.  Then an $A$-module $M$ is finitely generated if and only if $M/A_+M$ is a finite length object of $\cV$.
\end{corollary}

\begin{Proposition}
An $A$-module $M$ is projective if and only if it is of the form $A \otimes_{A_0} M_0$ for some projective $A_0$-module $M_0$.
\end{Proposition}

\begin{proof}
By adjointness, $\Hom_A(A \otimes_{A_0} M_0, -)=\Hom_{A_0}(M_0, -)$ is exact, and so $A \otimes_{A_0} M_0$ is a projective $A$-module if $M_0$ is a projective $A_0$-module.  Suppose now that $M$ is a projective $A$-module.  Let $M_0=M/A_+M$, a projective $A_0$-module.  Since $M \to M_0$ is a surjection of $A_0$-modules and $M_0$ is projective, we can choose a section $M_0 \to M$.  By Nakayama's lemma, the induced map $f \colon A \otimes_{A_0} M_0 \to M$ is surjective.  Let $N=\ker{f}$.  Since $M$ is projective, the sequence
\begin{displaymath}
0 \to N/A_+N \to M_0 \to M/A_+M \to 0
\end{displaymath}
is exact.  As the right map is an isomorphism, we see that $N/A_+N=0$, and so $N=0$ by Nakayama's lemma.  Thus $f$ is an isomorphism, and so $M$ is isomorphic to $A \otimes_{A_0} M_0$.
\end{proof}

\begin{corollary}
Suppose $A_0=\bC$.  Then an $A$-module $M$ is projective if and only if it is of the form $A \otimes V$, for some $V$ in $\cV$.
\end{corollary}

\subsection{Nilpotents and radicals}
\label{ss:tca-nilp}

\article
Let $A$ be a tca and let $x$ be an element of $A(\bC^{\infty})$.  We say that $x$ is {\bf nilpotent} if $x^n=0$ for some $n$.  We say that $x$ is {\bf strongly nilpotent} if there exist $g_1, \ldots, g_n \in \GL(\infty)$ such that the $g_i x$ are mutually disjoint and $(g_1 x) \cdots (g_n x)=0$.  We show in Corollary~\pref{cor:SNimpliesN} below that strongly nilpotent implies nilpotent.  We say that $A$ is {\bf reduced} if it has no nilpotents and {\bf weakly reduced} if it has no strong nilpotents.  Since $A(\bC^{\infty})$ is the union of the $A(\bC^n)$, it is clear that $A$ is reduced if and only if $A(\bC^n)$ is for each $n$.  We give an example in \pref{ex:weakdomain} which shows that weakly reduced does not imply strongly reduced.  It is not difficult to see that if $A$ has a strong nilpotent then it has a flat strong nilpotent (recall the definition of flat from \pref{art:defn:flat}).  It is therefore easy to see in the $\fs$-model if $A$ is weakly reduced.

\begin{Proposition} \label{prop:nilpotentideal}
An element $x$ is strongly nilpotent if and only if the ideal $\langle x \rangle$ is nilpotent.
\end{Proposition}

\begin{proof}
Suppose $\langle x \rangle^n=0$ for some $n \ge 1$.  Choose $g_1, \ldots, g_n \in \GL(\infty)$ such that the $g_i x$ are disjoint.  Then $g_i x$ belongs to $\langle x \rangle$ for each $i$, and so $(g_1 x) \cdots (g_n x)=0$.  This shows that $x$ is strongly nilpotent.  Now suppose that $x$ is strongly nilpotent, and choose $g_1, \ldots, g_n \in \GL(\infty)$ such that the $g_i x$ are disjoint and $(g_1 x) \cdots (g_n x)=0$.  Let $V$ be the $\GL$-module generated by $x$.  It suffices to show that any $n$-fold product of elements of $V$ vanishes.  By Proposition~\pref{prop:disjoint}, the element $(g_1 x) \otimes \cdots \otimes (g_n x)$ of $A \otimes \cdots \otimes A$ generates $V \otimes \cdots \otimes V$ as a $\GL(\infty)$-module.  It follows that $(g_1 x) \cdots (g_n x)=0$ generates the image of $\Sym^n(V) \to A$ as a $\GL(\infty)$-module, and therefore this image is 0.  This completes the proof.
\end{proof}

\begin{corollary} \label{cor:SNimpliesN}
``Strongly nilpotent'' implies ``nilpotent.''
\end{corollary}

\article
We define the {\bf (nil)radical} of $A$, denoted $\rad(A)$, to be the set of nilpotents and the {\bf strong (nil)radical}, denoted $\srad(A)$, to be the set of strong nilpotents; both are ideals:

\begin{proposition}
The set of strong nilpotents is an ideal.
\end{proposition}

\begin{proof}
Suppose $x$ and $y$ are strongly nilpotent.  Then $\langle x+y \rangle \subset \langle x \rangle + \langle y \rangle$, and is therefore nilpotent. Similarly, if $a$ is an arbitrary element of $A$ then $\langle ax \rangle \subset \langle a \rangle \langle x \rangle$, and is therefore nilpotent.  This shows that $x+y$ and $ax$ are strongly nilpotent by Proposition~\pref{prop:nilpotentideal}.
\end{proof}

For an ideal $I$, we define $\rad(I)$ to be the inverse image of $\rad(A/I)$ and $\srad(I)$ to be the inverse image of $\srad(A/I)$.  We always have an inclusions $I \subset \srad(A/I) \subset \rad(A/I)$.

\begin{Proposition}
If $A$ is weakly noetherian then $\srad(A)$ is nilpotent.
\end{Proposition}

\begin{proof}
Since $A$ is weakly noetherian, the ideal $\srad(A)$ is finitely generated, i.e., of the form $\sum_{i=1}^n \langle x_i \rangle$.  Since each $x_i$ is strongly nilpotent, each $\langle x_i \rangle$ is nilpotent, which shows that $\srad(A)$ is nilpotent.
\end{proof}

\begin{remark}
The example of \pref{ex:weakdomain} shows that the nilradical need not be nilpotent, even in a finitely generated weakly noetherian tca.
\end{remark}

\subsection{Domains and primes}
\label{ss:tca-prime}

\article
\label{domain}
Let $A$ be a tca and let $x$ be an element of $A(\bC^{\infty})$.  We say that $x$ is a {\bf zero-divisor} if it is non-zero and there exists a non-zero $y \in A(\bC^{\infty})$ such that $xy=0$.  We say that $x$ is a {\bf strong zero-divisor} it is non-zero and there exists a non-zero $y \in A(\bC^{\infty})$ disjoint from $x$ such that $xy=0$.  We say that $A$ is a {\bf domain} if it has no zero-divisors and a {\bf weak domain} if it has no strong zero divisors.  Since $A(\bC^{\infty})$ is the union of the $A(\bC^n)$, it is clear that $A$ is a domain if and only if $A(\bC^n)$ is for each $n$.  One easily sees that if $A$ is not a weak domain then there exist non-zero disjoint flat elements $x$ and $y$ with $xy=0$ (recall the definition of flat from \pref{art:defn:flat}).  It is therefore easy to tell in the $\fs$-model if $A$ is a weak domain.  Note that ``weak domain'' implies ``weakly reduced.''

\begin{Proposition} \label{prop:domain}
``Domain'' is equivalent to ``weak domain and reduced.''
\end{Proposition}

\begin{proof}
It is clear that a domain is a reduced weak domain.  We now prove the converse.  Suppose that $A$ is not a domain but is reduced; we will show that $A$ is not a weak domain.

Since $A$ is not a domain we can find non-zero elements $x$ and $y$ such that $xy=0$.  We claim that we can take $x$ and $y$ to be weight vectors.  Totally order $\bZ^{\infty}$ lexicographically.  This order respects addition, i.e., if $\mu<\mu'$ and $\nu<\nu'$ then $\mu+\nu<\mu'+\nu'$.  Write $x=\sum_{i=1}^n x_i$ where $x_i$ is a non-zero weight vector of weight $\lambda_i$ and $\lambda_1<\cdots<\lambda_n$.  Similarly, write $y=\sum_{j=1}^m y_i$ where $y_j$ is a non-zero weight vector of weight $\mu_j$ and $\mu_1<\cdots<\mu_m$.  Then $0=xy=\sum_{i,j} x_i y_j$.  The term $x_1 y_1$ is the only one of weight $\lambda_1+\mu_1$, since all other terms have larger weight.  Thus $x_1 y_1=0$, which proves the claim.

So let $x$ and $y$ be non-zero weight vectors with $xy=0$.  We claim that if $X$ is a degree $n$ element of $\rU(\mf{gl}(\infty))$ then $(Xx) y^{n+1}=0$.  This is clear if $n=0$.  Suppose it is true for $n-1$ and let us prove it for $n$.  Write $X=X_1 X_2$ where $X_2$ has degree $n-1$ and $X_1$ has degree 1.  Then $(X_2 x) y^n=0$ by induction.  Applying $X_1$, we find
\begin{displaymath}
(X x) y^n + (X_2 x) n y^{n-1} (X_1 y) = 0.
\end{displaymath}
Multiplying by $y$ kills the second term, and so $(X x)y^{n+1}=0$.  This proves the claim.

Now, we can choose $X \in \rU(\mf{gl}(\infty))$ such that $Xx$ is non-zero and disjoint from $y$.  Since $y$ is not nilpotent, the equation $(Xx)y^{n+1}=0$ shows that $A$ is not a weak domain.
\end{proof}

\begin{Proposition}
\label{weak:rad}
A weak domain with nilpotent nilradical is reduced, and thus a domain.
\end{Proposition}

\begin{proof}
Let $A$ be a tca.  Assume that $\rad(A)$ is non-zero but nilpotent, say $\rad(A)^n=0$ but $\rad(A)^{n-1} \ne 0$.  We can then choose non-zero disjoint elements $x \in \rad(A)$ and $y \in \rad(A)^{n-1}$.  We then have $xy=0$, which shows that $A$ is not a weak domain.  Thus if $A$ is a weak domain and $\rad(A)$ is nilpotent then $\rad(A)=0$.
\end{proof}

\article
We say that an ideal $\fp$ of $A$ is {\bf prime} if $A/\fp$ is a domain.  Similarly, we say that $\fp$ is {\bf weakly prime} if $A/\fp$ is a weak domain.  By definition, $\fp$ is prime if and only if whenever $xy \in \fp$ either $x \in \fp$ or $y \in \fp$; similarly, $\fp$ is weakly prime if and only if whenever $xy \in \fp$ and $x$ and $y$ are disjoint either $x \in \fp$ or $y \in \fp$.  By the comments of \pref{domain}, $\fp$ is prime if and only if $\fp(\bC^n)$ is a prime ideal of $A(\bC^n)$ for each $n$.  Proposition~\pref{prop:domain} shows that $\fp$ is prime if and only if it is radical and weakly prime.  If $\fp$ is a weak prime then $\rad(\fp)$ is both weakly prime and radical, and thus prime.

\article[A non-reduced weak domain]
\label{ex:weakdomain}
Let $A$ be the polynomial ring $\Sym(\bC\langle 2 \rangle) = \Sym(\bC^\infty \otimes \bC^\infty)$.  We let $x_i$ be a basis of $\bC^{\infty}$ and write $[x_i x_j]$ in place of $x_i \otimes x_j$, so that we can think of $A$ as the polynomial ring on variables $[x_i x_j]$.  Let $B$ be the quotient of $A$ by the ideal $\langle [x_1 x_2]^2 \rangle$.  Obviously, $B$ is not reduced; in fact, \emph{every} positive degree element of $B$ is nilpotent, since $[x_1 x_2]$ generates the ideal $A_+$ in $A$, and thus the ideal $B_+$ in $B$.  Nonetheless, we claim that $B$ is a weak domain.

We now prove this claim.  Applying $\frac{1}{2} X_{1,3}X_{2,4} \in \rU(\mf{gl}(\infty))$ to $[x_1 x_2]^2$ we obtain the flat element $[x_1 x_2][x_3 x_4]+[x_1 x_4][x_3 x_2]$; furthermore, applying $\frac{1}{2} X_{3,1}X_{4,2}$ to this gives $[x_1 x_2]^2$.  Thus $\langle [x_1x_2]^2 \rangle=\langle [x_1 x_2][x_3 x_4]+[x_1 x_4][x_3 x_2] \rangle$.  If the $\fs$-model, we can think of $A_L$ as the vector space with basis the set of tableaux with two rows in which every element of $L$ appears exactly once (the order of the columns is irrelevant).  So we can think of $B_L$ as the quotient of this space by the identity
\begin{displaymath}
\begin{array}{|c|c|}
\hline
a & b \\
\hline
c & d \\
\hline
\end{array}  = (-1) \cdot
\begin{array}{|c|c|}
\hline
a & b \\
\hline
d & c \\
\hline 
\end{array}.
\end{displaymath}
If we totally order $L$ then $A_L$ has a basis in which the first row of the tableaux is increasing from left to right.  Using the above relation, we see that $B_L$ has a basis consisting of tableaux in which each row is increasing.  Thus if $L$ and $L'$ are totally ordered and we totally order $L \amalg L'$ by putting $L'$ after $L$, then the multiplication map $B_L \otimes B_{L'} \to B_{L \amalg L'}$ takes distinct pairs of basis element to distinct basis elements (where we use the tensor product basis of $B_L \otimes B_{L'}$), and is therefore injective.  This shows that $B$ is a weak domain.

We thus have the following examples:
\begin{itemize}
\item A nilpotent which is not strongly nilpotent (the element $[x_1 x_2]$ of $B$).
\item A weak domain which is not a domain (the tca $B$).
\item A weak prime which is not a prime (the ideal $\langle [x_1 x_2]^2 \rangle$ of $A$).
\item A tca with a non-nilpotent nilradical ($B$, as if $\rad(B)$ were nilpotent then it would vanish by Proposition~\pref{weak:rad}).
\end{itemize}
The tca $A$ is finitely generated and weakly noetherian \pref{sym2:wnoeth} and we expect it to be noetherian, though we have not proved this yet; these properties transfer to the quotient $B$.  Thus these tca's are as nice as one could want, and so the above examples are not bizarre pathologies, but genuine behavior.  However, we will see that such examples do not exist for bounded tca's (see \S\ref{ss:bdprime}, especially Corollary~\pref{cor:bd-equiv}).

\begin{Example}
The tca $A=\Sym(\bC\langle 1 \rangle)$ has exactly two prime ideals:  the zero ideal and its maximal ideal $A_+$.
\end{Example}

\begin{Example}
Let $A=\Sym(\bC\langle 1 \rangle)$ with $\dim(U)=2$.  Then $A$ has the zero ideal and $A_+$ as its unique minimal and maximal prime ideals.  However, there are many more primes.  Let $U \to L$ be a one dimensional quotient of $U$.  Then the kernel $\mf{p}_L$ of $\Sym(U\langle 1 \rangle) \to \Sym(L\langle 1 \rangle)$ is a prime ideal of $A$.  There is yet one more prime ideal of $A$.  Thinking of $A(\bC^{\infty})=\Sym(U \otimes \bC^{\infty})$ as $2 \times \infty$ matrices, let $\mf{P}$ be the ideal generated by all $2 \times 2$ minors.  Then $\mf{P}$ is prime.  The primes just listed exhaust the prime ideals, and we have the following containments among them:
\begin{equation}
0 \subset \mf{p}_L \subset \mf{P} \subset A_+
\end{equation}
There are no other containments.
\end{Example}

\begin{Example}
Let $A=\Sym(\Sym^2)$.  One can think of $X=\Spec(A(\bC^{\infty}))$ as the space of quadratic forms on $\bC^{\infty}$ (although one should be careful here!).  Let $X_n \subset X$ (for $n \ge 0$) be the subspace consisting of quadratic forms of rank $\le n$.  Then $X_n$ corresponds to a prime ideal $\mf{p}_n$ of $A$.  In fact, $\mf{p}_n$ is the ideal generated by the partition $(2^{n+1})$.  The primes $\mf{p}_n$ are mutually unequal and exhaust all primes ideals of $A$.  They form a single decreasing chain:
\begin{displaymath}
0=\mf{p}_{\infty} \subset \cdots \subset \mf{p}_2 \subset \mf{p}_1 \subset \mf{p}_0=A_+
\end{displaymath}
By \pref{sym2:wnoeth}, $A$ is weakly noetherian.  This example shows that prime ideals in a weakly noetherian tca need not satisfy the descending chain condition. The analogous situation does not occur for noetherian local rings in usual commutative algebra \cite[Theorem 12.1]{eisenbud}.  There are other aspects of classical dimension theory that go awry in the tca setting as well.
\end{Example}

\begin{Proposition}
\label{minprime}
Let $A$ be a tca.  The minimal primes of $A$ coincide with the minimal primes of $A(\bC^{\infty})$.
\end{Proposition}

\begin{proof}
Since $\GL(\infty)$ is connected, it acts trivially on the set of minimal primes of $A(\bC^{\infty})$. (Geometrically, if $X$ is an irreducible component defined by a minimal prime, then $\GL(\infty) \times X$ is irreducible so cannot contain more than one component in its image under the action map.) So every minimal prime of $A(\bC^{\infty})$ is $\GL$-stable, and thus a prime ideal of $A$, obviously minimal.  If $\fp$ is a minimal prime of $A$, then $\fp(\bC^{\infty})$ is necessarily a minimal prime of $A(\bC^{\infty})$: if it were not then it would properly contain some minimal prime $\fq(\bC^{\infty})$, and then $\fp$ would properly contain $\fq$.
\end{proof}

\begin{Proposition}
\label{rad1}
Let $\fa$ be an ideal of $A$.  Then $\rad(\fa)$ is the intersection of the primes containing $\fa$.
\end{Proposition}

\begin{proof}
It suffices to show that $\rad(A)$ is the intersection of the minimal primes of $A$.  Since $\rad(A)$ is just the set of nilpotents of $A(\bC^{\infty})$ and the minimal primes of $A$ are just the minimal primes of $A(\bC^{\infty})$, this follows from the usual statement from commutative algebra \cite[Prop.~1.14]{atiyahmacdonald}.
\end{proof}

\begin{Proposition}
\label{rad2}
Let $\fa$ be an ideal of $A$.  Then $\srad(\fa)$ is the intersection of the weak primes containing $\fa$.
\end{Proposition}

\begin{proof}
We model this proof on that of \cite[Prop.~1.8]{atiyahmacdonald}.  Passing to $A/\fa$, it suffices to treat the case where $\fa=0$.  Any strongly nilpotent element is contained in every weak prime ideal, so one containment is clear.  Let $f$ be an element of $A(\bC^{\infty})$ which is not strongly nilpotent.  We will construct a weak prime ideal not containing $f$.  Let $S$ be the set of ideals which contain no power of $\langle f \rangle$.  This set is non-empty (it contains the zero ideal), and closed under directed unions, and therefore contains some maximal element $\fp$.  We claim that $\fp$ is weakly prime.  To show this, suppose that $x$ and $y$ are disjoint elements such that $xy \in \fp$ but $x, y \not\in \fp$.  Then $\fp+\langle x \rangle$ and $\fp+\langle y \rangle$ are strictly larger than $\fp$, and so do not belong to $S$.  We therefore have $\langle f \rangle^n \in \fp+\langle x \rangle$ and $\langle f \rangle^m \in \fp+\langle y \rangle$ for some $n,m>0$.  We therefore find
\begin{equation}
\langle f \rangle^{n+m} \in \fp+\langle x \rangle \langle y \rangle=\fp+\langle xy \rangle=\fp,
\end{equation}
where in the first equality we used Lemma~\pref{ideal:mult} and in the second equality, we used that $xy$ belongs to $\fp$.  This equation is a contradiction, so we conclude that $\fp$ is weakly prime.
\end{proof}

\subsection{Localization}

\noarticle
Localization is a central tool in the study of commutative rings.  Unfortunately, the class of tca's is not closed under localization; this is similar to the fact that the class of non-negatively graded rings is not closed under localization.  Nonetheless, we can still access module categories over what would be the localized algebras by considering a quotient category.  We now discuss this construction.

\article
Let $A$ be a tca, $M$ an $A$-module and $\fa$ an ideal of $A$.  We let $M[\fa]$ be the maximal submodule of $M$ annihilated by $\fa$.  In the $\GL$-model, $M[\fa](\bC^{\infty})$ is just the set of elements of $M(\bC^{\infty})$ annihilated by $\fa(\bC^{\infty})$.  We define $M[\fa^{\infty}]$ to be the union of the $M[\fa^n]$.

\article
Let $\fp$ be a prime ideal of $A$.  Consider the full subcategory $\Mod^{\fp}_A$ of $\Mod_A$ consisting of modules $M$ for which $M[\fp^{\infty}]=M$.  Then $\Mod^{\fp}_A$ is a Serre subcategory of $\Mod_A$, and we define $\Mod_{A_{\fp}}$ to be the Serre quotient category
\[
\Mod_{A_{\fp}} = \Mod_A / \Mod^{\fp}_A
\]
(see \cite[\S 3.1]{gabriel} for the relevant definitions). The objects of the quotient category are the same as the original category, but we have inverted all morphisms whose kernel and cokernel belong to the subcategory. Note that $A_{\fp}$ is just a symbol here, and does not refer to any algebra.  There is a natural localization functor $\Mod_A \to \Mod_{A_{\fp}}$ which is exact.

\article
\label{modK}
A particularly important case occurs when $A$ is a domain and $\fp$ is the zero ideal.  In this case, we usually introduce a formal symbol $K$ to mean (intuitively) the fraction field of $A$, and write $\Mod_K$ in place of $\Mod_{A_{\fp}}$.  In fact, there is a more explicit construction of this category, as follows.  Define $K$ to be the set of degree zero elements in the ring obtained by inverting all positive degree elements of $A(\bC^{\infty})$.  Then $K$ is an actual field, equipped with an action of $\GL(\infty)$.  A {\bf semilinear} representation of $\GL(\infty)$ over $K$ is a $K$-vector space $V$ equipped with an action of $\GL(\infty)$ such that $g(av)=(ga)(gv)$ for all $g \in \GL(\infty)$, $a \in K$ and $v \in V$.  Such a representation is {\bf polynomial} if there exists a $\bC$-subspace $U$ of $V$ which forms a polynomial representation of $\GL(\infty)$ and spans $V$ over $K$.  With this terminology in hand, we can describe $\Mod_K$ as the category of polynomial semilinear representations of $\GL(\infty)$ over $K$.  Furthermore, the objects of $\Mod_K$ which are localizations of finitely generated objects of $\Mod_A$ are precisely the finite dimensional polynomial semilinear representations; in particular, the image in $\Mod_K$ of a finitely generated object of $\Mod_A$ has finite length. 

\article
Let $A$ be a tca and let $\fp$ be a prime of $A$.  We introduce the formal symbol $\kappa(\fp)$ for the residue field of $\fp$, i.e., the localization of $A/\fp$ at 0.  If $M$ is an $A/\fp$-module, we let $\len_{\fp}(M)$ be the dimension of the image of $M$ under the localization functor $\Mod_{A/\fp} \to \Mod_{\kappa(\fp)}$.  More generally, if $M$ is an $A$-module annihilated by $\fp^n$ we put
\begin{equation}
\len_{\fp}(M) = \sum_{i=0}^{n-1} \len_{\fp}(\fp^i M/\fp^{i+1} M)
\end{equation}
This number is finite if $M$ and $\fp$ are both finitely generated.

\begin{Remark}
The example $A = \Sym(\bC\langle 1 \rangle)$ is studied in detail in \cite{symc1}. We summarize some of the results from that paper in \S\ref{ss:symc1}.
\end{Remark}

\xsection{Bounded tca's}
\label{sec:tca-bd}

\subsection{Boundedness}

\article
Recall that for a partition $\lambda$ we write $\ell(\lambda)$ for the number of parts of $\lambda$, or equivalently, the number of rows in the corresponding Young diagram.  We now extend this notation to objects of $\cV$.  Precisely, for an object $V$ of $\cV$ we define $\ell(V)$ to be the supremum of $\ell(\lambda)$ over $\lambda$ for which $\bM_{\lambda}$ occurs in $V$ with non-zero multiplicity.  We say that $V$ is {\bf bounded} if $\ell(\lambda)$ is finite.  The Littlewood--Richardson rule \pref{ss:lw} implies that $\ell(M \otimes N) = \ell(M)+\ell(N)$ for any $M$ and $N$ in $\cV$.  In particular, the class of bounded objects is closed under tensor product.

\article
Let $\cV^{\le n}$ be the full subcategory of $\cV$ on objects $V$ with $\ell(V) \le n$.  Then $\cV^{\le n}$ is an abelian subcategory of $\cV$.  The inclusion $\cV^{\le n} \to \cV$ has a two-sided adjoint $\tau^{\le n}\colon \cV \to \cV^{\le n}$.  This functor simply kills the simple objects $\bS_{\lambda}$ with $\ell(\lambda)>n$.  The category $\cV^{\le n}$ is not closed under tensor product.  However, one can define a tensor product on $\cV^{\le n}$ by first forming the tensor product in $\cV$ and then applying $\tau^{\le n}$.  This gives $\cV^{\le n}$ the structure of a symmetric tensor category.  We have the following extremely important proposition:

\begin{Proposition} \label{prop:truncation}
The functor $\cV^{\le n} \to \Rep^{\pol}(\GL(n))$ taking $F$ to $F(\bC^n)$ is an equivalence of symmetric tensor categories.
\end{Proposition}

\begin{proof}
By the Littlewood--Richardson rule \pref{ss:lw}, this functor preserves tensor products of simple objects. Since both categories in question are semisimple, we get the result in general.
\end{proof}

\begin{corollary} \label{cor:boundedtcasubspaces}
Let $M$ be a bounded object of $\cV$ and let $n \ge \ell(M)$.  Then the map $N \mapsto N(\bC^n)$ induces an order-preserving bijection
\begin{displaymath}
\{ \textrm{subobjects of $M$} \} \to \{ \textrm{$\GL(n)$-stable subspaces of $M(\bC^n)$} \}.
\end{displaymath}
\end{corollary}

\article
Corollary~\pref{cor:boundedtcasubspaces} gives the fundamental principle of bounded objects:  one can evaluate on $\bC^n$ for $n$ sufficiently large and not lose information.  This can be used to reduce many questions about bounded tca's to questions about finitely generated $\bC$-algebras.  In the next few results, we begin to see how this works.

\begin{Proposition} \label{prop:boundedtcasubmodules}
Let $A$ be a bounded tca.  Then any finitely generated $A$-module is bounded.  If $M$ is a finitely generated $A$-module and $n \ge \ell(M)$ then $N \mapsto N(\bC^n)$ induces an order-preserving bijection
\begin{displaymath}
\{ \textrm{$A$-submodules of $M$} \} \to \{ \textrm{$\GL(n)$-stable $A(\bC^n)$-submodules of $M(\bC^n)$} \}
\end{displaymath}
\end{Proposition}

\begin{proof}
Let $M$ be a finitely generated $A$-module.  Choose a finite length object $V$ such that $M$ is a quotient of $A \otimes V$.  Then $V$ is bounded since it has finite length, and so $A \otimes V$ is bounded, and so $M$ is bounded.  This proves the first statement.  As to the second, we know that the map is injective by Corollary~\pref{cor:boundedtcasubspaces}.  It thus suffices to show it is surjective.  Thus let $N_0$ be a $\GL(n)$-stable $A(\bC^n)$-submodule of $M(\bC^n)$.  By Corollary~\pref{cor:boundedtcasubspaces}, there is a unique subobject $N$ of $M$ such that $N_0=N(\bC^n)$.  We must show that $N$ is an $A$-submodule of $M$, i.e., that the dotted arrow can be filled in
\[
\xymatrix{ A \otimes N \ar[r] \ar@{-->}[rd] & M \\ & N \ar[u] }
\]
By our assumptions on $N_0$, this is true after evaluating on $\bC^n$, so by Proposition~\pref{prop:truncation}, it is true after we apply $\tau^{\le n}$. But $M = \tau^{\le n}(M)$, so the multiplication map factors as $A \otimes N \to \tau^{\le n}(A \otimes N) \to M$. Set the dotted arrow to be the composition $A \otimes N \to \tau^{\le n}(A \otimes N) \to \tau^{\le n}(N) = N$.
\end{proof}

\begin{remark}
The bijection $N \mapsto N(\bC^n)$ preserves intersections and sums of submodules.  This can either be seen directly, or from the fact that these operations are determined by the poset structure on the set of submodules.
\end{remark}

\begin{corollary} \label{cor:boundedideals}
Let $A$ be a bounded tca and let $n \ge \ell(A)$.  Then $\fa \mapsto \fa(\bC^n)$ gives an order-preserving bijection between ideals of $A$ and $\GL(n)$-stable ideals of $A(\bC^n)$.
\end{corollary}

\begin{Proposition}
\label{cor:noeth}
Let $A$ be a bounded tca.  Then the following are equivalent:
\begin{enumerate}[\rm (a)]
\item $A$ is noetherian.
\item $A$ is weakly noetherian.
\item $A_0$ is noetherian and $A$ is finitely generated over $A_0$.
\item $A(\bC^n)$ is noetherian for all $n$ (or even all $n \gg 0$).
\end{enumerate}
\end{Proposition}

\begin{proof}
We have previously shown (a) $\implies$ (b) $\implies$ (c) $\implies$ (d) (Proposition~\pref{finite:imps}), so we must show that (d) implies (a).  Thus let $A$ be a bounded tca such that $A(\bC^n)$ is noetherian for all $n>N$.  Let $M$ be a finitely generated $A$-module and let $n \ge \max(N, \ell(M))$.  An ascending chain of $A$-submodules of $M$ yields an ascending chain of $\GL(n)$-stable $A(\bC^n)$-submodules of $M(\bC^n)$.  Such a chain necessarily stabilizes, since $A(\bC^n)$ is a noetherian ring and $M(\bC^n)$ is a finitely generated $A(\bC^n)$-module.  Thus, by Proposition~\pref{prop:boundedtcasubmodules}, the original chain stabilizes as well.  This shows that $M$ is a noetherian $A$-module, and so $A$ is noetherian.
\end{proof}

\subsection{Examples}

\begin{Proposition}
The tca $A=\Sym(U\langle 1 \rangle)$ is bounded and noetherian if $\dim(U)$ is finite.  In fact, $\ell(A)=\dim(U)$.
\end{Proposition}

\begin{proof}
The computation of $\ell(A)$ follows from the decomposition \eqref{eq:multipol-1} and the remarks that follow it, while noetherianity follows from Proposition~\pref{cor:noeth}.
\end{proof}

\begin{corollary}
A finitely generated tca that is generated in degree $1$ is bounded and noetherian.
\end{corollary}

\article
The great divide between tca's generated in degree 1 and those generated in higher degrees is the following:

\begin{proposition} \label{prop:rankvar}
The tca $\Sym(\bS_\lambda)$ is unbounded if $|\lambda| > 1$.
\end{proposition}

\begin{proof}
If this were not true, say that $\ell(\Sym(\bS_\lambda)) \le n$. Then every element of $\bS_\lambda(\bC^\infty)$ belongs to a subspace $\bS_\lambda(W)$ where $\dim W = n$. In more geometric terms, consider $\bS_\lambda(\bC^m)$ as an affine space. Consider the Grassmannian $\Gr(n,\bC^m)$ of $n$-dimensional subspaces of $\bC^m$ and let $\mc{R}$ be its rank $n$ tautological bundle. We have an inclusion $\bS_\lambda \cR \subset \bS_\lambda(\bC^m) \times \Gr(n,\bC^m)$. The image $X_\lambda^{\le n}$ of $\bS_\lambda \cR$ under the projection onto the first factor is precisely all those elements of $\bS_\lambda(\bC^m)$ that belong to a subspace $\bS_\lambda(W)$ where $\dim W = n$. So the dimension of $X_\lambda^{\le n}$ is bounded above by $\rank \bS_\lambda \mc{R} + \dim \Gr(n,\bC^m)$. The first term is a constant (depending only on $n$) and the second term is $n(m-n)$, a linear polynomial in $m$. But $\dim \bS_\lambda(\bC^m)$ is a polynomial in $m$ of degree $|\lambda|>1$ (see \pref{art:hookcontent}), so we get a contradiction.
\end{proof}

\begin{corollary}
Let $V$ be a finite length object of $\cV$.  Then $\Sym(V)$ is bounded if and only if $V$ is concentrated in degrees $\le 1$.
\end{corollary}

\article
Note however that there are non-polynomial tca's which are bounded but not generated in degree 1. For example, the coordinate rings of the varieties $X_\lambda^{\le n}$ in the proof of Proposition~\pref{prop:rankvar} are bounded tca's. These are the {\bf rank varieties} which were studied in \cite{porras} (see also \cite[Chapter 7]{weyman}). The basic examples of rank varieties are in the tca's $\Sym(U\langle 1\rangle)$, $\Sym(\Sym^2)$, and $\Sym(\bigwedge^2)$. These include the determinantal varieties mentioned in \pref{ss:detlvar}.

\subsection{Domains and prime ideals}
\label{ss:bdprime}

\begin{Lemma}
\label{ind:2}
Let $A$ be an object of $\mc{V}$ and let $x \in A(\bC^n)$.  Then the locus in $\Hom(\bC^n, \bC^m)$ consisting of maps $f$ such that $f(x)=0$ is a Zariski closed subset.
\end{Lemma}

\begin{proof}
This is clear.
\end{proof}

\begin{Lemma}
\label{ind:3}
Let $A$ be an object of $\mc{V}$ with $\ell(A) \le n$ and let $x \in A(\bC^m)$ be non-zero, with $m \ge n$.  Then there exists a map $f \colon \bC^m \to \bC^n$ such that $f(x) \ne 0$.
\end{Lemma}

\begin{proof}
One immediately reduces to the case where $A=\bS_{\lambda}$ with $\ell(\lambda) \le n$.  Let $V$ be the set of elements $x \in A(\bC^m)$ such that $f(x)=0$ for all $f \colon \bC^m \to \bC^n$.  Since each such $f$ induces a linear map $A(\bC^m) \to A(\bC^n)$, the space $V$ is a vector subspace of $A(\bC^m)$, and is clearly stable by $\GL(m)$.  As $A(\bC^m)$ is irreducible under $\GL(m)$ and $V$ is not all of $A(\bC^m)$, we have $V=0$, which completes the proof.
\end{proof}

\begin{Proposition}
\label{ind:4}
Let $A$ be a bounded tca and let $n \ge \ell(A)$.  Then $A$ is a domain if and only if $A(\bC^n)$ is.
\end{Proposition}

\begin{proof}
If $A$ is a domain then $A(\bC^n)$ is for any $n$.  Thus suppose that $A(\bC^n)$ is a domain, with $n \ge \ell(A)$.  We may as well suppose that $\ell(A)<\infty$, otherwise we are already done.  Let $m>n$.  Suppose that $A(\bC^m)$ is not a domain and let $x$ and $y$ in $A(\bC^m)$ satisfy $xy=0$.  Let $X$ (resp.\ $Y$) be the locus in $\Hom(\bC^m, \bC^n)$ consisting of maps $f$ such that $f(x)$ (resp.\ $f(y)$) is zero.  Then $X$ and $Y$ are Zariski closed subsets of $\Hom(\bC^m, \bC^n)$ by Lemma~\pref{ind:2}, and neither is the full space $\Hom(\bC^m, \bC^n)$ by Lemma~\pref{ind:3}.  It follows that $X \cup Y$ is not the full space.  We can thus pick $f \colon \bC^m \to \bC^n$ which does not belong to $X$ or $Y$, and so $x'=f(x)$ and $y'=f(y)$ are non-zero elements of $A(\bC^n)$.  Since the map $f \colon A(\bC^m) \to A(\bC^n)$ is a ring homomorphism, we have $x'y'=0$, and so $A(\bC^n)$ is not a domain.
\end{proof}

\begin{Proposition} \label{prop:primebijection}
Let $A$ be a bounded tca and let $n \ge \ell(A)$.  Then $\fp \mapsto \fp(\bC^n)$ gives an order-preserving bijection between prime ideals of $A$ and $\GL(n)$-stable prime ideals of $A(\bC^n)$.
\end{Proposition}

\begin{proof}
By Corollary~\pref{cor:boundedideals}, $\fp \mapsto \fp(\bC^n)$ is an order-preserving bijection between ideals of $A$ and $\GL(n)$-stable ideals of $A(\bC^n)$. Proposition~\pref{ind:4} says that $\fp$ is prime if and only if $\fp(\bC^n)$ is.
\end{proof}

\begin{Proposition}
\label{ind:5}
Let $A$ a noetherian bounded tca and let $n \ge \ell(A)$.  Then $\fp \mapsto \fp(\bC^n)$ gives a bijection between minimal primes of $A$ and minimal primes of $A(\bC^n)$.
\end{Proposition}

\begin{proof}
This follows from Proposition~\pref{prop:primebijection} and Proposition~\pref{minprime}.
\end{proof}

\begin{corollary}
A noetherian bounded tca has finitely many minimal primes.
\end{corollary}

\begin{Proposition}
\label{ind:6}
The nilradical of a noetherian bounded tca is nilpotent.
\end{Proposition}

\begin{proof}
Let $n=\ell(A)$.  Since $\rad(A)(\bC^n)$ is the nilradical of $A(\bC^n)$, it is nilpotent, and so $\rad(A)(\bC^n)^k=0$ for some $k$.  Since $\rad(A)^k(\bC^n)=\rad(A)(\bC^n)^k$, we have $\rad(A)^k(\bC^n)=0$.  This implies $\rad(A)^k=0$, as $\ell(A)=n$.
\end{proof}

\begin{remark}
The above argument shows that in an arbitrary noetherian tca, given any $N$ there exists $n$ so that all partitions appearing in $\rad(A)^n$ have at least $N$ rows.
\end{remark}

\begin{Corollary}
\label{cor:bd-equiv}
For noetherian bounded tca's we have the following:
\begin{enumerate}[\rm (a)]
\item ``Domain'' is equivalent to ``weak domain.''
\item ``Prime'' is equivalent to ``weakly prime'' (for ideals).
\item ``Reduced'' is equivalent to ``weakly reduced.''
\item ``Nilpotent'' is equivalent to ``strongly nilpotent'' (for elements).
\end{enumerate}
\end{Corollary}

\begin{proof}
Let $A$ be a noetherian bounded tca.  Let $\fp$ be a weak prime of $A$.  Then $A/\fp$ is a weak domain which is still noetherian and bounded; by the above proposition its nilradical is nilpotent and so by Proposition~\pref{weak:rad} it is a domain, i.e., $\fp$ is prime.  This proves (b), and (a) follows by considering the zero ideal.  From (b) and Propositions~\pref{rad1} and~\pref{rad2}, we see that $\rad(A)=\srad(A)$.  This proves (c) and (d).
\end{proof}

\subsection{Noetherian induction}

\noarticle
We have the following version of a noetherian induction type result in the bounded case:

\begin{Proposition}
\label{ind:1}
Let $A$ be a noetherian bounded tca.  Let $\mc{P}$ be a property of ideals of $A$, i.e., a function assigning to each ideal $I$ of $A$ a boolean value $\mc{P}(I)$.  Suppose that
\begin{enumerate}[\rm (a)]
\item If $I$ is a prime ideal of $A$ and $\mc{P}(J)$ holds for all $J$ properly containing $I$ then $\mc{P}(I)$ holds.
\item If $\mc{P}(I)$ holds and $I \subset J$ then $\mc{P}(J)$ holds.
\item If $\mc{P}(I)$ and $\mc{P}(J)$ hold then $\mc{P}(IJ)$ holds.
\end{enumerate}
Then $\mc{P}(I)$ holds for all ideals $I$.
\end{Proposition}

\begin{Lemma}
\label{ind:8}
Let $A$ be a noetherian bounded tca and let $I$ be an ideal of $A$.  Then either $I$ is prime or there exist prime ideals $J_1, \ldots, J_n$ properly containing $I$ and $k \ge 1$ such that $(J_1 \cdots J_n)^k$ is contained in $I$.
\end{Lemma}

\begin{proof}
Suppose $I$ is not prime.  Write $\rad(I)=\bigcap_{i=1}^n J_i$ with $J_i$ prime, per Proposition~\pref{rad1}.  We have $I \subset J_i$ for each $i$, and since $I$ is not prime this containment is proper.  Let $\ol{J}_i$ denote the image of $J_i$ in $A/I$.  Then $\rad(A/I)=\bigcap_{i=1}^n \ol{J}_i$.  Thus $\bigcap_{i=1}^n \ol{J}_i$ is nilpotent by Proposition~\pref{ind:6}.  Since $\ol{J}_1 \cdots \ol{J}_n$ is contained in $\bigcap_{i=1}^n \ol{J}_i$, it too is nilpotent.  Letting $k$ be such that $(\ol{J}_1 \cdots \ol{J}_n)^k=0$, we see that $(J_1 \cdots J_n)^k \subset I$.
\end{proof}

\begin{Lemma}
\label{ind:9}
Let $\mc{S}$ be a poset satisfying the ascending chain condition and let $\mc{P}$ be a property of elements of $\mc{S}$, i.e., a function assigning to each $x \in \mc{S}$ a boolean value $\mc{P}(x)$.  Suppose that $\mc{P}$ satisfies the following condition:
\begin{itemize}
\item[($\ast$)] If $\mc{P}(y)$ holds for all $y>x$ then $\mc{P}(x)$ holds.
\end{itemize}
Then $\mc{P}(x)$ holds for all $x$.
\end{Lemma}

\begin{proof}
This is standard.
\end{proof}

\article
We now prove Proposition~\pref{ind:1}.  Let $\mc{S}$ be the set of ideals of $A$, given the structure of a poset via the relationship of inclusion.  This poset satisfies the ascending chain condition since $A$ is noetherian.  Let $I$ be an ideal of $A$ and suppose that $\mc{P}(J)$ holds for all $J$ strictly containing $I$.  If $I$ is prime then $\mc{P}(I)$ holds by (a).  Thus suppose that $I$ is not prime.  By Lemma~\pref{ind:8} we can find ideals $J_1, \ldots, J_n$ which properly contain $I$ such that $(J_1 \cdots J_n)^k$ is contained in $I$ for some $k \ge 1$.  Thus $\mc{P}(I)$ holds by (b) and (c).  The proposition now follows from Lemma~\pref{ind:9}.

\xsection{Existing applications} \label{sec:existingapps}

\subsection{Construction of pure resolutions} \label{ss:efw}

\noarticle
The constructions in this section appeared in \cite{efw} and were motivated by the Boij--S\"oderberg conjectures \cite{boijsoderberg}, which were solved in \cite{es:bs}. The specifics of the conjectures will not be discussed here, but we refer the reader to the survey articles \cite{es:survey} and \cite{floystad}. The constructions also appear in a different language in \cite{olver}.

First consider the polynomial ring $A = \bC[x_1,\dots,x_n]$, which we consider as a graded ring with $\deg(x_i) = 1$. Given a finitely generated graded $A$-module $M$, we note that the Tor modules $\Tor_i^A(M,\bC)$ are naturally graded. These can also be interpreted as the generators for the $i$th term of a minimal $A$-free resolution for $M$. We say that $M$ has a {\bf pure resolution} of type $(d_0, d_1, \dots)$ if for each $i$, we have that $\Tor_i^A(M,\bC)_j \ne 0$ if and only if $j = d_i$. By minimality, a necessary condition for such a sequence to be realized is that $d_0 < d_1 < \cdots$.

An equivariant construction for pure resolutions for any degree sequence $(d_0, d_1, \dots, d_n)$ was given in \cite{efw}. So now we write $A = \Sym(V)$ for an $n$-dimensional vector space $V$, which carries the action of $G = \GL(V)$.

Let $\alpha$ be a partition with $\ell(\alpha) \le n$ and let $\beta$ be a partition with $\beta_1 > \alpha_1$ and $\beta_i = \alpha_i$ for $i\ge 2$. Set $e_1 = \beta_1 - \alpha_1$ and $e_i = \alpha_{i-1} - \alpha_i + 1$ for all $i \ge 2$, and define $d_i = e_1 + e_2 + \cdots + e_i$ (we set $d_0 = 0$). Also define the partitions
\begin{align}
\alpha(d,i) = (\alpha_1 + e_1, \alpha_2 + e_2, \dots, \alpha_i + e_i, \alpha_{i+1}, \dots, \alpha_n)
\end{align}
with the convention that $\alpha(d,0) = \alpha$. We define graded free $A$-modules by
\begin{align}
\mb{F}_i = \bS_{\alpha(d,i)}(V) \otimes A(-d_i)
\end{align}
for $i=0,1,\dots,n$.

Note that $\alpha(d,i) / \alpha(d,i-1) \in \HS_{e_i}$ for all $i=1,\dots,n$, and so from \pref{art:gln-pieri}, we see that we have a unique up to scalar map
\begin{align}
\bS_{\alpha(d,i)}(V) \to \bS_{\alpha(d,i-1)}(V) \otimes \Sym^{e_i}(V)
\end{align}
which can be extended to an $A$-linear map
\begin{align}
d_i \colon \mb{F}_i \to \mb{F}_{i-1}.
\end{align}
It is also clear from Pieri's rule \pref{art:gln-pieri} that no $G$-equivariant maps $\bS_{\alpha(d,i)}(V) \to \mb{F}_{i-2}$ exist, so $\mb{F}_\bullet$ is a complex. 

One of the main results of \cite{efw} is that $\mr{H}_i(\mb{F}_\bullet) = 0$ for $i>0$, and hence $\mb{F}_\bullet$ is a pure resolution of $M = \mr{H}_0(\mb{F}_\bullet)$ (this also follows from \cite[Theorem 8.11]{olver}). This can also be proved using the following lemma: if $\bS_\lambda(V) \subset \bS_{\alpha(d,i)}(V) \otimes A$ and $\bS_\lambda(V) \subset \bS_{\alpha(d,i-1)}(V) \otimes A$, then its image under the differential $d_i$ is nonzero, and hence $d_i$ maps it isomorphically onto its image by irreducibility. This lemma appears in \cite[\S 8]{olver} and \cite[Lemma 1.6]{sw}. These differentials can be explicitly calculated using the {\tt Macaulay 2} package {\tt PieriMaps} written by the first author \cite{pierimaps}.

Although we started with a pair of partitions $\alpha, \beta$ and produced the degree sequence, it is easy to see that one can reverse the construction, and construct partitions $\alpha, \beta$ that give any degree sequence $(d_0 < d_1 < \cdots)$.

Now we point out that the dimension of $V$ was not important in the above discussion. In fact, the construction above makes since when we replace $V$ with $\bC^\infty$, so we get pure resolutions over the tca $\Sym(\bC\langle 1\rangle)$. Note that the complex $\mb{F}_\bullet$ becomes eventually linear (cf. Theorem~\pref{thm:fglinearstrands}).

\subsection{FI-modules} \label{ss:fimod}

\noarticle
FI-modules were introduced in \cite{fimodules} and are equivalent to modules over the tca $\Sym(\bC\langle 1 \rangle)$.  It was observed that many natural sequences of $S_n$-representations, such as the $i$th cohomology group of the configuration space of $n$ points of a fixed manifold $X$, carry the structure of a finitely generated $\Sym(\bC\langle 1 \rangle)$-module.  Therefore, algebraic results about $\Sym(\bC\langle 1 \rangle)$-modules yield (typically new) information about these sequences of representations.  For instance, the dimension of the $n$th graded piece of a finitely generated $\Sym(\bC\langle 1 \rangle)$-module is a polynomial in $n$, at least for $n \gg 0$, and so this property transfers to all examples.

\subsection{$\Delta$-modules} \label{ss:deltamod}

\noarticle
$\Delta$-modules were introduced in \cite{snowden} with the purpose of studying syzygies of the Segre embedding, and related embeddings.  We recall the definition.  A {\bf $\Delta$-module} consists of the following:
\begin{itemize}
\item For each non-negative integer $n$, a polynomial functor $F_n \colon \Vec^n \to \Vec$.
\item An $S_n$-equivariant structure on $F_n$.
\item A natural transformation
\begin{displaymath}
F_n(V_1, \ldots, V_{n-1}, V_{n} \otimes V_{n+1}) \to F_{n+1}(V_1, \ldots, V_{n+1}).
\end{displaymath}
\end{itemize}
The data are required to satisfy certain compatibility conditions.  The notion of a morphism of $\Delta$-modules is evident, and the resulting category of $\Delta$-modules is abelian.

The main examples of $\Delta$-modules are given by syzygies of Segre embeddings.  To be precise, fix $p \ge 0$.  Let $F_n(V_1, \ldots, V_n)$ be the $p$th syzygy module of the Segre embedding
\begin{displaymath}
\bP(V_1) \times \cdots \times \bP(V_n) \to \bP(V_1 \otimes \cdots \otimes V_n).
\end{displaymath}
Then the sequence $\{F_n\}$ constitutes a $\Delta$-module.  The map from $F_n$ to $F_{n+1}$ comes from an obvious commutative triangle of Segre embeddings.

The main results of \cite{snowden} state that finitely generated $\Delta$-modules are noetherian and have rational Hilbert series.  (Actually, these results are only proved for a certain class of $\Delta$-modules, the ``small'' ones.)  These results apply in particular to the $\Delta$-modules coming from syzygies of Segre embeddings, and show, in a sense, that there is only a finite amount of data in the $p$-syzygies.

The connection with twisted commutative algebras is this:  if $F$ is a small $\Delta$-module then, for fixed $V$, the sequence $\{F_n(V, \ldots, V)\}$ admits the structure of a finitely generated module over a finitely generated bounded tca.  Furthermore, the formation of this module often does not lose very much information.  The main results about $\Delta$-modules are deduced from corresponding results about modules over tca's.  As far as we are aware, this is the first application of the theory of tca's more general than $\Sym(\bC\langle 1 \rangle)$.

\subsection{Invariant theory}

\noarticle
Let $X$ be a projective variety and let $\mc{L}$ be a line bundle on $X$.  For an integer $n$, let $B_n$ denote the global sections of $\mc{L}^{\boxtimes n}$ on $X^n$.  Then $B_n$ has an action of $S_n$ and outer multiplication of sections defines a map $B_n \otimes B_m \to B_{n+m}$.  Thus $B$ forms a twisted commutative algebra.  If a group $G$ acts on $X$ and $\mc{L}$ is $G$-equivariant then $B$ carries an action of $G$, and $A_n=B_n^G$ can be identified with the global sections of the bundle on the GIT quotient $X^n/\!\!/G$ induced by $\mc{L}^{\boxtimes n}$.  The algebras $A$ gotten in this manner are bounded and finitely generated, and thus noetherian, though typically not generated in degree 1.

This point of view on GIT quotients can be very useful.  For instance, let $A$ be the tca coming from the above set-up with $X=\bP^1$, $\mc{L}=\mc{O}(k)$ and $G=\SL(2)$.  Elements of $A_n$ can be represented as regular graphs on $n$ vertices of degree $k$, modulo certain relations called the Pl\"ucker relations; see \cite[\S 2]{hmsv}.  The fact that $A$ is finitely generated can be interpreted combinatorially as follows:  there is an integer $r$ such that any regular degree $k$ graph on $>r$ vertices can be written as a sum of disconnected graphs using the Pl\"ucker relations.  This was proven ``by hand'' in \cite[\S 6]{hmsv} for certain small values of $k$.  The proofs given there are very combinatorial in nature, and did not give a conceptual explanation of this phenomenon; in particular, it was not clear if it would continue for higher values of $k$.  The perspective offered by tca's shows that it does, even for general $X$.

Let us give one more instance where the tca point of view is useful, again from \cite{hmsv}.  The main theorem of loc.\ cit.\ states that the defining equations of $(\bP^1)^n/\!\!/\SL(2)$ are quadratic, except when $n=6$ where there is a single cubic relation.  Furthermore, a generating set of quadratic relations are given and it is evident that in some sense they all come from the $n=8$ case.  This can be made precise using tca's:  the quadratic part of the ideal forms a module over a tca, and the theorem from loc.\ cit.\ simply states that it is generated in degree 8.

\section{Announcement of new results} \label{sec:announce}

\subsection{Hilbert series}

\article
The results from this section are from \cite{hilbert}, for the most part.  Let $M$ be a graded-finite object of $\cV$, taken in the sequence model.  We define its {\bf Hilbert series} by
\begin{align}
H_M(t)=\sum_{n \ge 0} \dim(M_n) \frac{t^n}{n!}.
\end{align}
We then have the following rationality theorem, proved in \cite[Thm.~3.1]{snowden}:  if $M$ is finitely generated over a tca finitely generated in degree 1 then $H_M(t)$ is a polynomial in $t$ and $e^t$.

\article
We now discuss two generalizations of this theorem.  The first concerns a rationality result for a modification of the Hilbert series.  To define it, we must first introduce some notation.  Let $\lambda$ be a partition.
\begin{itemize}
\item We write $c_{\lambda}$ for the conjugacy class in $S_{\vert \lambda \vert}$ corresponding to $\lambda$ (see \pref{ss:conj}).
\item We write $t^{\lambda}$ for $t_1^{m_1(\lambda)} t_2^{m_2(\lambda)} \cdots$, where $m_i(\lambda)$ denotes the number of times $i$ appears in $\lambda$.
\item We write $\lambda!$ for $m_1(\lambda)! m_2(\lambda)! \cdots$.
\end{itemize}
For example, if $\lambda=(2,1,1,1)$ then $c_{\lambda}$ is the conjugacy class of transpositions in $S_5$, $t^{\lambda}$ is $t_1^3 t_2$ and $\lambda!=3! \cdot 1!=6$.

Let $M$ be a graded-finite object of $\cV$.  We define its {\bf enhanced Hilbert series} by
\begin{align}
\wt{H}_M(t)=\sum_{\lambda} \trace(c_{\lambda} \vert M) \frac{t^{\lambda}}{\lambda!}.
\end{align}
The isomorphism class of $M$ is completely determined by $\wt{H}_M$.  The enhanced Hilbert series therefore contains much more information than the usual Hilbert series.  In fact, the usual Hilbert series is recovered easily from the enhanced Hilbert series by putting $t_i=0$ for $i>1$.  The enhanced Hilbert series is multiplicative:
\begin{align}
\wt{H}_{M \otimes N}=\wt{H}_M \wt{H}_N
\end{align}
(see, for example, \cite[Proposition 7.18.2]{stanley}).  It follows that the map
\begin{align}
K(\cV_{\gfin}) \otimes \bQ \to \bQ \lbb t_i \rbb, \qquad M \mapsto \wt{H}_M
\end{align}
is an isomorphism of rings.

\article
To state our main theorem on these series, we introduce a bit more notation.  For $n \ge 0$, put
\begin{align}
T_n=\sum_{i \ge 1} i(i-1) \cdots (i-n+1) t_i.
\end{align}

\begin{theorem}
Let $M$ be a finitely generated module over a tca finitely generated in degree $1$.  Then $\wt{H}_M$ belongs to $\bQ[t_i, T_j, \exp(T_0)]_{i,j \ge 1}$.
\end{theorem}

\article
We now describe our second result on Hilbert series.  The theorem that $H_M$ belongs to $\bQ[t,e^t]$ can be stated equivalently using differential equations:  it amounts to the existence of a polynomial $p$, whose roots are non-negative integers, such that $p(d/dt) H_M=0$.  Now, the operation $d/dt$ on Hilbert series is induced by the Schur derivative (see \pref{ss:deriv}) on $\cV$; that is, we have
\begin{align}
H_{\bD{M}}=\frac{d}{dt} H_M.
\end{align}
This follows immediately from the description of each in the sequence model.  Suppose now that $M$ is an $A$-module, and that $V$ is a subspace of $A_1$ (or, more generally, any space mapping to $A_1$).  We have a multiplication map
\begin{align}
A_1 \otimes M_n \to M_{n+1},
\end{align}
which, when restricted to $V$, yields a map
\begin{align}
V \otimes M \to \bD{M}.
\end{align}
We define $\partial_V(M)$ to be the complex $[V \otimes M \to \bD{M}]$.  More generally, for a complex of $A$-modules $M$ the above process gives a map of complexes $V \otimes M \to \bD{M}$, and we define $\partial_V(M)$ to be the cone on this complex.  The operation $\partial_V$ lifts the operation $d/dt-\dim{V}$ on Hilbert series.  It thus follows that we can find spaces $V_1, \ldots, V_n$ such that $\partial_{V_1} \cdots \partial_{V_n}(M)$ has Hilbert series 0.  The obvious question then is whether one can choose the $V_i$ so that $M$ itself is annihilated; this is answered affirmatively by our result:

\begin{theorem}
Let $A$ be a tca finitely generated in degree $1$ and let $M$ be a finitely generated $A$-module.  Then there exist subspaces $V_1, \ldots, V_n$ of $A_1$ such that the complex $\partial_{V_1} \cdots \partial_{V_n}(M)$ is acyclic.
\end{theorem}

This result can be viewed as providing a kind of system of parameters for modules over tca's.  To explain, consider the graded polynomial ring $A=\Sym(V)$.  For an $A$-module $M$ supported in non-negative degrees, let $\bD(M)$ be the $A$-module which is $M_{n+1}$ in degree $n \ge 0$, and 0 in negative degrees.  We have a multiplication map $A_1 \otimes M_n \to M_{n+1}$, and so given $V \subset A_1$ we can define $\partial_V(M)$ to be the complex $[V \otimes M \to \bD{M}]$.  The existence of a system of parameters for $M$ is equivalent to the existence of a space $V$ such that $\partial_V(M)$ has finite length homology.  Note that an $A$-module being finite length is equivalent to it being annihilated by some power of $\bD=\partial_0$, and so a system of parameters for $M$ gives a differential operator annihilating $M$ and vice versa.

\begin{Remark}
There is a common generalization of the above two theorems on Hilbert series.  We have not established this statement yet, but hope to soon.
\end{Remark}

\subsection{Finiteness properties of resolutions} \label{ss:finiteres}

\article
The results from this section are from \cite{koszul}.  The tca $A=\Sym(U\langle 1 \rangle)$ has infinite global dimension:  for example, the Koszul resolution of the simple module $\bC$ has infinite length.  Nonetheless, we have shown that the resolution of a finitely generated $A$-module has only a finite amount of data in it.  To state this result precisely, let $M$ be a finitely generated $A$-module and put
\begin{align}
\mc{T}_n(M)=\bigoplus_{p \ge 0} \Tor_p^A(M, \bC)_{p+n}.
\end{align}
The terms of the resolution of $M$ are recoverable from the $\mc{T}_n(M)$.  Standard properties of the Koszul complex show that $\mc{T}_n(M)$ is a comodule over $\bw{}(U\langle 1 \rangle)$.  Therefore $\mc{T}_n(M)^{\vee}$ is a module over $\bw{}(U^*\langle 1 \rangle)$.  Our main result is then:

\begin{theorem} \label{thm:fglinearstrands}
In the above situation, $\mc{T}_n(M)^{\vee}$ is finitely generated over $\bw{}(U^*\langle 1 \rangle)$ for all $n$ and non-zero for only finitely many $n$.
\end{theorem}

\article
Recall that the quantity
\begin{align}
{\rm reg}(M) = \max_i \{j \mid \Tor^A_i(M,\bC)_{i+j} \ne 0\}
\end{align}
is called the {\bf regularity} of $M$.  Clearly, the regularity of $M$ is just the largest value of $n$ for which $\mc{T}_n(M)$ is non-zero.  We thus obtain the following corollary:

\begin{corollary}
The regularity of a finitely generated $A$-module is finite.
\end{corollary}

\article
The $\delta$-functor $\{\mc{T}^{\vee}_n\}$ lifts to a functor $\mc{T}^{\vee}$ between derived categories, and induces the following version of Koszul duality.

\begin{theorem}
The functor $\mc{T}^{\vee}$ gives an anti-equivalence of $\rD^b(A)$ with $\rD^b(\bw{}(U^*\langle 1 \rangle))$.
\end{theorem}

Here we use the derived categories of finitely generated modules.  The corresponding result for polynomial rings is equivalent to the fact that finitely generated modules have finite resolutions.  Thus the above result allows us to make the following statement:  $A$-modules lack finite projective resolutions only because $\bw{}(U^*\langle 1 \rangle)$ has infinite length.

\article
We can take Koszul duality one step further, as follows.  The transpose of $\bw{}(U^*\langle 1 \rangle)$ is the tca $A'=\Sym(U^*\langle 1 \rangle)$.  Furthermore, transpose puts the categories of $\bw{}(U^*\langle 1 \rangle)$-modules and $A'$-modules in equivalence.  In particular, if $M$ is an $A$-module then $\mc{F}_n(M)=(\mc{T}_n^{\vee})^{\dag}$ is an $A'$-module.  

\begin{theorem}
The functor $\mc{F}$ gives an anti-equivalence of $\rD^b(A)$ with $\rD^b(A')$.
\end{theorem}

Choosing an identification of $U$ with $U^*$ gives an identification of $A$ with $A'$, and so:

\begin{corollary}
The category $\rD^b(A)$ is equivalent to its opposite.
\end{corollary}

We call the functor $\mc{F}$ the {\bf Fourier transform}.  Nothing like it exists in usual commutative algebra due to the lack of the transpose operation.

\article
We now give an application of the above results.  Define the {\bf Poincar\'e series} of an $A$-module $M$ by
\begin{align}
P_M(t,q)=\sum_{n \ge 0} (-q)^n H_{\Tor^A_n(M, \bC)}(t).
\end{align}
This is a more subtle invariant than the Hilbert series.  For instance, it detects information about the $A$-module structure on $M$ (which the Hilbert series does not) and it is not additive in short exact sequences.  In typical commutative algebra, the Poincar\'e series of a module over a polynomial ring is obviously a polynomial, while Poincar\'e series over more general rings need not have any sort of rationality properties \cite[\S 7]{anick}.  In the case of $A$-modules that we are now considering, the situation is more interesting:  the Poincar\'e series is typically infinite, but since the ring $A$ is ``nice'' one can expect it to be well-behaved.  And indeed this is the case:

\begin{theorem}
Let $M$ be a finitely generated $A$-module.  Then $P_M$ belongs to $\bQ[t,q^{\pm 1},e^{-tq}]$
\end{theorem}

\begin{proof}
A simple manipulation shows that
\begin{align*}
P_M(t,q)=\sum_{n \ge 0} q^{-n} H_{\mc{F}_n(M)}(-tq),
\end{align*}
where $\mc{F}_n$ denotes the homology of $\mc{F}$.  Since the above sum is finite and each $\mc{F}_n(M)$ is finitely generated over $A'$, the result follows from what we know about Hilbert series.
\end{proof}

\subsection{Depth and dimension}

\article
The results of this section are from \cite{koszul} (see also \cite{symc1} for the case $\dim U = 1$).  Let $A=\Sym(U\langle 1 \rangle)$ and let $M$ be a finitely generated $A$-module. Let $d_M(n)$ be the depth of the $A(\bC^n)$-module $M(\bC^n)$. Then this function is well-behaved:

\begin{proposition}
There are integers $0 \le a \le \dim U$ and $b$ such that $d_M(n) = an+b$ for $n \gg 0$. If $a = \dim U$, then $M$ is projective and $b=0$.
\end{proposition}

Specialized to $\dim U = 1$, we can also define the depth of $M$ to be the depth of the $A(\bC^{\infty})$-module $M(\bC^{\infty})$, forgetting the $\GL$-action. 

\begin{corollary}
Let $\dim U = 1$ and suppose $M$ is not projective.  Then the depth of $M$ is finite, and agrees with the depth of $M(\bC^n)$ (considered as an $A(\bC^n)$-module) for all $n \gg 0$.
\end{corollary}

\article
A similar result holds for the Krull dimension of $A$-modules. Since the Krull dimension of a module is defined to be the Krull dimension of its support algebra, we state the result for quotients of $A$. Let $B$ be a quotient of $A$ and define $\dim_B(n)$ to be the Krull dimension of $B(\bC^n)$.

\begin{proposition}
Let $B$ be a quotient of $A$. Then there exist integers $0 \le a \le \dim U$ and $0 \le b \le (\dim U - a)a$ such that $\dim_B(n) = an+b$ for $n \gg 0$.

Now assume that $B$ is a domain. If $b=0$, then $B = \Sym(U'\langle 1 \rangle)$ for some $a$-dimensional quotient $U'$ of $U$. If $b=(\dim U - a)a$, then $B$ is the determinantal variety of rank $\le a$ matrices.
\end{proposition}

\subsection{Structure theory of $\Sym(\bC\langle 1 \rangle)$-modules} \label{ss:symc1}

\noarticle

The results of this section are from \cite{symc1}.  Let $A=\Sym(\bC\langle 1 \rangle)$ and let $\Mod_A$ be the category of finitely generated $A$-modules.  Denote by $\Mod_A^{\tors}$ the category of torsion (or, equivalently, finite length) objects in $\Mod_A$ and by $\Mod_K$ the Serre quotient of $\Mod_A$ by $\Mod_A^{\tors}$.  We think of $A$ as analogous to $\bC[t]$, and so $\Mod_K$ should be thought of as modules over the ``generic point'' of $A$.  (Note that the Serre quotient $\Mod_{\bC[t]}/\Mod_{\bC[t]}^{\tors}$ is the category of vector spaces over $\bC(t)$.)  One can describe $\Mod_K$ more concretely:  it is a certain category of semi-linear representations of $\GL(\infty)$ on vector spaces over the fraction field of $A(\bC^{\infty})$ (see \pref{modK}).

We prove the following results about $\Mod_K$:
\begin{itemize}
\item $\Mod_K$ is equivalent to $\Mod_A^\tors$.
\item Every object has finite length.
\item The injective objects are the images of $A$-modules of the form $A \otimes \bS_{\lambda}$.
\item The simple objects are labeled by partitions $\lambda$. Each simple object $L_\lambda$ occurs as the socle of a unique $A \otimes \bS_{\lambda}$, and all other simple objects $L_\mu$ in the composition series of $A \otimes \bS_\lambda$ satisfy $\mu \subset \lambda$. In particular, $\Mod_K$ is a highest weight category \cite{hwcat}.
\item Every object has finite injective dimension.  In fact, we construct explicit injective resolutions of the simple objects.
\item We give an explicit description of $\Mod_K$ as the category of representations of a locally finite quiver with relations.
\item The Grothendieck group of $\Mod_K$ is free of rank 1 as a module over $K(\cV_{\fin})$.
\end{itemize}

The above results imply certain results about the structure of $\Mod_A$.  In particular, we find that its Grothendieck group is free of rank 2 over $K(\cV_{\fin})$, with the free module $[A]$ and the simple module $[\bC]$ forming a basis.  Using these structural results, we give quick proofs of our results on Hilbert series and Koszul duality discussed above (in the case of $\Sym(\bC\langle 1 \rangle)$-modules).

Furthermore, we prove the following properties about $\Mod_A$:
\begin{itemize}
\item The projective objects of $\Mod_A$ are also injective.
\item Every object has finite injective dimension. Moreover, each finitely generated $A$-module has a finite injective resolution by finitely generated injective objects.
\item All finitely generated $A$-modules $M$ have finite regularity and the linear strands of a minimal projective resolution of $M$ are finitely generated over the Koszul dual exterior algebra $\bigwedge \bC^\infty$.
\end{itemize}

Furthermore, the localization functor $\Mod_A \to \Mod_K$ has a right adjoint which allows us to define a local cohomology theory for objects of $\Mod_A$. We also explicitly calculate the local cohomology for lifts $\tilde{L}_\lambda$ of the simple objects $L_\lambda$ of $\Mod_K$. The Hilbert series of $\tilde{L}_\lambda$ gives an algebraic model of the character polynomials \cite[Example I.7.14]{macdonald}, and its local cohomology gives a homological model for the behavior of character polynomials at small parameters.

\subsection{Representation theory of infinite rank groups} \label{ss:infrank}

\noarticle
The results of this section are from \cite{reptheory}.  As we have seen, $\cV$ can be described as the category of polynomial representations of $\GL(\infty)$.  It is natural to ask if one can enlarge this category to include rational representations of $\GL(\infty)$, or if there are good categories of representations of other infinite rank reductive groups like $\bO(\infty)$ or $\Sp(\infty)$.  In fact, these categories do exist, alongside many other examples. These categories were studied, for example, in \cite{penkovstyrkas},  \cite{koszulcategory}, and \cite{olshanskii}. For expository purposes, we focus on $\bO(\infty)$.  The basic phenomena here are present in other cases as well.

Let $\Rep(\bO)$ denote the category of ``algebraic'' representations of $\bO(\infty)$, i.e., those which appear as a subquotient of a finite direct sum of tensor powers of the standard representation $\bC^{\infty}$.  The most prominent difference between $\Rep(\bO)$ and $\cV$ is that the former is not semi-simple.  A simple example of a non-split surjection is given by the quadratic form $\Sym^2(\bC^{\infty}) \to \bC$.  Now, one may think that the failure for this to split is somewhat artificial:  the quadratic form should correspond to an element of $\Sym^2(\bC^{\infty})$ which is some kind of infinite sum.  In fact, one can work with pro-finite spaces and then $\Sym^2(\bC^{\infty})$ does have an invariant; however, the map $\Sym^2(\bC^{\infty}) \to \bC$ can no longer be defined, as one might need to add up infinitely many numbers. This shows that there are two candidates for $\Rep(\bO)$, one using ``ind-finite'' spaces and one using ``pro-finite'' spaces.  In fact, the two are opposite categories of each other, so it is enough to study one of them.  We stick to the ``ind-finite'' version.

Polynomial representations of $\GL(d)$ can be described using the symmetric group via Schur--Weyl duality.  There is a similar theory for the orthogonal group $\bO(d)$, where the Brauer algebra replaces the symmetric group \cite{brauer}.  Let us recall the definition of this algebra.  Let $n$ be a non-negative integer and let $t$ be a parameter.  Elements of the Brauer algebra, denoted $\Br_n(t)$, are formal linear combinations of matchings on $2n$ points, which are represented as two rows of $n$ points.  The product of two graphs is formed by placing one above the other, removing all loops and multiplying by $t^k$, where $k$ is the number of loops removed.  The symmetric group sits inside the Brauer algebra as the matchings which are bipartite (every edge has one vertex in the top row and one in the bottom row).

Equip $\bC^d$ with the standard symmetric inner product.  We define an action of $\Br_n(d)$ on $(\bC^d)^{\otimes n}$ as follows.  An edge between vertices in the top row pairs the corresponding vectors using the inner product; an edge between two vertices in the bottom row inserts a copy of the form into those two tensor factors; and edges going from the top row to the bottom row simply move tensor factors around.  It is important here that the parameter in the Brauer algebra equals the dimension of $\bC^d$, as the norm of the pairing is $d$.  Now, it is clear that the actions of $\Br_n(d)$ and $\bO(d)$ on $(\bC^d)^{\otimes n}$ centralize each other.  For $d$ large compared to $n$, the Brauer algebra is semi-simple and the multiplicity spaces form irreducible representations of the orthogonal group.

It is tempting to try to carry out these constructions for $d=\infty$.  However, one runs into two closely related problems:  first, one cannot set the parameter in the Brauer algebra to $\infty$ in a meaningful way; and second, the action of the Brauer algebra on $(\bC^d)^{\otimes n}$ makes use of both the maps $\Sym^2(\bC^d) \to \bC$ and $\bC \to \Sym^2(\bC^d)$, whereas in $\Rep(\bO)$ we only have the former map.  The solution to these problems is to simply use the piece of the Brauer algebra which does carry over.  Precisely, we define the {\bf downwards Brauer category}, denoted $\db$, to be the category whose objects are finite sets and where a morphism $L \to L'$ is a pair $(f, \Gamma)$, where $\Gamma$ is a (possibly non-perfect) matching on $L$ and $f$ is bijection of $L \setminus V(\Gamma)$ with $L'$.  To connect this with $\Br_n(t)$, one should think of $L$ as the top row of vertices, $L'$ as the bottom row of vertices, $\Gamma$ as the edges in the top row and $f$ as edges between the two rows; there are no edges in the bottom row.  The purpose of this definition is that the rule $L \mapsto (\bC^{\infty})^{\otimes L}$ defines an object $K$ of $\Vec^{\db}$:  functoriality is obtained using the pairing.  Of course, the object $K$ also carries an action of $\bO(\infty)$.  We thus obtain a contravariant functor
\begin{align}
\Vec^{\db} \to \Rep(\bO), \qquad M \mapsto \Hom_{\db}(M, K).
\end{align}
We show that this functor is an anti-equivalence.  This provides a perfect analogue of Schur--Weyl duality for $\bO(\infty)$.  We note that there is an ``upwards'' Brauer category $\ub$, and $\Vec^{\ub}$ is opposite to $\Vec^{\db}$.  Thus the above result can be rephrased as an equivalence $\Vec^{\ub} \to \Rep(\bO)$.

Although we motivated the study of $\Rep(\bO)$ as a category analogous to $\cV$ which might harbor analogues of tca's, it turns out that $\Rep(\bO)$ can be described in terms of tca's!  Precisely, $\Rep(\bO)$ is equivalent to the category of finite length modules over the tca $\Sym(\bC\langle 2\rangle)$ (though the tensor product is not the usual one).  In fact, it is just a matter of unraveling definitions to see that this module category is equivalent to $\Vec^{\ub}$.  This result is interesting for at least two reasons:  (1) it gives a description of representations of $\bO(\infty)$ in terms of polynomial representations of $\GL(\infty)$; and (2) to our knowledge, it is the first occurrence of an unbounded tca ``in nature.''

All of this is just the beginning of the story.  The most important results we establish give categorical interpretations to earlier results of Koike and Terada on universal character rings \cite{koiketerada}; for instance, we construct and study a specialization functor $\Rep(\bO) \to \Rep(\bO(d))$ which categorifies the Koike--Terada specialization homomorphisms. These specialization functors are left exact (using the ``pro-finite'' model, they would be right exact), and their derived functors applied to simple objects are nonzero in at most one homological degree. This is the main result of \cite{ssw}, which also explicit calculates these derived functors. This gives a refined version of the ``modification rules'' of \cite{koiketerada}. Furthermore, these modification rules can be phrased in terms of a Weyl group dot action, so one can view these results as an analogue of Bott's theorem \cite[Chapter 4]{weyman} for the cohomology of homogeneous bundles on projective homogeneous spaces.

\addtocontents{toc}{\vskip.6\baselineskip}

\end{document}